\documentclass[11pt]{article} 
\usepackage{latexsym, amsfonts, amssymb, amsmath, amsthm, mathtools}
\usepackage{graphicx}
\usepackage{multirow}
\usepackage{multicol}
\graphicspath{ {images/} }
\usepackage[dvipsnames]{xcolor}
\usepackage[hidelinks,
            pdftitle={Existence, and uniqueness of monotone traveling wave solutions of parabolic-elliptic chenotaxis models with logistic-type source},
            pdfauthor={Wenxian Shen},
            pdfkeywords={existence, uniqueneess, stability, monotonicity, traveling waves parabolic--elliptic chemotaxis models, logistic-type source},
            pdfpagemode=FullScreen,
            breaklinks=true,
           ]{hyperref}
\usepackage[left=1in, right=1in, top=1in, bottom=1in]{geometry}
\usepackage{colortbl}
\usepackage{ulem}
\usepackage{caption}

\usepackage{pgfplots}
\pgfplotsset{compat=1.18}
\usepackage{tikz}
\usetikzlibrary{shapes.geometric,positioning,arrows,intersections,patterns}

\newif\ifrefaudit
\refauditfalse
\ifrefaudit
  \usepackage[nocites]{refcheck}
\fi

\definecolor{TableRowGray}{gray}{0.9}

\usepackage{amsrefs}

\BibSpec{article}{%
    +{}  {\PrintAuthors}                {author}
    +{,} { \textit}                     {title}
    +{.} { }                            {part}
    +{:} { \textit}                     {subtitle}
    +{,} { \PrintContributions}         {contribution}
    +{.} { \PrintPartials}              {partial}
    +{,} { }                            {journal}
    +{}  { \textbf}                     {volume}
    +{}  { \PrintDatePV}                {date}
    +{,} { \issuetext}                  {number}
    +{,} { \eprintpages}                {pages}
    +{,} { }                            {status}
    +{,} { \url}                        {url}
    +{,} { \PrintDOI}                   {doi}
    +{,} { available at \eprint}        {eprint}
    +{}  { \parenthesize}               {language}
    +{}  { \PrintTranslation}           {translation}
    +{;} { \PrintReprint}               {reprint}
    +{.} { }                            {note}
    +{.} {}                             {transition}
    +{}  {\SentenceSpace \PrintReviews} {review}
}

\newtheorem{theorem}{Theorem}[section] 

\newtheorem{lemma}{Lemma}[section]

\newtheorem{proposition}{Proposition}[section]

\theoremstyle{remark}
\newtheorem{remark}{Remark}[section]
\theoremstyle{remark}

\numberwithin{equation}{section}
\numberwithin{figure}{section}
\numberwithin{table}{section}

\def\R{\mathbb R}

\newcommand\RR{\ensuremath{\mathbb{R}}}

\begin{document}

\title{Existence,  uniqueness, stability, and  monotonicity of  traveling waves for repulsion/attraction  chemotaxis models with  logistic type  source
}
\author{
  Wenxian Shen \\
Department of Mathematics and Statistics\\
Auburn University, Auburn, AL 36849, USA}

\date{}
\maketitle

\begin{abstract}
This paper is devoted to the study of existence,  uniqueness, stability, and monotonicity of traveling wave solutions 
to the following parabolic-elliptic chemotaxis system with logistic type source
$$
\begin{cases}
u_t=u_{xx}-\chi(u^m v_x)_x +u(1-u^\alpha),\quad &x\in\RR\cr
0=v_{xx}-v+u^\gamma,\quad&x\in\RR,
\end{cases}\eqno(CM)
$$
where $m,\alpha,\gamma\ge 1$ and $\chi\in\RR$. System (CM) can be used to describe   the evolution of a biological species influenced by a chemical substance produced by the species itself. 
In this context,
 the function $u$ denotes the population density of
the biological species, and $v$ denotes the concentration of the chemical agent. Traveling wave solutions of (CM)  connecting the two constant solutions $(1,1)$ and $(0,0)$ are among important types of solutions, which characterize  the front propagation phenomena in (CM). The existence of such traveling wave solutions to (CM) with $m=\alpha=\gamma=1$ has been studied in several papers. However,
there is little study on the uniqueness,  stability,  and monotonicity of traveling wave solutions of (CM) in literature and there is also no study on the existence of traveling wave solutions of (CM) connecting $(1,1)$ and $(0,0)$ for general $m,\alpha,\gamma\ge 1$.
In the current paper,  we  prove the existence of  traveling wave solutions of (CM)
connecting $(1,1)$ and $(0,0)$ for any $\chi\le 0$  with speed $c$ large than some number $c^*_{\chi,m,\gamma}$,  or for $0<\chi<1/2$  with any speed $c>2$. We prove that the traveling wave solutions are monotone when $\chi\le 0$.  We also prove the uniqueness and stability  of traveling wave solutions of (CM)  connecting $(1,1)$ and $(0,0)$ when the speed $c$ is larger than some number $c^{**}_{\chi,m,\alpha, \gamma}(\ge c^*_{\chi,m,\gamma})$.

Note that when $m=\gamma=\alpha=1$, it is proved in \cite{GrHeTu2023} that negative chemotaxis  speeds up 
front propagation and when $\chi\ll -1$, the speed $c\ge O(\sqrt{|\chi|})$.
The lower bound $c^*_{\chi,m,\gamma}$ for wave speeds of traveling wave solutions of (CM) with $\chi\le 0$ is not expected to be optimal, but it is close to optimal when $|\chi|$ is small and large in the sense that
when $\chi=0$ and $\gamma=1$, $c^*_{\chi,m,\gamma}=2$,
which is the optimal lower bound for the wave speeds of (CM) with $\chi=0$, and when $\chi\to -\infty$, 
$c^*_{\chi,m,\gamma}=O(\sqrt{|\chi|})$. 
Also,  $c^{**}_{\chi,m, \alpha,\gamma}$ is not expected to be optimal,   but $c^{**}_{\chi,m,\alpha,\gamma}$ is close to optimal when $|\chi|$ is small  and $\gamma=1$ in the sense
that $c^{**}_{\chi,m,\alpha,\gamma}=2+O(|\chi|^{1/6})$ when $|\chi|\to 0$.
\end{abstract}

\noindent \textbf{Keywords}.  Existence, uniqueness,  stability,  monotonicity, traveling waves, parabolic-elliptic chemotaxis model,  logistic type  source.

\medskip

\noindent \textbf{AMS Subject Classification (2020)}: 35B35, 35C07, 35K45, 35Q92, 92C17, 92D25.

\newpage

\tableofcontents

\newpage

\section{Introduction}

\subsection{Overview} 

The  current paper is concerned with the existence,  stability, and monotonicity of  traveling wave solutions 
to the following parabolic-elliptic chemotaxis system with logistic type source,
\begin{equation}
\label{main-eq}
\begin{cases}
u_t=u_{xx}-\chi(u^m v_x)_x +u(1-u^\alpha),\quad &x\in\RR\cr
0=v_{xx}-v+u^\gamma,\quad&x\in\RR,
\end{cases}
\end{equation}
where $m,\alpha,\gamma\ge 1$ and $\chi\in\RR$. System \eqref{main-eq} can be used to describe   the evolution of a biological species influenced by a chemical substance which is  produced by the species itself
and diffuses very fast. 
In this context,
 the function $u$ denotes the population density of
the biological species, and $v$ denotes the concentration of the chemical agent. The constant $\chi$ 
is referred to as chemotaxis sensitivity coefficient.  
The chemotaxis sensitivity coefficient $\chi >0$ indicates that the signal is a chemoattractant,  in this case, \eqref{main-eq} is called an attraction chemotaxis system,
while $\chi<0$ indicates that the signal is  a chemorepellent, in this case, \eqref{main-eq} is called a repulsion chemotaxis system.

Considering chemotaxis models, it is interesting to study how the chemical signal affects the evolution of the biological species, say, whether  the chemical signal  induces finite time blow-up, whether it creates certain pattern formation, etc. Numerous works have been carried out toward these central problems for various
chemotaxis models on both bounded  and unbounded domains.  Chemotaxis models on unbounded domains can 
describe the influence of chemical signals from more angles, say, from the angle of spatial spreading, which is strongly related to traveling wave solutions connecting two different states.

It is clear that \eqref{main-eq} has two constant states $(1,1)$ and $(0,0)$. {\it Positive traveling wave solutions of \eqref{main-eq} connecting these two states  with speed $c$}  are  solutions $(u(t,x),v(t,x))$
of the form $(u(t,x),v(t,x))=(\phi(x-ct),\psi(x-ct))$ for some one variable positive  functions $\phi(\cdot)$, $\psi(\cdot)$ satisfying $\lim_{x\to -\infty}(\phi(x),\psi(x))=(1,1)$ and $\lim_{x\to\infty}(\phi(x),\psi(x))=(0,0)$.
Such solutions  describe how fast the biological species spreads from higher density regions to lower density regions. Throughout this paper, we only consider positive traveling wave solutions, which we may just call traveling wave solutions.

When $\chi=0$, the first equation in \eqref{main-eq} is decoupled from the second one and becomes
\begin{equation}
\label{fisher-kpp-eq}
u_t=u_{xx}+u(1-u^\alpha),\quad x\in\RR.
\end{equation}
Equation \eqref{fisher-kpp-eq} is referred to as  a Fisher-KPP type equation in literature due to the pioneering work 
\cite{Fisher}, \cite{KPP} for the case $\alpha=1$. It is well known that \eqref{fisher-kpp-eq}  has positive  traveling wave solutions $u(t,x)=\phi(x-ct)$ connecting $1$ and $0$ $(i.e., \phi(-\infty)=1,\phi(\infty)=0)$ of all speeds $c\geq c_0^*:=2$ and
has no such traveling wave solutions of slower
 speed. The traveling wave solutions of \eqref{fisher-kpp-eq} connecting $1$ and $0$ are also monotone and stable in proper weighted spaces.    Moreover, 
 for any
nonnegative solution $u(t,x)$ of \eqref{fisher-kpp-eq}, if  the support ${\rm supp}(u(0,\cdot))$ of
$u(0,x)$  is nonempty and compact, then
\begin{equation}
\label{spreading-eq1}
\liminf _{|x|\le c''t, t\to \infty} u(t,x) >0\quad \forall\,  0<c''<c_0^*,
\end{equation}
and
\begin{equation}
\label{spreading-eq2}
\limsup_{|x| \ge c't, t\to \infty}u(t,x)=0 \quad \forall \, c'>c_0^*.
\end{equation}
Hence, in
literature,  the minimal wave speed $c^*_0=2$ is   also called the {\it
spreading speed} for \eqref{fisher-kpp-eq}.  Since the pioneering works by \cite{Fisher} and \cite{KPP},  a huge amount  of research has been carried out toward the front propagation dynamics of
  reaction diffusion equations of the form,
\begin{equation}
\label{general-fisher-eq}
u_t=\Delta u+u f(t,x,u),\quad x\in\R^N,
\end{equation}
where $f(t,x,u)<0$ for $u\gg 1$,  $\partial_u f(t,x,u)<0$ for $u\ge 0$ (see \cite{ArWe2,BeHaNa1,BeHaNa2,Henri1,Fre,FrGa,LiZh,LiZh1,Nad,NoRuXi,NoXi1,She1,She2,Wei1,Wei2,Zla}, etc.).

\smallskip

Consider \eqref{main-eq} with $\chi\not =0$.  It is important to study whether  \eqref{main-eq} has traveling wave solutions connecting $(1,1)$ and $(0,0)$, and if so, whether the traveling wave solutions are monotone and stable in proper weighted spaces; whether \eqref{main-eq} possesses a spreading speed;  whether the chemotaxis speeds up or slows down the wave speeds or the spatial spreading of the biological species, etc. 
There are several studies on the existence of  traveling wave solutions  and spreading speeds of \eqref{main-eq} with $m=\alpha=\gamma=1$,  see, for example, \cite{AvHoSc2025, GrHeTu2023, HaHe2020, Hen2022,  HeRe2024, NaPeRy2008, Sal2019, SaSh2017, SaSh2017-1,
SaSh2018, SaSh2019, SaSh2020, SaShXu2019}.

 Suppose that $m=\alpha=\gamma=1$.   Among those, the following have been proved/observed  in literature:

\begin{itemize}

\item  ({\it Spatial spreading speeds}).  The spreading property \eqref{spreading-eq1}   holds for nonnegative solutions $(u(t,x)$, $v(t,x))$ of \eqref{main-eq} with ${\rm supp} (u(0,\cdot))$ being nonempty 
(see \cite[Theorem 1.1]{HaHe2020},  \cite[Theorem 1.1]{SaShXu2019}).  This indicates that chemotaxis does not slow down the spatial spreading of the biological species.  When $0<\chi\le 1$, the spreading property \eqref{spreading-eq2} holds for nonnegative solutions $(u(t,x)$, $v(t,x))$ of \eqref{main-eq} with ${\rm supp} (u(0,\cdot))$ being nonempty and compact (see \cite[Theorem 1.1]{SaShXu2019}). This indicates that when $0<\chi\le 1$,  chemotaxis neither slows down nor speeds up the spatial spreading of the biological species. In general,  \eqref{spreading-eq2} holds 
 with $c_0^*$ being replaced by some $c_+^*\ge 2$ (see \cite[Theorem 1.1]{HaHe2020}, \cite[Theorem B]{SaSh2017-1},  \cite[Theorem C]{SaSh2020}).

\item   ({\it Lower bound of wave speeds or minimal wave speed}). 
{$c_0^*=2$ is a lower bound of the wave speeds in the sense that  if $(u(t,x),v(t,x))=(\phi(x-ct),\psi(x-ct))$ is a traveling wave solution of \eqref{main-eq} connecting $(1,1)$ and $(0,0)$ with speed $c$, then $c\ge 2$}  (it follows from \cite[Theorem 1.2]{HaHe2020}, \cite[Proposition 1.3]{Hen2022},  \cite[Theorem 1.2]{HeRe2024},  \cite[Theorem C]{SaSh2020}). This indicates that no matter the chemotaxis sensitivity is positive or negative, it does not slow down the wave speeds.
Moreover, {if $|\chi|$ is relatively small, then $c_0^*=2$ is the minimal wave speed of traveling wave solutions, or, there are  traveling wave solutions  of \eqref{main-eq} connecting $(1,1)$ and $(0,0)$ with speed $c= 2$} 
(it follows from \cite[Theorem 1.4]{Hen2022}, \cite[Theorem 1.4]{SaShXu2019}). This indicates that the chemotaxis sensitivity does not speed up the spatial spreading of the biological species when $|\chi|$ is small.

\item  ({\it Negative chemotaxis sensitivity and speed-up}). 
{For any $\chi<0$, there is $c_{\chi}^*\ge 2$ such that for any $c>c_{\chi}^*$, \eqref{main-eq} has a traveling wave solution connecting $(1,1)$ and $(0,0)$ with speed $c$}
(it follows from \cite[Theorem A]{SaSh2020}).  
When $\chi\ll -1$, any speed $c$ of a traveling wave solution of \eqref{main-eq} connecting
$(1,1)$ and $(0,0)$ is very large  and hence strong negative sensitivity speeds up the wave speeds
(see \cite[Theorem 2.1]{GrHeTu2023}).

\item ({\it Positive chemotaxis sensitivity and oscillatory  traveling waves}).
When $0<\chi<1/2$, for any $c\ge 2$, \eqref{main-eq} admits a traveling wave solution connecting
$(1,1)$ and $(0,0)$ with speed $c$ (see \cite[Theorem 1.4]{SaShXu2019}). 
 For general $\chi>0$, there  exist a $c\ge 2$ and  a traveling wave solution $(u(t,x),v(t,x))=(\phi(x-ct),\psi(x-ct))$ with
$\liminf_{x\to \infty}\phi(x)>0$ and $\lim_{x\to\infty}\phi(x)=0$ (it follows from \cite[Theorem 1.2]{HeRe2024}). Note that this traveling wave solution may not converge to $(1,1)$ and may be oscillatory as $x\to -\infty$ (see the numerical simulations in \cite{HeRe2024}). Moreover,
It is observed numerically that 
 the minimal speed \eqref{main-eq} appears to always be $2$  for $\chi >0$ (see \cite{HeRe2024}).

\end{itemize}

The above results from literature provide deep  insights into  the understanding of front propagation phenomena in \eqref{main-eq}. 
Though there exist many studies on traveling wave solutions of \eqref{main-eq}, there  is still  little study on  numerous interesting problems about traveling wave solutions of \eqref{main-eq}, for example,  there is little study about the stability and monotonicity of traveling wave solutions of \eqref{main-eq} in literature, even when $m=\alpha=\gamma=1$; there is no study on the existence of traveling wave solutions of \eqref{main-eq} for general $m,\alpha,\gamma\ge 1$.

The objective of the current paper is to study the existence,  uniqueness, stability, and monotonicity of traveling wave solutions \eqref{main-eq} connecting
$(1,1)$ and $(0,0)$ for general $m,\alpha,\gamma\ge 1$, including $m=\alpha=\gamma=1$ as a special case.  To this end, we first  study the Cauchy problem for general $m,\alpha,\gamma$  and the stability of
the positive constant solution $(1,1)$ since they have not been studied for general $m,\alpha,\gamma\ge 1$ and play  important roles in the  study of traveling waves.  Before we state our main results in next subsection,  let us  highlight some of the contributions of this paper in the following:

\begin{itemize}

\item ({\it Global well-posedness of the Cauchy problem})  This paper establishes  explicit parameter regions for the global existence and boundedness of positive classical solutions  (see Proposition \ref{global-existence-prop}),
which provides  impacts of the parameters on the global existence and
extends \cite[Theorem 1.4]{HaShZh2024}  for the case $m=\alpha=\gamma=1$ to the general case $m,\alpha,\gamma\ge 1$.

\item ({\it Stability of the positive constant solution}) This paper  establishes explicit  parameter regions for the global stability of 
the positive constant solution $(1,1)$ with respect to strictly positive perturbations (see Proposition \ref{constant-stability-prop}), which   provides   impacts of the parameters on the stability of the positive constant solution and extends \cite[Theorem B]{SaSh2019} for the case $m=\alpha=\gamma=1$ to the general case $m,\alpha,\gamma\ge 1$.

\item ({\it Existence of monotone traveling wave solutions with negative sensitivity})
This paper finds an explicit lower bound $c^{*}_{\chi,m,\gamma}$ for the wave speeds and shows that
for any $\chi\le 0$ and $c>c^{*}_{\chi,m,\gamma}$, \eqref{main-eq} with $\alpha\le m+\gamma-1$  has a monotone traveling wave solution
connecting $(1,1)$ and $(0,0)$ with speed $c$ and explicit decay rate (see Theorem \ref{monotone-thm}(1)), which not only extends the existence result in  \cite[Theorem A]{SaSh2020}  for the case $m=\alpha=\gamma=1$ to the general case $m,\alpha,\gamma\ge 1$, but also shows the monotonicity and explicit decay rate
of traveling wave solutions for the first time. Moreover, the lower bound $c^*_{\chi,m,\gamma}$ is close to be optimal in certain sense (see Remark \ref{rk2}(2)).

\item ({\it Existence of traveling wave solutions with positivity sensitivity})  This paper proves   the existence of traveling wave solutions of \eqref{main-eq} connecting $(1,1)$ and $(0,0)$ with  explicit exponentially decay rate for any speed  $c>2$ 
when $0\le \chi<\min\{\frac{1}{2},\frac{2\gamma+2}{2\gamma+m+1}\}$ and $\alpha=m+\gamma-1$,
 which   extends the existence results in \cite[Theorem 1.2]{SaShXu2019} for the case $m=\alpha=\gamma=1$ to the more general case $\alpha=m+\gamma-1$ (see
Theorem \ref{monotone-thm}(2)) and also recovers the existing  results for the Fisher-KPP equation.

\item ({\it Uniqueness and stability of traveling wave solutions)} This paper finds a positive $c^{**}_{\chi,m,\alpha,\gamma}$, which depends on $\chi, m,\alpha, \gamma$ explicitly and goes to $2$ as $\chi\to 0$, and proves that  for any $c>c^*_{\chi,m,\alpha,\gamma}$,   traveling wave solutions of \eqref{main-eq} connecting $(1,1)$ and $(0,0)$ with speed $c$ (if exist) are stable in some weighted spaces (see Theorem \ref{stability-thm}), which is new even for the case that $m=\alpha=\gamma=1$. Uniqueness of traveling wave solutions of \eqref{main-eq} connecting $(1,1)$ and $(0,0)$ with speed $c>c^{**}_{\chi,m,\alpha,\gamma}$ is also obtained in this paper.

\item ({\it  New techniques/ideas and extension of existing methods}) This paper develops several techniques and extends  some existing methods for the study of Cauchy problem and existence, uniqueness,  and stability of traveling wave solutions of \eqref{main-eq} (see   Remark \ref{rk1}(1),(2), Remark \ref{rk2} (4), Remark \ref{rk3} (2), and Remark \ref{rk4}(2)). These techniques and methods can also be applied to some other chemotaxis models.
\end{itemize}

It should be pointed out that  there are a lot of works  on the  studies of  traveling
wave solutions of various  types of chemotaxis models, see, for example 
\cite{AiHuWa,AiWa, ChCh, ChZo,FuMiTs,HoSt, LiLiXu,LiLiWa,MaNoSh,NaPeRy2008,Wan}, etc.

\subsection{Statements of the main results}

In this subsection, we state our main results on the existence,  stability, and monotonicity of traveling wave solutions of \eqref{main-eq} connecting $(1,1)$ and $(0,0)$. 

To this end, we first discuss the Cauchy problem of \eqref{main-eq}.
By the arguments of \cite[Theorem 1.1]{SaSh2017}, for any $u_0\in C_{\rm unif}^b(\RR)$, there is
$T_{\max}(u_0)\in (0,\infty]$ such that \eqref{main-eq} has a unique classical solution
$(u(t,x;u_0), v(t,x;u_0))$ with $u(0,x;u_0)=u_0(x)$ for $t\in (0,T_{\max}(u_0))$, and if $T_{\max}(u_0)<\infty$, then
$$
\limsup_{t\to T_{\max}(u_0)-} \|u(t,\cdot;u_0)\|_\infty=\infty.
$$
When $m=\alpha=\gamma=1$, it is proved in \cite{HaShZh2024}  (see \cite[Theorem 1.4]{HaShZh2024}) that for any $u_0\in C_{\rm unif}^b(\RR)$ with $u_0\ge 0$,  it holds that 
\begin{equation}
\label{global-existence-eq0}
T_{\max}(u_0)=\infty\quad {\rm and}\quad 
\limsup_{t\to\infty} \|u(t,\cdot;u_0)\|_\infty<\infty.
\end{equation}
Moreover,  it is proved in  \cite{SaSh2019} (see \cite[Theorem B]{SaSh2019})  that, if $\chi<\frac{1}{2}$, then for any $u_0\in C_{\rm unif}^b(\RR)$ with $\inf_{x\in\RR}u_0(x)>0$, it holds that
\begin{equation}
\label{stability-eq0}
\lim_{t\to\infty} \Big(\|u(t,\cdot;u_0)-1\|_\infty+\|v(t,\cdot;u_0)-1\|_\infty\Big)=0.
\end{equation}

For general $m,\alpha,\gamma\ge 1$,  the work \cite{GaSaTe}  studied the global existence of positive solutions and stability of the positive constant solution of the counterpart of \eqref{main-eq} on bounded domain with Neumann boundary condition. 
However, there is little study on the global existence, boundedness, and stabilization of positive classical solutions of \eqref{main-eq}, which play important roles in the study of existence and stability of traveling wave solutions connecting $(1,1)$ and $(0,0)$.  We therefore give a study on the Cauchy problem  of \eqref{main-eq} in this paper.

Let
$$
C_{\rm unif}^b(\RR)=\{u\in C(\RR,\RR)\,|\, u\,\,\, \text{is uniformly continuous and bounded on}\,\,\RR\}
$$
with norm $\|u\|=\sup_{x\in\RR}|u(x)|$.  We have
the following two propositions, which will be proved in Section \ref{cauchy-problem-section}.

\begin{proposition}[Global existence and boundedness of classical solutions]
\label{global-existence-prop}
Assume that $m,\alpha,\gamma\ge 1$.
\begin{itemize}
\item[(1)] If $\chi\le 0$, then for any $u_0\in C_{\rm unif}^b(\RR)$ with $u_0\ge 0$,
the classical solution $(u(t,x;u_0),v(t,x;u_0))$ of \eqref{main-eq} with $u(0,x;u_0)=u_0(x)$ exists globally
(i.e. $T_{\max}(u_0)=\infty$),
\begin{equation}
\label{upper-bound-eq1}
u(t,x;u_0)\le \max\{1,\sup_{x\in\RR}u_0(x)\}\quad \forall\, t\ge 0,\,\, x\in\RR,
\end{equation}
and
\begin{equation}
\label{upper-limit-eq1}
\limsup_{t\to\infty} \sup_{x\in\mathbb{R}}u(t,x;u_0)\le 1.
\end{equation}

\item[(2)]  If  $\chi> 0$ and $\alpha>m+\gamma-1$, or $0<\chi<\min\{ \frac{2m-1}{m-1}, \frac{m+\gamma -1}{\gamma-1}\}$ and $\alpha= m+\gamma -1$,  then  for any $u_0\in C_{\rm unif}^b(\RR)$ with $u_0\ge 0$,
the classical solution $(u(t,x;u_0),v(t,x;u_0))$ of \eqref{main-eq} with $u(0,x;u_0)=u_0(x)$ exists globally
(i.e. $T_{\max}(u_0)=\infty$) and
\begin{equation}
\label{upper-bound-eq2}
\limsup_{t\to\infty}\|u(t,\cdot;u_0)\|_\infty<\infty.
\end{equation}
If, in addition,  $0<\chi<1$, then 
\begin{equation}
\label{upper-limit-eq2}
\limsup_{t\to\infty} \sup_{x\in\mathbb{R}}u(t,x;u_0)\le \Big(\frac{1}{1-\chi}\Big)^{\frac{1}{\alpha}}.
\end{equation}
\end{itemize}
\end{proposition}

\begin{proposition}
[Stability of the positive constant solution]
\label{constant-stability-prop}
Assume that $m,\alpha,\gamma\ge 1$.
\begin{itemize}
\item[(1)] If $\chi\le  0$, then  for any $u_0\in C_{\rm unif}^b(\RR)$ with $u_0(x)\ge 0$
and  $\inf_{x\in\RR} u_0(x)>0$, there holds
\begin{equation}
\label{constant-stability-eq}
\lim_{t\to\infty} \|u(t,\cdot;u_0)-1\|_\infty =0.
\end{equation}

\item[(2]  If  
If  $0<\chi<\frac{1}{2}$  and  $\alpha\ge m+\gamma-1$, then for any $u_0\in C_{\rm unif}^b(\RR)$ with $u_0(x)\ge 0$ and   $\inf_{x\in\RR} u_0(x)>0$,  \eqref{constant-stability-eq} holds.
\end{itemize}
\end{proposition}

Next, we state our main results on the existence of traveling  wave solutions of \eqref{main-eq}
connecting $(1,1)$ and $(0,0)$. 

\begin{theorem}[Existence of  traveling waves]
\label{monotone-thm}
$\quad$

\begin{itemize}
\item[(1)] {\rm (Existence of monotone traveling  wave solutions with negative sensitivity)}

 Assume that  $\alpha\le m+\gamma -1$ and $\chi\le 0$. For any   $c>0$ satisfying
\begin{align}
\label{cond-on-c-eq1}
c>c^*_{\chi,m,\gamma}:=\max\Big\{ \frac{1}{m}+m, \frac{1}{\sqrt{m\gamma |\chi|+\gamma^2 |\chi|+\gamma^2}} + \sqrt {m\gamma |\chi| +\gamma^2 |\chi|+\gamma^2}\Big\},
\end{align}
 there is a traveling wave solution
$(u(t,x),v(t,x))=(U^*(x-ct),V^*(x-ct))$ of \eqref{main-eq}  satisfying that
\begin{equation}
\label{wave-eq1}
0<U^*(x)<\max\{1, e^{-\kappa x}\},\quad 
U^*_x(x)\le 0,\quad V^*_x(x)\le 0\quad \forall\, x\in\RR,
\end{equation}
where $\kappa =\frac{c-\sqrt{c^2-4}}{2}$,
and
\begin{equation}
\label{wave-eq2}
\lim_{x\to -\infty}(U^*(x),V^*(x))=(1,1),\quad \lim_{x\to\infty} (U^*(x), V^*(x))=(0,0).
\end{equation}
Moreover, for any $\kappa_1$ satisfying $\kappa<\kappa_1<  \min\{(1+\alpha)\kappa,  m\kappa +1/2,  1\}$, 
 it holds that
\begin{equation}
\label{wave-eq3}
\lim_{x\to\infty} e^{(\kappa_1-\kappa)x} \Big(\frac{U^*(x)}{e^{-\kappa x}}-1\Big)=0.
\end{equation}

\item[(2)] {\rm (Existence of traveling wave solutions with positive sensitivity)}

Assume  that  $\alpha=  m+\gamma -1$ and $0\le \chi<\min\{\frac{1}{2},\chi^*\}$, where
\begin{equation}
\label{chi-star-eq}
\chi^*=\chi^*(m,\alpha,\gamma):=\min\Big\{1, \frac{2m+2\gamma}{m^2+m+2\gamma}\Big\}.
\end{equation}
 For any  $c>2$  there is a traveling wave solution 
$(u(t,x),v(t,x))=(U^*(x-ct),V^*(x-ct))$ of \eqref{main-eq} satisfying 
 \eqref{wave-eq2},  \eqref{wave-eq3}, and
\end{itemize}
\begin{equation}
\label{wave-eq4}
0<U^*(x)<\min\Big\{\big(\frac{1}{1-\chi}\big)^{1/\alpha}, e^{-\kappa x}\Big\}\quad \forall\, x\in\RR.
\end{equation}
\end{theorem}

Theorem \ref{monotone-thm} will be proved in Section \ref{existence-section}.
The following theorem is about the stability of traveling wave solutions and will be proved in Section \ref{uniqueness-stability-section}.

\begin{theorem}[Stability of traveling waves connecting constant solutions]
\label{stability-thm}
$\quad$
Assume  that $\chi< 0$ and $\alpha\le m+\gamma-1$, or $0\le \chi<\chi^*$ and $\alpha= m+\gamma-1$.
Let $M_\chi=1$ if $\chi< 0$ and $M_\chi=\big(\frac{1}{1-\chi}\big)^{1/\alpha}$ if $0\le \chi<\chi^*$.
There is $c^{**}_{\chi,m,\alpha,\gamma}>1+|\chi|^{1/6}+\frac{1}{1+|\chi|^{1/6}}$ with $c^{**}_{\chi,m,\alpha,\gamma}-\big(\gamma+\frac{1}{\gamma}\big)=O(|\chi|^{1/6})$ as $\chi\to 0$ such that  the following hold:
 If $c>c_{\chi,m,\alpha,\gamma}^{**}$ and  
$(u(t,x),v(t,x))=(U^*(x-ct),V^*(x-ct))$ is a traveling wave solution of \eqref{main-eq} connecting $(1,1)$ and $(0,0)$ with speed $c$ and satisfying
\begin{equation}
\label{decay-rate-eq1}
0<U^*(x)<\min\{M_{\chi},e^{-\kappa x}\},\,\, {\rm and}\,\, \lim_{x\to\infty} e^{(\kappa_1-\kappa)x} \Big(\frac{U^*(x)}{e^{-\kappa x}}-1\Big)=0
\end{equation}
for some $\kappa_1\in (\kappa,1)$,  where $\kappa=\frac{c-\sqrt{c^2-4}}{2}$,
then 
 for any $\eta\in (\kappa,\frac{1}{1+|\chi|^{1/6}})$  and    any $u_0\in C_{\rm unif}^b(\RR)$ satisfying that
 \begin{equation}
\label{stability-u0-eq1}
u_0\ge 0,\quad
\liminf_{x\to -\infty}u_0(x)>0,\quad {\rm and}\quad \int_{\RR} e^{2\eta x}|u_0(x)-U^*(x)|^2<\infty,
\end{equation} 
 there hold
\begin{equation}
\label{main-stability-eq1}
\lim_{t\to\infty}\int_{\RR} e^{2\eta x}|u(t,x;u_0)-U^*(x-ct)|^2dx=0,
\end{equation}
and 
\begin{equation}
\label{main-stability-eq2}
\lim_{t\to\infty} \sup_{x\in\RR} |u(t,x;u_0)-U^*(x-ct)|=0.
\end{equation}
\end{theorem}

\begin{theorem}[Uniqueness of traveling waves connecting constant solutions]
\label{uniqueness-thm}
$\quad$
Assume  that $\chi< 0$ and $\alpha\le m+\gamma-1$, or $0\le \chi<\chi^*$ and $\alpha= m+\gamma-1$.
Let $M_\chi=1$ if $\chi< 0$ and $M_\chi=\big(\frac{1}{1-\chi}\big)^{1/\alpha}$ if $0\le \chi<\chi^*$. Suppose that  $c>c^{**}_{\chi,m,\alpha,\gamma}$ and 
$(u_i(t,x),v_i(t,x))=(U_i^*(x-ct),V_i^*(y-ct))$ $(i=1,2)$ are two  traveling wave solution of \eqref{main-eq} connecting $(1,1)$ and $(0,0)$ with speed $c$ and satisfying
\begin{equation}
\label{decay-rate-eq2}
0<U_i^*(x)<\min\{M_{\chi},e^{-\kappa x}\},\,\, {\rm and}\,\, \lim_{x\to\infty} e^{(\kappa_1-\kappa)x} \Big(\frac{U_i^*(x)}{e^{-\kappa x}}-1\Big)=0
\end{equation}
for $i=1,2$ and some $\kappa_1\in (\kappa,1)$, then
$$
(U_1^*(x),V_1^*(x))\equiv (U_2^*(x),V_2^*(x)).
$$
\end{theorem}

Theorem \ref{uniqueness-thm} will be proved in Section \ref{uniqueness-stability-section}. 

\subsection{Remarks}

In this subsection, we give some remarks/discussions on our main results and the techniques established in the proofs of the results.

\begin{remark}
\label{rk1}
{\it This remark is about the global existence, boundedness, and stabilization  of positive classical solutions of \eqref{main-eq}.}

\begin{itemize}

\item[(1)] ({\it Global existence and boundedness}) 
The results stated in Proposition \ref{global-existence-prop} are new for general $m,\alpha,\gamma\ge 1$.
When $m=\alpha=\gamma=1$, part (1) of this proposition is proved in  \cite[Theorem A]{SaSh2019} and part (2) is proved in \cite[Theorem 1.4]{HaShZh2024}. 
We extend the arguments in  \cite[Theorrm 1.4]{HaShZh2024} and \cite[Theorem A]{SaSh2019} for the special case $m=\alpha=\gamma=1$ 
with $\chi>0$ and $\chi<0$ to prove Proposition \ref{global-existence-prop}  for the case $\chi>0$ and $\chi<0$, respectively. Such extension can also be applied for \eqref{main-eq} in higher dimensional space as well as some other chemotaxis models.

\item[(2)] ({\it Stabilization}) The results stated in Proposition \ref{constant-stability-prop} are also new for general $m,\alpha,\gamma\ge 1$, and are utilized in the proof of $\lim_{x\to -\infty}(u(t,x),v(t,x))=(1,1)$ for the traveling wave solutions in Theorem \ref{monotone-thm} and  in the proof of stability of traveling wave solutions. When $m=\alpha=\gamma=1$,  Proposition \ref{constant-stability-prop} is proved in \cite[Theorem B]{SaSh2019}. But it is not easy to extend the arguments of  \cite[Theorem B]{SaSh2019} for the case $m=\alpha=\gamma=1$ to the general case $m,\alpha,\gamma\ge 1$.
Our idea to prove Proposition \ref{constant-stability-prop} in the case $\chi<0$  can be describe as follows: let
$\bar u(t)=\sup_{x\in\RR} u(t,x;u_0)$ and $\underline u(t)=\inf_{x\in\RR}u(t,x;u_0)$ and prove that if $\bar u(t_0)>1$ (resp. $0<\underline u(t_0)<1$) for some $t_0>0$, then
$\bar u(t)$ (resp. $\underline u(t)$) is decreasing (resp. increasing) on $[0,t_0)$, which implies that
$
\lim_{t\to\infty}
\|u(t,\cdot;u_0)-1\|_\infty=0.
$
This idea is new and can be applied to  \eqref{main-eq} in higher dimensional space as well as some other chemotaxis models.  We adopt the following rectangle/ODE approach to prove Proposition \ref{constant-stability-prop} in the case $0<\chi<1/2$:  Consider the following system of ODEs,
$$
\begin{cases}
\overline U_t=\chi \overline U^m (\overline U^\gamma -\underline U^\gamma)+ \overline U(1-\overline U^\alpha)\cr
\underline U_t=\chi \underline U^m (\underline U^\gamma-\overline U^\gamma)+\underline U(1-\underline U^\alpha)\cr
\overline U(0)=\max\{\|u_0\|_\infty, 1\}+\epsilon,\quad \underline U(0)=\min\{\inf_{x\in\RR}u_0(x),1\}-\epsilon
\end{cases}
$$
for some $0<\epsilon\ll 1$, and prove that
\begin{equation*}
0<\underline U(t)\le u(t,x;u_0)\le \overline U(t)\quad \forall\, t\in [0,\infty),\,\, x\in\RR,
\end{equation*}
and
\begin{equation*}
\lim_{t\to\infty} \overline U(t)=\lim_{t\to\infty}\underline U(t)=1
\end{equation*}
which implies that $
\lim_{t\to\infty}
\|u(t,\cdot;u_0)-1\|_\infty=0.
$
This approach is an extension of the rectangle /ODE approach developed in \cite[Theorem 1.2]{GaSaTe} for  chemotaxis models of the form  \eqref{main-eq} on bounded domains to
bounded domains, and  can also be applied for \eqref{main-eq} in higher dimensional space as well as some other chemotaxis models.
\end{itemize}

\end{remark}

\begin{remark}
\label{rk2} 
{\it This remark is about traveling wave solutions of \eqref{main-eq} connecting $(1,1)$ and $(0,0)$ when $\chi\le 0$. }

\begin{itemize}

\item[(1)] ({\it Monotonicity and exponential decay rate of wave profiles}).
The monotonicity and exponential decay rate of traveling wave solutions of \eqref{main-eq} with $\chi<0$ are  proved for the first time even when $m=\alpha=\gamma=1$. 

\item[(2)] ({\it  Lower bound  $c^*_{\chi,m,\gamma}$ of wave speeds}). 
The lower bound $c^*_{\chi,m,\gamma}$ established in Theorem \ref{monotone-thm}(1) for the  wave speeds in the case $\chi\le 0$  depends only on $\chi,m$, and $\gamma$, and is
not  expected to be optimal, but it is close to be optimal in certain sense. 
 
To be more precise, 
suppose that $\chi<0$ and $(u(t,x),v(t,x))=(\tilde U(x-ct),\tilde V(x-ct))$ is a traveling wave solution of \eqref{main-eq}
connecting $(1,1)$ and $(0,0)$ with speed $c$.  Then $(\tilde U,\tilde V)$ satisfy
$$
\begin{cases}
\tilde U_{xx}+c\tilde U_x-\chi (\tilde U^m \tilde V_x)_x+\tilde U(1-\tilde U^\alpha)=0,\quad &x\in\RR\cr
\tilde V_{xx}-\tilde V+\tilde U^\gamma=0,\quad &x\in\RR.
\end{cases}
$$
Let $U(x)=\tilde U(\sqrt{\frac{|\chi|\gamma}{m+\gamma}} x),\quad V(x)=\tilde V(\sqrt{\frac{|\chi|\gamma}{m+\gamma}} x),\quad c=\sqrt{\frac{|\chi|\gamma}{m+\gamma}}\bar c_\chi.
$
Then $(U(x),V(x))$ satisfy
$$
\begin{cases}
\frac{1}{|\chi|} \frac{m+\gamma}{\gamma} U_{xx}+ \bar c_\chi  U_x +\frac{m+\gamma}{\gamma} (U^mV_x)_x+U(1-U^\alpha)=0,\quad &x\in\RR\cr
\frac{1}{\sqrt{|\chi|}}\frac{m+\gamma}{\gamma}V_{xx}-V+U^\gamma=0,\quad &x\in\RR,
\end{cases}
$$
which formally converges to
\begin{equation}
\label{porous-eq}
\bar c  U_x+ (U^{m+\gamma})_{xx}+U(1-U^\alpha)=0,\quad x\in\RR
\end{equation}
as $\chi\to  -\infty$ provided that $\bar c_{\chi}\to \bar c$. Note that solutions of  \eqref{porous-eq} give rise to traveling wave solutions of the following  porous medium equation with logistic type source,
\begin{equation}
\label{porous-eq1}
U_t= (U^{m+\gamma})_{xx}+U(1-U^\alpha),\quad x\in\RR.
\end{equation}
 Note also that \eqref{porous-eq1}
 admits a minimal wave speed $c_\infty^*(m,\gamma,\alpha)>0$ (see \cite{Aro1979, AtRo, Biro}, etc.). Therefore, formally, it is expected that
$$
\liminf_{\chi\to -\infty}\frac{c^*_{\chi,m,\gamma}}{\sqrt{|\chi|}}\ge c_p^*(m,\gamma,\alpha)\sqrt{\frac{\gamma}{m+\gamma}},\,\, \text{in particular},\,\, \frac{c^*_{\chi,m,\gamma}}{\sqrt{|\chi|}}=O(1)\,\,\, {\rm as}\,\, \chi\to -\infty,
$$
which is proved to be true for the case that $m=\gamma=1$ (see \cite[Theorem 2.1]{GrHeTu2023}).

Note that
\begin{equation}
\label{speed-limit-eq}
\lim_{\chi \to -\infty}\frac{c^*_{\chi,m,\gamma}}{\sqrt {|\chi|}}=\sqrt {m\gamma +\gamma^2}.
\end{equation}
The lower bound $c^*_{\chi,m,\gamma}$ for the wave speeds  is hence  close to optimal in the sense that 
when $\chi=0$ and $\gamma=1$, $c^*_{\chi,m,\gamma}=2$,
which is the optimal lower bound for the wave speeds of \eqref{fisher-kpp-eq}, and when $\chi\to -\infty$, 
$\frac{c^*_{\chi,m,\gamma}}{\sqrt{|\chi|}}=O(1)$.

\item[(3)]({\it Speed-up of wave speeds by strong negative sensitivity}). When $m=\alpha=\gamma=1$, by \cite[Theorem 2.1]{GrHeTu2023}, strong negative chemotaxis sensitivity speeds up the wave speeds. We conjecture that 
strong negative chemotaxis sensitivity  also speeds up the wave speeds for general $m,\alpha,\gamma\ge 1$. 
Thanks to the length of the paper, we do not carry out a study on this issue in this paper.

\item[(4)] ({\it Approaches/techniques}). We   prove  Theorem \ref{monotone-thm}(1) by properly modified arguments, in particular,  super- and sub-solutions to some equation relevant to \eqref{main-eq},    developed in  \cite{SaSh2017-1, SaSh2020, SaShXu2019} for the case $m=\alpha=\gamma=1$.
\end{itemize}
\end{remark}

\begin{remark}
\label{rk3}
{\it This remark is about traveling wave solutions of \eqref{main-eq} connecting $(1,1)$ and $(0,0)$ when $\chi\ge 0$. }
\begin{itemize}

\item[(1)] ({\it Relatively small positive $\chi$ and traveling waves connecting $(1,1)$ and $(0,0)$}).
Theorem \ref{monotone-thm}(2) extends  \cite[Theorem 1.4]{SaShXu2019} for the case $m=\alpha=\gamma=1$ to the  general case $m,\alpha,\gamma\ge 1$. It is proved  by properly modified arguments, in particular,  super- and sub-solutions to some equation relevant to \eqref{main-eq},    developed in  \cite{SaSh2017-1, SaSh2020, SaShXu2019} for the case $m=\alpha=\gamma=1$.  The exponential decay rate of the traveling wave solutions is newly established even for the case $m=\alpha=\gamma=1$.

\item[(2)] ({\it Strong positive sensitivity and oscillatory traveling waves}).    In the case that $\frac{1}{2}<\frac{2\gamma+2}{2\gamma+m+1}$,
by the arguments of Theorem \ref{monotone-thm}(2), when $\frac{1}{2}\le \chi<\min\{\frac{2\gamma+2}{2\gamma+m+1},1\}$,  for any $c>2$, \eqref{main-eq} has a traveling wave solution
$(u(t,x),v(t,x))=(U(x-ct),V(x-ct))$ satisfying
$$
\lim_{x\to\infty} (U(x),V(x))=(0,0), \quad \liminf_{x\to -\infty} U(x)>0.
$$
As it is  mentioned before, when $m=\alpha=\gamma=1$,  for 
 general $\chi>0$,   it is proved in  \cite[Theorem 1.2]{HeRe2024} that   there  exist a $c\ge 2$ and  a traveling wave solution $(u(t,x),v(t,x))=(\phi(x-ct),\psi(x-ct))$ with
$\liminf_{x\to \infty}\phi(x)>0$ and $\lim_{x\to\infty}\phi(x)=0$. Such traveling wave solution with $\chi$ being  not small may not converge to $(1,1)$ and may be oscillatory as $x\to -\infty$ (see the numerical simulations in \cite{HeRe2024}).  Moreover,
It is observed numerically that 
 the minimal speed \eqref{main-eq}  with $m=\alpha=\gamma=1$ appears to always be $2$  for $\chi >0$ (see \cite{HeRe2024}). We conjecture that, for general $m,\alpha,\gamma\ge 1$, \eqref{main-eq} admits oscillatory traveling wave solutions when $\chi$ is not small and strong positive sensitivity does not speed up the wave speeds.  Again, thanks to the length of the paper, we do not carry out a study on this issue in this paper.

\end{itemize}

\end{remark}

\begin{remark}
\label{rk4}
{\it This remark is about the stability and uniqueness of traveling wave solutions of \eqref{main-eq} connecting $(1,1)$ and $(0,0)$.}
\begin{itemize}

\item[(1)] ({\it Lower bound $\chi^{**}_{\chi,m,\alpha,\gamma}$ of wave speeds}). Theorems \ref{stability-thm} and \ref{uniqueness-thm} indicate that traveling wave solutions of \eqref{main-eq} connecting $(1,1)$ and $(0,0)$ with relatively large speed $c$, namely, $c>c^{**}_{\chi,m,\alpha,\gamma}$, are stable and unique in certain weighted norm space.
Such stability and uniqueness  results are proved for the first time for traveling wave solutions of \eqref{main-eq} connecting $(1,1)$ and $(0,0)$. 
In spite that $c^{**}_{\chi,m, \alpha,\gamma}$ is not expected to be optimal,  $c^{**}_{\chi,m,\alpha,\gamma}$ is close to optimal when $|\chi|$ is small  and $\gamma=1$ in the sense
that $c^{**}_{\chi,m,\alpha,\gamma}-2=O(|\chi|^{1/6})$.

\item[(2)] ({\it Approaches/techniques}).  It is natural to study stability of traveling wave solutions in certain weighted space (see \cite{Sat1, Sat2}).  However,  thanks to the presence of chemotaxis, it is nontrivial to determine $c^{**}_{\chi,m,\alpha,\gamma}$ in Theorem \ref{stability-thm},  and  prove the stability of the traveling wave solutions of \eqref{main-eq} with speed $c>c^{**}_{\chi,m,\alpha,\gamma}$ in certain  weighted space. Our main idea is first to provide some a priori estimates for traveling wave solutions
$(u,v)$, in particular, for $|u_x|$ and $\frac{|u_x|}{u}$ (see  Lemmas \ref{estimate-lm1} and \ref{estimate-lm2});  next  set up some a priori assumption on the wave speed (see \eqref{cond-on-c});  then determine $c^{**}_{\chi,m,\alpha,\gamma}$ (see \eqref{c-star-star-eq}) via various nontrivial estimates; finally prove \eqref{main-stability-eq1} and \eqref{main-stability-eq2}. We believe that the techniques developed in this paper can be applied to some other chemotaxis models.
\end{itemize}

\end{remark}

\begin{remark}
\label{rk5}
{\it This remark is about some comments given by Zhi-An Wang}.  The author of the current paper gave a presentation of the main results of this paper  at  the International Conference on Nonlinear PDEs and
Applications in Mathematical Biology, Shanghai Jiao-Tong University. 
April 20-24, 2026. After the presentation, Wang commented that, in an on-going work, they proved  the monotonicity and stability of traveling wave solutions of  \eqref{main-eq} connecting $(1,1)$ and $(0,0)$ in the special case $m=\alpha=\gamma=1$ by some different approaches.
\end{remark}

The rest of the paper is organized as follows. In Section 2, we present some preliminary lemmas to be used in the proofs of the main results.  We study global existence and boundedness of positive classical solutions of \eqref{main-eq}  and  the stability of the positive constant solution $(1,1)$ of \eqref{main-eq}, and prove Propositions \ref{global-existence-prop} and
\ref{constant-stability-prop} in Section 3. In Section 4,   we study the existence of traveling wave solutions of \eqref{main-eq} connecting
$(1,1)$ and $(0,0)$, and prove Theorem \ref{monotone-thm}. Finally, in Section 5,  we study the stability of traveling wave solutions of \eqref{main-eq} and prove Theorem \ref{stability-thm}. 

\section{Preliminary lemmas}
\label{prelim-section}

In this section, we present   some basic estimates for the analytic semigroup $e^{(\Delta -I)t}$ and  some basic properties for the solutions of the elliptic equation
$v_{xx}-\lambda v+\mu u=0$ for  $\lambda,\mu>0$.

\subsection{Basic estimates for the analytic semigroup $e^{(\Delta -I)t}$}

Let $T(t)=e^{(\Delta -I)t}$ be the analytic semigroup generated by $\Delta-I$ on
$X=C_{\rm unif}^b(\RR)$ or $L^p(\RR)$ with $1\le p<\infty$.  In this  subsection, we present some basic estimates for  $e^{(\Delta -I)t}$.

\begin{lemma}
\label{prelim-semigroup-lm1}
For any $1<p\le q<\infty$, there are $C_p>0$ and $C_{p,q}>0$  such that for any $u\in L^p(\RR)$, the following hold:
\begin{equation}
\label{prelim-semigroup-eq1}
\|e^{(\Delta -I)t}u\|_{L^q(\RR)}\le C_{p,q} t^{-\frac{1}{2}\big(\frac{1}{p}-\frac{1}{q}\big)} e^{-t}\|u\|_{L^p(\RR)},
\end{equation}
\begin{equation}
\label{prelim-semigroup-eq2}
\|\nabla e^{(\Delta -I)t}u\|_{L^q(\RR)}\le C_{p,q} t^{-\frac{1}{2}-\frac{1}{2}\big(\frac{1}{p}-\frac{1}{q}\big)} e^{-t}\|u\|_{L^p(\RR)},
\end{equation}
\begin{equation}
\label{prelim-semigroup-eq3}
\|e^{(\Delta -I)t}\nabla \cdot u\|_{L^\infty(\RR)}\le C_p t^{-\frac{1}{2}-\frac{1}{2p}}e^{-t}\|u\|_{L^p(\RR)}.
\end{equation}
\end{lemma}

\begin{proof}
The estimates \eqref{prelim-semigroup-eq1} and \eqref{prelim-semigroup-eq2} follow from  the $L^p$-$L^q$ estimates for the convolution product 
and \eqref{prelim-semigroup-eq3}  follows from \cite[Lemma 3.1]{SaSh2017}.
\end{proof}

\subsection{Basic properties of the elliptic equation $v_{xx}-\lambda  v+ \mu u=0$}

In this subsection, we present some basic properties for the solutions of the elliptic equation
\begin{equation}
\label{general-v-eq}
v_{xx}-\lambda  v+ \mu u=0,\quad x\in\RR,
\end{equation}
where $\lambda,\mu>0$ are positive constants. 
For  $u\in C^b_{\rm unif}(\R)$, let $\Psi(x;u,\lambda,\mu)$ be the solution of \eqref{general-v-eq}.
Note that
\begin{equation}\label{psi-eq1}
\Psi(x;u,\lambda,\mu)=\mu \int_{0}^{\infty}\int_{\R}\frac{e^{-\lambda  s}e^{-\frac{|y-x|^2}{4s}}}{\sqrt{4\pi s}}u(y)dyds.
\end{equation}

\begin{lemma}
\label{prelim-v-lm1} It holds that
\begin{equation}\label{psi-eq2}
\Psi(x;u,\lambda, \mu)=\frac{\mu  }{2\sqrt{\lambda}}\int_{\R}e^{-\sqrt{\lambda } |x-y|}u(y)dy,
\end{equation}
and
\begin{equation}\label{psi-eq3}
\frac{d}{dx}\Psi(x;u,\lambda,\mu )=-\frac{\mu  }{2}e^{-\sqrt{\lambda}  x}\int_{-\infty}^xe^{\sqrt{\lambda}  y}u(y)dy + \frac{\mu }{2}e^{\sqrt {\lambda} x}\int_{x}^{\infty}e^{-\sqrt{\lambda}  y}u(y)dy.
\end{equation}
\end{lemma}

\begin{proof}
It follows from \cite[Lemma 2.1]{SaShXu2019}.
\end{proof}

\begin{lemma}\label{prelim-v-lm2}
For every $u\in C^b_{\rm unif}(\R)$, $u(x)\ge 0$, it holds that
\begin{equation}\label{estimate-on-space-derivative-1}
\left| \frac{d}{dx}\Psi(x;u,\lambda,\mu)\right|\leq \sqrt{\lambda} \Psi(x;u),\ \quad  \forall\ x\in\R.
\end{equation}
\end{lemma}

\begin{proof}
It follows from \cite[Lemma 2.2]{SaShXu2019}.
\end{proof}

\begin{lemma}
\label{prelim-v-lm3}
For any $0<\kappa<1$, let $\varphi_\kappa =e^{-\kappa x}$. Then,  for any $M\ge 1$ and 
$u_0\in C_{\rm unif}^b(\RR)$ with  $0\le u\le \min\{M,\varphi_\kappa\}$,  it holds that
$$
\Psi(x;u):=\Psi(x;u,1,1)\le \min\big\{M,\frac{1}{1-\kappa^2} e^{-\kappa x}\big\}\quad \forall\, x\in\RR.
$$
\end{lemma}

\begin{proof}
By \eqref{psi-eq2},
\begin{align*}
\Psi(x;u)&=\frac{1}{2}\int_{-\infty}^\infty e^{-|x-y|} e^{-\kappa y} u(y)dy\\
&\le \frac{1}{2}e^{-x}\int_{-\infty}^x e^{(1-\kappa)y}dy+\frac{1}{2} e^x \int_x^\infty e^{(-1-\kappa)y}dy\\
&=\frac{1}{2(1-\kappa)} e^{-\kappa x}+\frac{1}{2(1+\kappa)} e^{-\kappa x}=\frac{1}{1-\kappa^2} e^{-\kappa x}.
\end{align*}
Similarly, we have
$$
\Psi(x;u)\le M\quad \forall\, x\in\RR.
$$
The lemma is thus proved.
\end{proof}

Note that for $u\in C_{\rm unif}^b(\RR)$,  $\Psi(\cdot;u,\lambda,\mu)\in L_{\rm loc}^p(\RR)$ for any $p\ge 1$, but we may not have $\Psi(\cdot;u,\lambda,\mu)\in L^p(\RR)$. To estimate $\|\Psi(\cdot;u,\lambda,\mu)\|_{L_{\rm loc}^p(\RR)}$, it suffices to estimate
$\|\Psi(\cdot;u,\lambda,\mu)\psi(\cdot)\|_{L^p(\RR)}$ for certain  exponential decay function $\psi$. 
Note also that,     for  any {$\kappa_1>0$,  there are  $\kappa_2>0$} and
     $\psi\in C^\infty(\R^n)$  such that
    \begin{equation}\label{psi-eq00}
        0<\psi(x)\le e^{-\kappa_2|x|},\quad | \psi_x(x)|\le \kappa_1 \psi(x),\quad | \psi_{xx} (x)|\le \kappa_1\psi \quad  \forall\, x\in\RR
    \end{equation}
(see \cite[Lemma 2.4]{HaShZh2024}). 

\begin{lemma}
\label{prelim-v-lm4}
For given $p>1$, 
there is $C'_p$ such that for any $\kappa_1\ll  1$,
$$
\int_{\RR} |\nabla \Psi(\cdot;u^\gamma,\lambda,\mu) |^{p}\psi\le C'_p \int_{\RR} u^{\gamma p}\psi,
$$
where $\psi$ is as in \eqref{psi-eq00}, and $\gamma>0$.
\end{lemma}

\begin{proof}
It can be proved by the arguments of \cite[Proposition 3.1]{HaShZh2024}. For reader's  convenience, we give a proof in the following.

First, let   $v(x)=\Psi(x;u^\gamma,\lambda,\mu)$ and $w=v\psi$.  Then $w$ satisfies 
$$
w_{xx}-2v_x\psi_x+v\psi_{xx}-\lambda w+\mu u^\gamma \psi=0.
$$
By \eqref{psi-eq1}, we have
\begin{align*}
w&= \int_{0}^{\infty}\int_{\RR}\frac{e^{-\lambda  s}e^{-\frac{|y-x|^2}{4s}}}{\sqrt{4\pi s}} \big(  -2v_x\psi_x+v\psi_{xx}+\mu u^\gamma \psi\big) dyds\\
&=\int_0^\infty e^{(1-\lambda)s} e^{(\Delta -I)s}  \big(  -2v_x\psi_x+v\psi_{xx}+\mu u^\gamma \psi\big) 
ds,
\end{align*}
and 
\begin{align*}
\nabla w
&=\int_0^\infty \nabla e^{(1-\lambda)s} e^{(\Delta -I)s}  \big(  -2v_x\psi_x+v\psi_{xx}+\mu u^\gamma \psi\big) 
ds.
\end{align*}
Then by \eqref{prelim-semigroup-eq1}, it holds that
    \begin{align}
    \label{main-estimate-eq1}
    \|v(\cdot)\psi(\cdot)\|_{L^{p}}
    &\le C_{p,p}  \int_0^\infty e^{-\lambda s} \Big(2 \kappa_1\|(\nabla v (\cdot)) \psi\|_{L^p}+\kappa_1\|v(\cdot) \psi\|_{L^p}+\mu \|u^\gamma (\cdot)\psi\|_{L^p}\Big)ds\nonumber\\
    &\le \frac{C_{p,p}}{\lambda}  \mu \|u^\gamma (\cdot)\psi\|_{L^p}  + \frac{C_{p,p}}{\lambda}  \Big(2 \kappa_1 \|\nabla v (\cdot) \psi\|_{L^p}+\kappa_1
    \|v(\cdot) \psi\|_{L^p}\Big).
    \end{align}
   By \eqref{prelim-semigroup-eq2}, it holds that
    \begin{align}
    \label{main-estimate-eq2}
    &\|(\nabla v(\cdot)) \psi\|_{L^{p}}=\|\nabla (v(\cdot)\psi)-v(\cdot)\nabla \psi\|_{L^{p}}\nonumber\\
    &\le \|v(\cdot) \nabla \psi\|_{L^{p}}+\|\nabla (v(\cdot)\psi)\|_{L^{p}}\nonumber\\
    &\le \kappa_1\|v(\cdot)\psi\|_{L^{p}}+C_{p,p}\int_0^\infty e^{-\lambda s} s^{-\frac{1}{2}}\Big(2 \kappa_1\|(\nabla v (\cdot)) \psi\|_{L^p}+\kappa_1\|v(\cdot) \psi\|_{L^p}+\mu \|u^\gamma (\cdot)\psi\|_{L^p}\Big)ds\nonumber\\
    &\le \kappa_1{\|v(\cdot)\psi\|_{L^{p}}}+C_{p,p} \mu \|u^\gamma (\cdot)\psi\|_{L^p}\int_0^\infty e^{-\lambda s}s^{-\frac{1}{2}} ds\nonumber\\
    &\quad + C_{p,p}\kappa_1 \Big(2 \|\nabla v(\cdot)\psi\|_{L^p}+\|v(\cdot)\psi\|_{L^p}\Big) \int_0^\infty e^{-\lambda s}s^{-\frac{1}{2}} ds.
    \end{align}
By \eqref{main-estimate-eq1} and \eqref{main-estimate-eq2},  
$$
\|v\psi\|_{L^p}+\|\nabla v\psi\|_{L^p}(\RR)\le \tilde C_p \|u^\gamma \psi\|_{L^p(\RR)}
$$
for all $0<\kappa_1\ll 1$, where $\tilde C_p=2\Big(\frac{\mu C_{p,p}}{\lambda}+\mu C_{p,p}\int_0^\infty e^{-\lambda s} s^{-1/2} ds\Big)$.

Replacing $\psi$ be $\psi^{1/p}$ in the above arguments, there is $C'_p$ independent of $u$ such that 
$$
\int_{\RR} |\nabla v|^p\psi\le C'_p\int_{\RR}u^{\gamma p}\psi
$$
for all $\kappa_1\ll 1$.
The lemma is thus proved.
    \end{proof}

\section{Global existence, boundedness, and stabilization of positive classical solutions and proofs of Propositions  \ref{global-existence-prop} and
\ref{constant-stability-prop}}
\label{cauchy-problem-section}

In this section, we study the global existence and boundedness of classical solutions of \eqref{main-eq} with 
initial condition $u_0\in C_{\rm unif}^b(\RR)$  and $u_0\ge 0$, and  the stability of the positive constant solution $(1,1)$ of \eqref{main-eq}, and prove Propositions \ref{global-existence-prop} and
\ref{constant-stability-prop}.

\subsection{Global existence and boundedness of positive classical solutions
and
proof of Proposition \ref{global-existence-prop}}

In this subsection, we prove Proposition \ref{global-existence-prop}.

\begin{proof} [Proof of Proposition \ref{global-existence-prop}(1)]
 Assume that   $\chi\le 0$.   We first prove $T_{\max}(u_0)=\infty$ and \eqref{upper-bound-eq1} holds. 
The idea to do so is  described 
as follows. For given $T>0$, let
$$
\mathcal{X}_T=C([0,T],C_{\rm unif}^b(\RR))
$$
endowed with the norm
$$
\|u\|_{\mathcal{X}_T}=\sum_{k=1}^\infty\frac{1}{2^k}\|u(\cdot,\cdot)\|_{L^\infty([0,T]\times [-k,k])}.
$$
For  given $T>0$ and $u_0\in C_{\rm unif}^b(\RR)$ with $u_0\ge 0$, let
$$
\underline u_0=\min\{1,\inf_{x\in\RR}u_0(x)\},\quad \bar u_0=\max\{1,\sup_{x\in\RR}u_0(x)\},
$$
and
$$
\mathcal{E}_T(u_0)=\{ \mathcal{X}_T\,|\, u(0,x)=u_0(x),\,\,   \underline u_0\le u(t,x)\le\bar u_0\}.
$$
Then $\mathcal{X}_T$ is a locally convex topological vector space and  $\mathcal{E}_T$ is a nonempty convex closed  subset of $\mathcal{X}_T$.
We prove via 
 Schauder fixed-point theorem that for any $T>0$,  there hold $T<T_{\max}(u_0)$ and $u(\cdot,\cdot;u_0)\in \mathcal{E}_T$.  We provide the detailed proof in three steps.

\smallskip

\noindent{\bf Step 1.} In this step, we construct a mapping $\mathcal{T}:\mathcal{E}_T(u_0)\to \mathcal{E}_T(u_0)$.

\smallskip

For any $u\in \mathcal{E}_T(u_0)$, let 
$$
V(t,x;u)=\Psi(x;u^\gamma(t,\cdot),1,1),
$$
where $\Psi(x;u^\gamma(t,\cdot),1,1)$ is as in \eqref{psi-eq1} with $\mu=\lambda=1$.  Note that
$$
\underline u_0^\gamma\le u^\gamma(t,x)\le \bar u_0^\gamma\quad \forall\, t\in [0,T],\,\, x\in\mathbb{R}.
$$
By comparison principle for elliptic equations, it holds that
$$
\underline u_0^\gamma\le V(t,x;u)\le \bar u_0^\gamma\quad \forall\, t\in [0,T],\,\, x\in\mathbb{R}.
$$
 Let
$U(t,x;u)$ be the solution of
\begin{equation}
\label{aux-eq1}
\begin{cases}
U_t=U_{xx}+|\chi|m U^{m-1} U_x V_x +U\big(1+|\chi| U^{m-1} V - (U^\alpha +|\chi|U^{m+\gamma -1})\big),\quad &x\in\RR\cr
U(0,x;u_0)=u_0(x),\quad &x\in\RR.
\end{cases}
\end{equation}
It is clear that $U\equiv \underline u_0$ is a sub-solution of \eqref{aux-eq1} and $U\equiv \bar u_0$ is a super-solution of \eqref{aux-eq1}.
Then by comparison principle for parabolic equations, we have
$$
\underline u_0\le U(T,x;u)\le\bar u_0\quad \forall\, t\in [0,T],\,\, x\in\mathbb{R}.
$$
Hence
$$
U(\cdot,\cdot;u)\in \mathcal{E}_T,\quad U(0,\cdot;u)=u_0(\cdot),
$$
and then the mapping  $\mathcal{T}$ defined by 
$\mathcal{T}(u)=U(\cdot,\cdot;u)$ maps 
$\mathcal{E}_T(u_0)$ into itself.

\smallskip

\noindent{\bf Step 2.} In this step, we prove that the mapping $\mathcal{T} :\mathcal{E}_T(u_0)\to  \mathcal{E}_T(u_0)$ is  continuous and compact. 

\smallskip

The compactness and continuity of the mapping $\mathcal{T}$ can be proved by  the similar arguments  of \cite[Lemma 4.3]{SaSh2017-1}.
For reader's convenience, we outline the proof through  the following two claims.

\smallskip

\noindent{\bf Claim 1.} For any sequence $\{u_n\}\subset \mathcal{E}_T(u_0)$, there is a subsequence $\{u_{n_k}\}$ such that $\{\mathcal{T}(u_{n_k})\}$ converges  to some $U\in\mathcal{E}_T(u_0)$ in the norm $\|\cdot\|_{\mathcal{X}_T}$, which implies  that the  mapping $\mathcal{T} :\mathcal{E}_T(u_0)\to  \mathcal{E}_T(u_0)$ is compact.  The idea to prove this claim is as follows.  First,  let
$U_n(t,x)=U(t,x;u_n)$. Then
\begin{align*}
U_n(t,\cdot)&=e^{(\Delta -I)t}u_0(\cdot)-\chi \int_0^t e^{(\Delta-I)(t-s)}\partial_x \big(U_{n}^m(s,\cdot)V_x(s,\cdot;u_{n}(s,\cdot))\big)dx\\
&\quad +\int_0^t e^{(\Delta-I)(t-s)}U_{n}(s,\cdot) 
\big(2-U^\alpha_{n}(s,\cdot)\big)ds.
\end{align*}
Applying this equality, the boundedness of $U_n$,  and the estimates in Lemma \ref{prelim-semigroup-lm1}, we can prove
that
\begin{equation}
\label{new-aux-eq1}
\lim_{t\to 0+} \|U(t,\cdot;u_{n})-u_0(\cdot)\|_\infty=0\quad \text{
uniformly in}\quad n\ge 1.
\end{equation} 

Next,  applying \eqref{new-aux-eq1}, a priori estimates for
parabolic equations,  and the diagonal argument, we can prove that there are $\{u_{n_k}\}$ and  $U(t,x)\in\mathcal{X}_T$ such that
\begin{equation}
\label{new-aux-eq2}
\lim_{n_k\to\infty} U(t,x;u_{n_k})=U(t,x)\quad \text{locally uniformly in}\quad (t,x)\in [0,T]\times\mathbb{R}.
\end{equation}
By \eqref{new-aux-eq1} and \eqref{new-aux-eq2},
$$
\lim_{n_k\to\infty}\|U(\cdot,\cdot;u_{n_k})-U(\cdot,\cdot)\|_{\mathcal{X}_T}=0,
$$
and $U(0,x)=u_0(x)$. Hence the mapping $\mathcal{T}$ is compact.

\smallskip

\noindent {\bf Claim 2.} The mapping $\mathcal{T}$ is continuous. The idea to  prove this claim is as follows.  Suppose that $u_n,\tilde u\in \mathcal{E}_T(u_0)$, and $\lim_{n\to\infty}\|u_n-\tilde u\|_{\mathcal{X}_T}=0$.  First,  applying \eqref{psi-eq2} and \eqref{psi-eq3}, we can prove that
$$
\lim_{n\to\infty}\Big( \|V(\cdot,\cdot;u_n)-V(\cdot,\cdot;\tilde u)\|_{\mathcal{X}_T}+
\|V_x(\cdot,\cdot);u_n)-V_x(\cdot,\cdot;\tilde u)\|_{\mathcal{X}_T}\Big)=0.
$$
Next, we prove $\lim_{n\to\infty}\|U(\cdot;\cdot;u_n)-U(\cdot;\cdot;\tilde u)\|_{\mathcal{X}_T}=0$
by contradiction. Assume that  $U(t,x;u_n)$ does not converge to $U(t,x;\tilde u)$ in $\mathcal{X}_T$. Then there is $\epsilon_0>0$ and $u_{n_k}$ such that
\begin{equation}
\label{new-aux-eq3}
\|U(\cdot,\cdot;u_{n_k})-U(\cdot,\cdot;\tilde u)\|_{\mathcal{X}_T}\ge \epsilon_0.
\end{equation}
By the compactness of $\mathcal{T}$, without loss of generality, we may assume that
there is
$\tilde U\in\mathcal{X}_T(u_0)$ such that
$$
\lim_{n_k\to\infty}\|U(\cdot,\cdot;u_{n_k})-\tilde U(\cdot,\cdot)\|_{\mathcal{X}_T}=0.
$$
Let $\tilde V(t,x)=V(t,x;\tilde u)$. Then both $\tilde U(t,x)$ and $u(t,x;\tilde u)$ are solutions of
\eqref{aux-eq1} with $V=V(t,x;\tilde u)$.
Then we must have $U(t,x;\tilde u)=\tilde U(t,x)$ and hence $\lim_{n_k\to\infty}\|U(\cdot,\cdot;u_{n_k})-u(\cdot,\cdot;\tilde u)\|_{X_T}=0$, which contradicts to \eqref{new-aux-eq3}.
This implies $\lim_{n\to\infty}\|U(\cdot;\cdot;u_n)-U(\cdot;\cdot;\tilde u)\|_{\mathcal{X}_T}=0$, that is, the mapping $\mathcal{T}$ is continuous.

\smallskip

\noindent{\bf Step 3.}  In this step, we show that for any $T>0$,  we have $T<T_{\max}(u_0)$ and $u(\cdot,\cdot;u_0)\in\mathcal{X}_T(u_0)$.

\smallskip

 First, for any $T>0$, by the arguments in Steps 1 and 2, the mapping $\mathcal{T}:\mathcal{E}_T(u_0)\to\mathcal{E}_T(u_0)$ is continuous and compact. Then,  by Schauder fixed-point theorem, there is $u^*(\cdot,\cdot)\in \mathcal{E}_T(u_0)$ such that
$$
U(t,x;u^*)=u^*(t,x)\quad \forall\, t\in [0,T],\,\, x\in\RR.
$$
It is clear that 
$$
u(t,x;u_0)=u^*(t,x),\quad v(t,x;u_0)=V(t,x;u^*)\quad \forall\, t\in [0,T_{\max}(u_0))\cap [0,T],\,\, x\in\RR.
$$
Note that
$$
\limsup_{t\to T_{\max}(u_0)-}\|u(t,\cdot;u_0)\|_\infty=\infty \quad {\rm if}\quad T_{\max}(u_0)<\infty.
$$
It then follows that  $T<T_{\max}(u_0)$ and $u(\cdot,\cdot;u_0)\in \mathcal{E}_T$
for any $T<\infty$. Therefore, $T_{\max}(u_0)=\infty$ and $u(t,x;u_0)\le \max\{1,\sup_{x\in\RR}u_0(x)\}\quad \forall\, t\ge 0,\,\, x\in\RR$,  i.e., \eqref{upper-bound-eq1} holds.

In the following, we show that \eqref{upper-limit-eq1} holds. 
To this end,  we first  prove that if $\bar u(t_0)>1$ for some $t_0>0$, then
$\bar u(t)$  is  non-increasing on $[0,t_0)$. 
Suppose that $t_0\in (0,\infty)$ is such that $\bar u(t_0)>1$.  Note that $\bar u(t)$ is continuous in $t\in [0,\infty)$.
Assume that there is $t_1\in [0,t_0)$ such that  $\bar u(t_1)<\bar u(t_0)$. Then there is $t_2\in [t_1,t_0)$ such that $1\le \bar u(t_2)<\bar u(t_0)$.  By the arguments of \eqref{upper-bound-eq1}, $\bar u(t_0)\le \bar u(t_2)$, which is a contradiction.
Hence, for any $t\in [0,t_0)$, $\bar u(t)\ge \bar u(t_0)>1$. Then, for any $t_1<t_2$,   by the arguments in Proposition 
\ref{global-existence-prop}(1) again, 
$\bar u(t_2)\le \bar u(t_1)$. Therefore,  $\bar u(t)$ is non-increasing on $[0,t_0)$.

Next, we prove that
\eqref{upper-limit-eq1} holds by contradiction. 
Assume that \eqref{upper-limit-eq1} does not hold. 
Then there are $t_n\to \infty$ and $x_n\in\RR$ such that $\lim_{n\to\infty} u(t_n,x_n)$ exists, and 
$$
\lim_{n\to\infty}u(t_n,x_n;u_0)>1.
$$
 By the above arguments,  $\bar u(t)$ is non-increasing on $[0,\infty)$. Hence
$$
\bar u(t)>1\quad \forall\, t>0\quad {\rm and}\quad u_+:=\lim_{t\to\infty}\bar u(t)>1.
$$
Let $\tilde x_n\in\mathbb{R}$ be such that 
$$
|u(t_n,\tilde x_n;u_0)-\bar u(t_n)|<\frac{1}{n}\quad {\rm for}\quad n=1,2,\cdots.
$$
Without loss of generality, we may assume that
the limit $(u^*(t,x),v^*(t,x)):=\lim_{n\to\infty} (u(t+t_n$, $x+\tilde x_n;u_0), v(t+t_n,x+\tilde x_n;u_0))$ exists locally uniformly in $t\in\mathbb{R}$ and $x\in\mathbb{R}$,  and
$$
u^*_t=u^*_{xx}-\chi m (u^*)^{m-1} u^*_x v^*_x -\chi (u^*)^m (v^*-(u^*)^\gamma)+u^*(1-(u^*)^\alpha),\quad x\in\mathbb{R},
$$
and
$$
0=v^*_{xx}-v^*+(u^*)^\gamma.
$$
Then, by the monotonicity of $\bar u(t)$, there holds
$$
\bar u^*(t):=\sup_{x\in\mathbb{R}}u^*(t,x) \le u_+(>1) \quad \forall\, t\in\mathbb{R},
\quad {\rm and}\quad 
u^*(0,0)=u_+.
$$
Hence $u^*(t,x)$ achieves its maximum $u_+$ at $(0,0)$.
By comparison principle for elliptic equations,
$$
v^*(t,x)\le u_+^\gamma \quad \forall\, t\in\mathbb{R},\,\, x\in\mathbb{R}.
$$
It then follows that
$$
0\le   u^*_t(0,0)=u^*_{xx}(0,0)-\chi (v^*(0,0)-u_+^\gamma)+u_+(1-u_+^\alpha)<0,
$$
which is a contradiction.  Therefore    \eqref{upper-limit-eq1} holds.  Proposition \ref{global-existence-prop}(1) is thus proved.
\end{proof}

\begin{proof} [Proof of Proposition \ref{global-existence-prop}(2)]
Assume that the condition (ii) holds, that is, 
$\chi>0$ and $\alpha>m+\gamma-1$ or  $0<\chi<\min\{ \frac{2m-1}{m-1}, \frac{m+\gamma -1}{\gamma-1}\}$ and $\alpha=m+\gamma-1$.

\smallskip

We first prove $T_{\max}(u_0)=\infty$ and \eqref{upper-bound-eq2} holds.
The idea to do so  is described as follows:
 First, we  prove $\|u(t,\cdot;u_0)\|_{L_{\rm loc}^p(\RR)}$ stays bounded for some $p>\max\{1,m,\gamma\}$, next we  prove that  $ \|\nabla v(t,\cdot;u_0)\|_\infty$ stays bounded, and then  we prove that $\|u(t,\cdot;u_0)\|_\infty$ stays bounded. We provide the detailed proof in three steps.

\smallskip

\noindent{\bf Step 1.} In this step, we prove  that $\|u(t,\cdot;u_0)\|_{L_{\rm loc}^p(\RR)}$ stays bounded. 

\smallskip

To this end,  first choose $\kappa_1$ to be sufficiently small (to be determined later),  and let
     $\psi\in C^\infty(\R^n)$  be as in \eqref{psi-eq00}.
 To prove $\|u(t,\cdot;u_0)\|_{L_{\rm loc}^p(\RR)}$ stays bounded, it then suffices to prove $\int_\Omega u^p(t,x;u_0)\psi(x-x_0)dx$ stays bounded uniformly in $x_0\in\RR$.

 Note that, for any  $p>1$.
    we have
    \begin{align*}
\frac{1}{p}\frac{d}{dt}\int_{\RR} u^p \psi dx&=\int_{\RR}u^{p-1}\psi\Delta u-\chi\int_{\RR}u^{p-1}\psi\nabla \cdot(u^m \nabla v)dx+\int_{\RR}u^p \psi(1- u^\alpha )dx\nonumber\\
    &=-(p -1)\int_{\RR} u^{p-2}|\nabla u|^2\psi-\int_{\RR} u^{p -1}\nabla u\cdot\nabla \psi\nonumber\\
    &\quad +\chi (p-1)\int_{\RR} u^{m+p -2} \nabla u\cdot (\nabla v)\psi
    +\chi\int_{\RR}u^{p+m-1} \nabla v\cdot\nabla \psi+\int_{\RR}u^p (1-u^\alpha)\psi.
    \end{align*}
By $\Delta v= v- u^\gamma $, there holds
    \begin{align*}
    &\quad\,\,    \chi(p-1)\int_{\RR}u^{p+m-2}\nabla u\cdot (\nabla v)\psi+\chi\int_{\RR} u^{p+m-1} \nabla v\cdot\nabla \psi\\
    &= -\frac{\chi(p-1)}{p+m-1}\int_{\RR}u^{p+m-1}  \nabla v\cdot\nabla \psi-\frac{\chi(p-1)}{p+m-1}\int_{\RR}u^{p+m-1}  (\Delta v) \psi+\chi\int_{\RR} u^{p+m-1} \nabla v\cdot\nabla \psi\\
    &=\frac{\chi m}{p+m-1}\int_{\RR}
    u^{p+m-1} \nabla v\cdot\nabla \psi-\frac{\chi(p-1)}{p+m-1}\int_{\RR} u^{p+m-1} v\psi+\frac{\chi(p-1)}{p+m-1}\int_{\RR} u^{p+m+\gamma-1}\psi\\
    &\le \frac{\chi m \kappa_1}{p+m-1}\int_{\RR}
    u^{p+m-1} |\nabla v|\psi  +\frac{\chi(p - 1)}{p+m-1}\int_{\RR} u^{p+m+\gamma -1}\psi.
    \end{align*}
We then have
 \begin{align}
\label{prop1-2-eq1}
\frac{1}{p}\frac{d}{dt}\int_{\RR} u^p \psi dx&\le -(p -1)\int_{\RR} u^{p-2}|\nabla u|^2\psi
+\frac{\kappa_1}{2}\int_{\RR} u^{p-1} |\nabla u|^2 \psi +\frac{\kappa_1}{2}\int_{\RR} u^p \psi\nonumber\\
&\quad + \frac{\chi m \kappa_1}{p+m-1}\int_{\RR}
    u^{p+m-1} |\nabla v|\psi  +\frac{\chi(p - 1)}{p+m-1}\int_{\RR} u^{p+m+\gamma-1}\psi\nonumber\\
    &\quad +\int_{\RR}u^p \psi-\int_{\RR} u^{p+\alpha}\psi.
    \end{align}

In the case   $\alpha>m+\gamma -1$, fix a  $p$ such that $\max\{1,m,\gamma\}<p<m+\gamma$. In the case
that $\alpha=m+\gamma-1$ and   $0<\chi<\min\{ \frac{2m-1}{m-1}, \frac{m+\gamma -1}{\gamma-1}\}$, fix a $p$ such that  $\max\{1,m,\gamma\}<p<m+\gamma$ and 
$$
\chi<\frac{p+m-1}{p-1}.
$$
Note that
\begin{align*}
\int_{\RR} u^{p+m-1}|\nabla v|\psi \le \frac{1}{4}\int_{\RR} u^{p+\alpha} +\int_{\RR} |\nabla v|^{\frac{p+\alpha}{\alpha+1-m}}\psi.
\end{align*}
By  Lemma \ref{prelim-v-lm4},  there are $C^*_p>0$ and  $\kappa_1^*>0$ such that for any $\kappa_1<\kappa_1^*$, 
$$
\int_{\RR} |\nabla v|^{\frac{p+\alpha}{\alpha+1-m}}\psi\le C^*_p \int_{\RR} u^{\frac{\gamma (p+\alpha)}{\alpha+1-m}}\psi.
$$
We then have
\begin{align}
\label{prop1-2-eq2}
\frac{1}{p}\frac{d}{dt}\int_{\RR} u^p \psi dx&\le -(p -1)\int_{\RR} u^{p-2}|\nabla u|^2\psi
+\frac{\kappa_1}{2}\int_{\RR} u^{p-1} |\nabla u|^2 \psi +\frac{\kappa_1}{2}\int_{\RR} u^p \psi\nonumber\\
&\quad + \frac{\chi m \kappa_1}{4(p+m-1)}\int_{\RR} u^{p+\alpha}\psi+C^*_p \frac{\chi m \kappa_1}{p+m-1} \int_{\RR} u^{\frac{\gamma (p+\alpha)}{\alpha+1-m}}\psi\nonumber\\
&\quad +\frac{\chi(p - 1)}{p+m-1}\int_{\RR} u^{p+m+\gamma-1}\psi+\int_{\RR}u^p \psi-\int_{\RR} u^{p+\alpha}\psi.
    \end{align}

Next, 
in the case that $\alpha>m+\gamma -1$, it holds 
$$
\frac{\gamma(p+\alpha)}{\alpha+1-m}< p+\alpha.
$$ 
By Young's inequality, for any $\epsilon>0$, there is $C_\epsilon>0$ such that
$$
C^*_p \frac{\chi m \kappa_1}{p+m-1} \int_{\RR} u^{\frac{\gamma (p+\alpha)}{\alpha+1-m}}\psi\le \epsilon \int_{\RR} u^{p+\alpha}\psi +C_\epsilon \int_{\RR}\psi,
$$
and
$$
\frac{\chi(p - 1)}{p+m-1}\int_{\RR} u^{p+m+\gamma-1}\psi\le \epsilon \int_{\RR} u^{p+\alpha}\psi +C_\epsilon\int_{\RR}\psi.
$$
Choose  $\kappa_1<\min\{\kappa_1^*, 2(p-1), \frac{p+m-1}{\chi m}\}$ and $\epsilon=\frac{1}{8}$. Then by \eqref{prop1-2-eq2},
\begin{align}
\label{prop1-2-eq2-1}
\frac{1}{p}\frac{d}{dt}\int_\RR u^p \psi dx\le 2C_\epsilon \int_{\RR} \psi+\big(1+\frac{\kappa_1}{2}\big)\int_{\RR}u^p\psi -\frac{1}{2}\int_{\RR} u^{p+\alpha}\psi.
\end{align}

In the case that  $\alpha=m+\gamma-1$,  
$$
\frac{\gamma(p+\alpha)}{\alpha+1-m}=p+\alpha.
$$ 
By \eqref{prop1-2-eq2}, we
have
 \begin{align*}
\frac{1}{p}\frac{d}{dt}\int_{\RR} u^p \psi dx&\le -(p -1)\int_{\RR} u^{p-2}|\nabla u|^2\psi
+\frac{\kappa_1}{2}\int_{\RR} u^{p-1} |\nabla u|^2 \psi +\big(1+\frac{\kappa_1}{2}\big)\int_{\RR} u^p \psi\nonumber\\
&\quad + \Big(\frac{\chi m \kappa_1}{4(p+m-1)}+C^*_p \frac{\chi m \kappa_1}{p+m-1} +\frac{\chi(p - 1)}{p+m-1}-1\Big)\int_{\RR} u^{p+\alpha}\psi.
    \end{align*}
Note that $\frac{\chi(p - 1)}{p+m-1}-1<0$. Hence we can 
choose $0<\kappa_1<\min\{\kappa_1^*,2(p-1)\}$ such that
$$
\Big(\frac{\chi m \kappa_1}{4(p+m-1)}+C^*_p \frac{\chi m \kappa_1}{p+m-1} +\frac{\chi(p - 1)}{p+m-1}-1\Big)<0.
$$
This together with \eqref{prop1-2-eq2} implies that   there is  $\lambda>0$ such that
\begin{align}
\label{prop1-2-eq2-2}
\frac{1}{p}\frac{d}{dt}\int_{\RR} u^p \psi dx\le 
(1+\kappa_1/2)\int_{\RR} u^p\psi -\lambda \int_{\RR} u^{p+\alpha}\psi.
\end{align}

By \eqref{prop1-2-eq2-1} in the case that $\alpha>m+\gamma-1$ and \eqref{prop1-2-eq2-2} in the case $\alpha=m+\gamma-1$, we have 
\begin{equation*}
\sup_{t\in [0,T_{\max}(u_0))}\int_{\RR} u^p\psi<\infty.
\end{equation*}
Moreover, 
by replacing $\psi(\cdot)$ by $\psi(\cdot-x_0)$ in the above arguments,  there is $p>\max\{1,m,\gamma\}$ such that
\begin{equation}
\label{lp-eq1}
\sup_{t\in [0,T_{\max}(u_0)),x_0\in\RR} \int_{\RR} u^p(t,x;u_0)\psi(x-x_0)dx<\infty.
\end{equation}

\noindent{\bf Step 2.} In this step, we  prove that  $ \|\nabla v(t,\cdot;u_0)\|_\infty$ stays bounded.

\smallskip

To then end,
 let  $p'=\frac{p}{\gamma}$. Then $p'>1$ and by \eqref{lp-eq1}, there holds
$$
\sup_{t\in [0,T_{\max}(u_0)),x_0\in\RR} \int_{\RR} \big(u^\gamma(t,x;u_0)\big)^{p'}\psi(x-x_0)<\infty.
$$
This together with   \cite[Theorem 9.11]{GiTr}  implies that
\begin{equation}
\label{lp-eq2}
\sup_{t\in [0,T_{\max}(u_0))} \|\nabla v(t,\cdot;u_0)\|_\infty <\infty.
\end{equation}

\noindent{\bf Step 3.} In this step, 
 we prove that  $T_{\max}(u_0)=\infty$ and $\|u(t,\cdot;u_0)\|_\infty$ stays bounded.

\smallskip

  To this end, let  $w=u\psi$. Then $w$ satisfies 
\begin{align*}
w_t&=w_{xx}-2(u\psi_x)_x+3 u\psi_{xx}-\chi (u^m v_x\psi)_x+\chi u^m v_x\psi_x+u(1-u^\alpha)\psi.
\end{align*}
It then follows that
\begin{align}
\label{lp-eq3}
w(t,\cdot)&=\underbrace{e^{(\Delta -I) t} w(0,\cdot)}_{I_1}-\underbrace{\int_0^ t e^{(\Delta -I)(t-s)}\big(2u\psi_x +\chi u^m v_x\psi\big)_x ds}_{I_2}\nonumber\\
&\quad +\underbrace{\int_0^t e^{(\Delta -I)(t-s)} \big(3 u \psi_{xx}+\chi u^m v_x\psi_x +2 u \psi -u^{1+\alpha}\psi\big)}_{I_3}.
\end{align}
It is clear that
\begin{equation}
\label{lp-eq4}
\|I_1\|_\infty\le \|w(0,\cdot)\|_\infty.
\end{equation}
Since $\alpha+1> m$,  there is $M'>0$ such that
$$
3 u \psi_{xx}+\chi u^m v_x\psi_x +2 u \psi -u^{1+\alpha}\psi\le 3\kappa_1 u \psi+\chi \kappa_1\|v_x\|_\infty u^m \psi +2 u\psi -u^{1+\alpha} \psi\le M'.
$$ 
Therefore, 
\begin{equation}
\label{lp-eq5}
I_3 \le M'\int_0^t \|e^{(\Delta-I)(t-s)}\|_\infty ds\le M'
\end{equation}
Let $p''=\frac{p}{m}$. Then $p''>1$. By Lemma \ref{prelim-semigroup-lm1}, \eqref{lp-eq1}, and \eqref{lp-eq2},   there  are $C_p''$ and $M''$  such that 
\begin{align}
\label{lp-eq6}
\|I_2\|_\infty&\le C_p'' \int_0 ^t \big(t-s\big)^{-\frac{1}{2}-\frac{1}{2p''}} e^{-\lambda (t-s)} \|2u\psi_x+\chi u^m v_x\psi\|_{L^{p''}} ds \nonumber\\
&\le 2 C_p'' \big(2\chi \kappa_1+\chi \|v_x\|_\infty\big)\int_0 ^t \big(t-s\big)^{-\frac{1}{2}-\frac{1}{2p''}} e^{-\lambda (t-s)}  \Big(\int_{\RR} (u^{p''}\psi +u^{mp''} \psi)\Big)^{1/p''} ds\nonumber\\
&\le M''.
\end{align}
It then  follows  from \eqref{lp-eq3}-\eqref{lp-eq4} that
$$ 0\le w(t,x)= u(t,x)\psi(x) \le \|w(0,\cdot)\|_\infty +M'+M''\quad \forall\, t\in [0,T_{\max}), \,\, x\in\RR.
$$
Replaying $\psi(\cdot)$ by $\psi(\cdot -x_0)$ in the above arguments yields
that
$$
\sup_{t\in (0,T_{\max}(u_0)), x_0\in\RR}\|u(t,\cdot) \psi(\cdot-x_0)\|_\infty <\infty,
$$
which implies that  $T_{\max}(u_0)=\infty$ and 
$
\limsup_{t\to \infty}\|u(t,\cdot;u_0)\|_\infty<\infty,
$ i.e., \eqref{upper-bound-eq2} holds.

In the following, we prove that \eqref{upper-limit-eq2} holds. 
First,  note that
\begin{align*}
u_t&=u_{xx}-\chi m u^{m-1} u_x v_x-\chi u^m (v-u^\gamma)+u(1-u^\alpha)\\
&\le u_{xx}-\chi m u^{m-1} u_x v_x +u (1- u^\alpha +\chi u^{m+\gamma -1}).
\end{align*}
Let $\bar u(t;u_0)$ be the solution of the following ODE:
\begin{equation*}
\begin{cases}
\frac{d \bar u}{dt}=\bar u(1-\bar u^\alpha +\chi \bar u^{m+\gamma-1})\cr
\bar u(0)=\sup_{x\in\mathbb{R}} u_0(x).
\end{cases}
\end{equation*}
By comparison principle for parabolic equations,
$$
u(t,x;u_0)\le \bar u(t;u_0)\quad \forall\, t\ge 0,\,\, x\in\mathbb{R}.
$$
It then suffices to prove that
\begin{equation}
\label{upper-limit-eq3}
\limsup_{t\to\infty} \bar u(t;u_0)\le \Big(\frac{1}{1-\chi}\Big)^{\frac{1}{\alpha}}.
\end{equation}

Assume that \eqref{upper-limit-eq3} does not hold. Then there are $0<t_1<t_2<\cdots <t_n<\cdots$ with $t_n\to\infty$  such that
$$
\lim_{n\to\infty}\bar u(t_n;u_0)>\Big(\frac{1}{1-\chi}\Big)^{\frac{1}{\alpha}}.
$$
Without loss of generality, we may assume that
\begin{equation}
\label{aaux-eq1}
\bar u(t_n;u_0)\ge \Big(\frac{1}{1-\chi}\Big)^{\frac{1}{\alpha}}+\epsilon_0\quad\forall\, n\ge 1
\end{equation}
for some $\epsilon_0>0$. 
Note that for $ u>  \Big(\frac{1}{1-\chi}\Big)^{\frac{1}{\alpha}}$,  
$$
1-u^\alpha +\chi u^{m+\gamma -1}\le 1-u^\alpha +\chi u^\alpha=1-(1-\chi) u^\alpha< 0.
$$
Hence at any $t_0$ with $\bar u(t_0;u_0)> \Big(\frac{1}{1-\chi}\Big)^{\frac{1}{\alpha}}$, 
$\bar u(t;u_0)$ is decreasing for $t$ near $t_0$. This together with \eqref{aaux-eq1} implies that
$u(t,;u_0)$ is decreasing on $[t_1,\infty)$.
Note also that 
$$
1-\bar u^\alpha (t_n;u_0)+\chi \bar u^{m+\gamma -1}(t_n;u_0)\le 1-(1-\chi) \Big[ \Big(\frac{1}{1-\chi}\Big)^{\frac{1}{\alpha}}+\epsilon_0\Big]^\alpha<0\quad \forall\, n\ge 0.
$$
It then follows that there is $\delta_0>0$ such that
$$
\bar u(t_{n+1};u_0)\le \bar u(t_n;u_0)-\delta_0\quad \forall\, n\ge 1.
$$
This implies that
$$
\lim_{n\to\infty} \bar u(t_{n+1};u_0)\le \lim_{n\to\infty} \bar u(t_n;u_0)-\delta_0,
$$
which is a contradiction. Hence \eqref{upper-limit-eq3} holds.
Proposition \ref{global-existence-prop}(2)  is thus proved.
\end{proof}

\subsection{Stabilization of positive classical solutions and proof of Proposition \ref{constant-stability-prop}}

In this subsection, we prove Proposition \ref{constant-stability-prop}.

\begin{proof}[Proof of Proposition \ref{constant-stability-prop}(1)]

Assume that $\chi\le 0$ and $\inf_{x\in\mathbb{R}} u_0(x)>0$.
 Put
$$
\bar u(t)=\sup_{x\in\RR} u(t,x;u_0)\quad {\rm and}\quad \underline u(t)=\inf_{x\in\RR} u(t,x;u_0).
$$
By the similar  arguments of \eqref{upper-limit-eq1},  we can prove that,  if $0<\underline u(t_0)<1$ for some $t_0>0$, then
$\underline u(t)$ is non-decreasing on $[0,t_0)$, and that
$$
\liminf_{t\to\infty} \underline u(t)\ge 1.
$$
This together with \eqref{upper-limit-eq1} implies that \eqref{constant-stability-eq} holds.
Proposition \ref{constant-stability-prop}(1) is thus proved.
\end{proof}

\begin{proof}[Proof of Proposition \ref{constant-stability-prop}(2)]
Assume that  $0<\chi<\frac{1}{2}$,  $\alpha\ge m+\gamma-1$, and  $\inf_{x\in\mathbb{R}}u_0(x)>0$.
We adopt rectangle idea developed in \cite{GaSaTe}
to prove that  \eqref{constant-stability-eq} holds.
This idea can be described  as follows: Fix an $\epsilon\in (0,\min\{\inf_{x\in\RR} u_0(x),1\})$. 
 Consider the following system of ODEs,
$$
\begin{cases}
\overline U_t=\chi \overline U^m (\overline U^\gamma -\underline U^\gamma)+ \overline U(1-\overline U^\alpha)\cr
\underline U_t=\chi \underline U^m (\underline U^\gamma-\overline U^\gamma)+\underline U(1-\underline U^\alpha)\cr
\overline U(0)=\max\{\|u_0\|_\infty, 1\}+\epsilon,\quad \underline U(0)=\min\{\inf_{x\in\RR}u_0(x),1\}-\epsilon.
\end{cases}
$$
Let
$$
\underline u(t):=\inf_{x\in\mathbb{R} }u(t,x;u_0),\quad \bar u(t):=\sup_{x\in\mathbb{R} }u(t,x;u_0).
$$
We prove that
\begin{equation}
\label{ineq-eq1}
0<\underline U(t)<\underline u(t)\le \bar u(t)<  \overline U(t)\quad \forall\, t\in [0,\infty),
\end{equation}
and
\begin{equation}
\label{ineq-eq2}
\lim_{t\to\infty} \overline U(t)=\lim_{t\to\infty}\underline U(t)=1.
\end{equation}

First, by the arguments in \cite[Lemmas 3.1, 3.2]{GaSaTe}, $(\underline U(t),\overline U(t))$ exists for all $t\ge 0$,
$$
0<\underline U(t)<1<\overline U(t),\quad \forall\, t\ge 0,
$$
and
$$
\lim_{t\to\infty}\underline U(t)=\lim_{t\to\infty }\overline U(t)=1.
$$
Hence \eqref{ineq-eq2} holds.

Next, we prove \eqref{ineq-eq1} holds. Assume by contradiction that \eqref{ineq-eq1} 
does not hold. Since \eqref{ineq-eq1} holds for $0\le t\ll 1$. Then there is $t_0>0$ such that
\eqref{ineq-eq1} holds for $0\le t<t_0$, and there is $\{x_n\}\subset\RR$ such that
$$
\lim_{n\to\infty} u(t_0,x_n;u_0)=\underline U(t_0)\quad {\rm or}\quad \lim_{n\to\infty} u(t,x_n;u_0)=\overline U(t_0).
$$
Assume that $\lim_{n\to\infty} u(t_0,x_n;u_0)=\underline U(t_0)$.  Without loss of generality, we may assume that
$$
u^*(t,x):=\lim_{n\to\infty}u(t,x+x_n;u_0), \quad v^*(t,x):=\lim_{n\to\infty} v(t,x+x_n;u_0)
$$
exist locally uniformly in $(t,x)\in (0,\infty)\times \RR$, moreover, $u^*(t,x)$ satisfies
\begin{align*}
u^*_t&=u^*_{xx}-\chi m (u^*)^{m-1} u^*_x v^*_x+u^* \Big(1-(u^*)^\alpha +\chi\big( (u^*)^{m+\gamma -1}-(u^*)^{m-1} v^*\big)\Big)\\
&\ge   u^*_{xx}-\chi m (u^*)^{m-1} u^*_x v^*_x +\chi ((u^*)^{m+\gamma}- (u^*)^m \overline U^{\gamma}(t))+ u^*(1-(u^*)^\alpha),\quad 0<t\le t_0,\quad x\in\RR.
\end{align*}
Note that  $w=\underline U(t)$ satisfies
$$
w_t=w_{xx}-\chi m w^{m-1} w_{x} v^*_x + \chi(w^{m+\gamma}-w^m \overline U^{\gamma}(t))+w(1-w^{\alpha})
$$
and
$$
\inf_{x\in\RR} u^*(t_0/2,x )>\underline U(t_0/2).
$$
Then by strong comparison principle for parabolic equations, we have
$$
u^*(t,x)>\underline U(t)\quad \forall\, t\in [t_0/2,t_0],\,\, x\in\RR.
$$
This implies that
$$
\underline U(t_0)< u^*(t_0,0)=\lim_{n\to\infty} u(t_0,x_n;u_0)=\underline U(t_0),
$$
which is a contradiction.
Hence, it is not possible that $\lim_{n\to\infty} u(t_0,x_n)=\underline U(t_0)$.

Similarly, it is not possible $\lim_{n\to\infty} u(t_0,x_n)=\overline U(t_0)$. Then we must have \eqref{ineq-eq1} holds. By \eqref{ineq-eq1} and \eqref{ineq-eq2}, we have
$$
\lim_{t\to\infty}\|u(t,\cdot;u_0)-1\|_\infty=0.
$$
 Proposition \ref{constant-stability-prop}(2) is thus proved.
\end{proof}

\section{Existence of traveling wave solutions and proof of Theorem \ref{monotone-thm}}
\label{existence-section}

In this section, we study the existence of traveling wave solutions of \eqref{main-eq} connecting
$(1,1)$ and $(0,0)$, and prove Theorem \ref{monotone-thm}.

To study traveling wave solutions of \eqref{main-eq} with speed $c$, we introduce the moving coordinate
$(t,x-ct)$. Note that in this moving coordinate, \eqref{main-eq} becomes
\begin{equation}
\label{main-moving-eq1}
\begin{cases}
u_t=u_{xx}+cu_x-\chi(u^m v_x)_x +u(1-u^\alpha),\quad &x\in\RR\cr
0=v_{xx}-v+u^\gamma,\quad&x\in\RR.
\end{cases}
\end{equation}
It is clear that if $(u,v)=(\phi(x),\psi(x))$ is a stationary solution of \eqref{main-moving-eq1} connecting $(1,1)$ and $(0,0)$, then
$(u,v)=(\phi(x-ct),\psi(x-ct))$ is a traveling wave solution of \eqref{main-eq} connection $(1,1)$ and $(0,0)$ with speed $c$.

 In the following, we first present some super- and sub-solutions  of  some equations relevant to \eqref{main-moving-eq1} in Subsection 3.1, and then prove Theorem \ref{monotone-thm}(1) and (2) in Subsections
3.2 and 3.3, respectively.

\subsection{Super- and sub-solutions}

In this subsection, we present  some  super- and sub-solutions of  some equations relevant to \eqref{main-moving-eq1},  which will be used to construct traveling wave solutions in next two subsections.

We first present some  super-solutions of  some equations relevant to \eqref{main-moving-eq1}. To this end, we first introduce
 some notations.

For given  $0<\kappa<1$, $M\ge 1$, and $T>0$, let
$$
U_{\kappa, M}^+=\min\{M,e^{-\kappa x}\},
$$
and
$$
\mathcal{E}_{\kappa,M,T}=\{u\in C\big([0,T], C_{\rm unif}^b(\RR)\big)\,|\, 0\le u(t,x)\le U_{\kappa,M}^+(x)\quad \forall\, t\in [0,T],\,\, x\in\RR\}.
$$
For given  $u\in\mathcal{E}_{\kappa,M,T}$, let 
$$V(t,x)=V(t,x;u):=\Psi(x; u^\gamma(t,\cdot), 1,1),
$$
where $\Psi$ is as in \eqref{psi-eq1} with $\lambda=\mu=1$,
and define
$$
\mathcal{A}(W;u)=W_{xx}+cW_x-\chi m W^{m-1}  V_x W_x + W\Big(1-\chi W^{m-1} V -\big(W^\alpha -\chi  W^{m+\gamma-1}\big)\Big)
$$
for $W\in C^2(\RR)$ with $W\ge 0$, where
$$
c=\kappa+\frac{1}{\kappa}.
$$


\begin{lemma}[Super-solutions]
\label{prelim-super-lm} 
Assume that one of the following sets of conditions is satisfied:
\begin{itemize}
\item[(i)]  $\alpha\le m+\gamma -1$,  $\chi\le 0$, 
   $0<\kappa <1$ is such that $c:=\kappa+\frac{1}{\kappa}>c_{\chi,m,\gamma}^*$,
 and
$1\le M\le M_{\chi,\kappa,m,\gamma}:=\frac{1}{\kappa \sqrt {\gamma^2+\gamma^2|\chi|+m\gamma |\chi|}}(>1)$, where $c^*_{\chi,m,\gamma}$ is as in  \eqref{cond-on-c-eq1}.
  
\item[(ii)] $\alpha=m+\gamma-1$,    $0\le \chi<\chi^*(m,\alpha,\gamma)$,  
   $0<\kappa <1$,  and   $M\ge M_\chi:=\big( \frac{1}{1-\chi}\big)^{1/\alpha}$, where $\chi^*(m,\alpha,\gamma)$ is as in \eqref{chi-star-eq}.
\end{itemize}
Then 
for any $u\in \mathcal{E}_{\kappa,M,T}$, the following hold:
\begin{itemize}
\item[(1)]  $W= e^{-\kappa x}$ is a super-solution of
$$
W_t=\mathcal{A}(W;u), \quad t\in (0,T),\, \,\, x\ge -\frac{\ln M}{\kappa}.
$$

\item[(2)]   $W=M$ is a super-solution of
$$
W_t=\mathcal{A}(W;u), \quad t\in (0,T),\,\,\, x\in\RR.
$$
\end{itemize}
\end{lemma}

\begin{proof} The Lemma can be proved by properly modified similar arguments of \cite[Theorem 2.1(2)]{SaSh2020}.
For reader's convenience and for seeing the effects of the parameters $m,\alpha,\gamma$ , we provide a proof in the following.
We divide the proof into two cases.

\smallskip

\noindent {\bf Case 1.} Assume that the condition (i) holds, i.e.,  $\chi\le 0$,  $\alpha\le m+\gamma -1$,   $0<\kappa <1$ is such that $c=\kappa+\frac{1}{\kappa}$ satisfies \eqref{cond-on-c-eq1},
and   $1\le M\le M_{\chi,\kappa,m,\gamma}$.  
Note that, by the assumption $c=\kappa+\frac{1}{\kappa}>c^*_{\chi,m,\gamma}$,  there holds 
\begin{equation}
\label{aux-M-eq1}
\kappa< \frac{1}{\sqrt {\gamma^2+\gamma^2|\chi|+m\gamma |\chi|}}\,\,\Longrightarrow \, M_{\kappa,\chi,m,\gamma}>1.
\end{equation}

\smallskip

In this case,
$$
\mathcal{A}(W;u)=W_{xx}+cW_x+m W^{m-1} |\chi| V_x W_x + W\big(1+|\chi|W^{m-1} V -(W^\alpha +|\chi| W^{m+\gamma-1})\big),
$$
where $V(t,x)=\Psi(x;u^\gamma (t,\cdot),1,1)$.

(1) 
It suffices to prove that
$$
\mathcal{A}(e^{-\kappa x};u)\le 0\quad \forall\,  x\ge -\frac{\ln M}{\kappa}.
$$

First, note that
\begin{align}
\label{super-solu-eq1}
&\mathcal{A}(e^{-\kappa x};u)\nonumber\\
&=  (e^{-\kappa x})_{xx}+c (e^{-\kappa x})_x + m |\chi|(e^{-\kappa x})^{m-1}  (e^{-\kappa x})_xV_x \nonumber\\
&\quad + e^{-\kappa x}\Big(
1+ |\chi| (e^{-\kappa x})^{m-1} V-\big(( e^{- \kappa x})^\alpha+|\chi| (e^{-\kappa x})^{m+\gamma-1}\big)\Big)\nonumber\\
&=e^{-\kappa x}\Big(-\kappa |\chi| m  (e^{-\kappa x})^{m-1} V_x + |\chi|(e^{-\kappa x})^{m-1} V -\big(( e^{- \kappa x})^\alpha+|\chi| (e^{-\kappa x})^{m+\gamma-1}\big)\Big).
\end{align}

Next, by Lemma \ref{prelim-v-lm1}, we have
\begin{align*}
&-\kappa m |\chi| V_x+|\chi|V\\
&= -\kappa |\chi|\left[ -\frac{m }{2}e^{- x}\int_{-\infty}^xe^{ y}u^\gamma (y)dy + \frac{m }{2}e^{x}\int_{x}^{\infty}e^{- y}u^\gamma (y)dy\right]+ \frac{|\chi| }{2}\int_{\R}e^{- |x-y|}u^\gamma (y)dy
\\
&=\frac{|\chi|}{2}(m\kappa +1)e^{-x}\int_{-\infty}^ x e^y u^\gamma (y)dy +\frac{|\chi|}{2}(1-m\kappa)e^{x}\int_x^\infty e^{-y} u^\gamma (y)dy.
\end{align*}
Note that $u\le e^{-\kappa x}$ and  $c=\kappa+\frac{1}{\kappa}$. By the assumption $c>c^*_{\chi,m,\gamma}$,  we have  $\kappa \gamma <1$ and $\kappa m<1$.
It then  holds that
\begin{align}
\label{super-solu-eq3}
-\kappa m |\chi| V_x+|\chi|V&\le   \frac{|\chi|}{2}(1+m\kappa )e^{-x}\int_{-\infty}^ x e^y e^{-\kappa \gamma y} dy +  \frac{|\chi|}{2}(1-m\kappa)e^{x}\int_x^\infty e^{-y} e^{-\kappa \gamma y}dy\nonumber\\
&=\frac{|\chi|}{2}\frac{1+m \kappa}{1-\kappa\gamma } e^{-\gamma \kappa x}+\frac{ |\chi|}{2}\frac{1-m \kappa}{1+\gamma \kappa} e^{-\kappa\gamma  x}\nonumber\\
&= |\chi| \frac{1+m\gamma \kappa^2}{1-\gamma^2 \kappa^2} e^{-\kappa\gamma  x}.
\end{align}

\smallskip

Now, by \eqref{super-solu-eq1} and \eqref{super-solu-eq3},  we have
\begin{align}
\label{super-solu-eq5}
\mathcal{A} (e^{-\kappa x};u)
\le  \Big (    |\chi| \frac{1+m\gamma \kappa^2}{1-\gamma^2 \kappa^2} e^{-(m+\gamma-1)\kappa  x} -\big( e^{- \alpha \kappa x} +|\chi| e^{-(m+\gamma -1)\kappa x}\big)\Big) e^{-\kappa x}.
\end{align}
By \eqref{super-solu-eq5} and $\alpha\le m+\gamma-1$, we have 
\begin{align}
\label{super-solu-eq7}
\mathcal{A} (e^{-\kappa x};u)&\le   \Big [   |\chi|\Big( \frac{1+m\gamma \kappa^2}{1-\gamma^2 \kappa^2} -1\Big)\Big( e^{-\kappa x}\Big)^{m+\gamma -1-\alpha} -1\Big] e^{-(\alpha+1)\kappa x}  \nonumber\\
&\le  e^{-(1+\alpha )\kappa x}
\Big( |\chi| \frac{1+m \gamma \kappa^2}{1-\gamma^2 \kappa ^2} M -(1+|\chi| M)\Big)\,\,\, \quad {\rm if}\,\, x\ge -\frac{\ln M}{\kappa}.
\end{align}

Finally, by \eqref{aux-M-eq1}, there holds 
$$
\kappa\sqrt M\le \kappa M \le   \frac{1}{\sqrt {\gamma^2+\gamma^2|\chi|+m\gamma |\chi|}}\,\Longrightarrow\, (\gamma^2+\gamma^2|\chi|+m\gamma|\chi|)\kappa^2 M  \le 1.
$$
This implies that
\begin{align*}
(\gamma^2+\gamma^2|\chi|M)\kappa^2+m\gamma|\chi|\kappa^2 M  +|\chi| M\le 1+|\chi| M\, \Longrightarrow\,
|\chi|\big(1+m\gamma\kappa^2\big)M\le \big(1-\gamma^2\kappa^2\big)(1+|\chi|M).
\end{align*}
Hence
$$
 |\chi| \frac{1+m \gamma \kappa^2}{1-\gamma^2 \kappa ^2} M -(1+|\chi| M)\le 0.
$$
It then follows from  \eqref{super-solu-eq7} that
$$
\mathcal{A}(W;u)(x)\le 0\quad \forall\, x\ge -\frac{\ln M}{\kappa},
$$
and (1) is thus proved.

\smallskip

(2)  Note that $u\le M$. Hence 
$$
V\le M^\gamma,
$$
and
\begin{align*}
\mathcal{A}(M;u)&=M \big(1+|\chi| M^{m-1} V -(M^\alpha +|\chi| M^{m+\gamma-1})\big)\\
&\le M \big(1+|\chi| M^{m-1} M^\gamma  -(M^\alpha +|\chi| M^{m+\gamma-1})\big)\\
&\le 0.
\end{align*}
This completes the proof of (2).

\smallskip

\noindent {\bf Case 2.} Assume that the condition (ii) holds, i.e., $0\le \chi<\chi^*(m,\alpha,\gamma)$,  $\alpha=  m+\gamma-1$, 
$M\ge \Big(\frac{1}{1-\chi}\Big)^{1/\alpha}$, and  $0<\kappa <1$.

\smallskip

In this case,
$$
\mathcal{A}(W;u)=W_{xx}+cW_x- m W^{m-1} |\chi| V_x W_x + W\big(1- |\chi|W^{m-1} V -(W^\alpha -|\chi| W^{m+\gamma-1})\big),
$$
where $c=\kappa+\frac{1}{\kappa}$.

(1)  
Note that
\begin{align*}
\mathcal{A}( e^{-\kappa x};u)
&=  (e^{-\kappa x})_{xx}+c (e^{-\kappa x})_x - m |\chi|(e^{-\kappa x})^{m-1}  (e^{-\kappa x})_xV_x \\
&\quad + e^{-\kappa x}\Big(
1-  |\chi| (e^{-\kappa x})^{m-1} V-\big(( e^{- \kappa x})^\alpha-|\chi| (e^{-\kappa x})^{m+\gamma-1}\big)\Big)\\
&= \Big(\kappa |\chi| m  (e^{-\kappa x})^{m-1} V_x- |\chi|(e^{-\kappa x})^{m-1} V -\big(( e^{- \kappa x})^\alpha-|\chi| (e^{-\kappa x})^{m+\gamma-1}\big)\Big)e^{-\kappa x}.
\end{align*}

By $0\le u\le e^{-\kappa x}$ and Lemma \ref{prelim-v-lm1}, we have 
\begin{align*}
&\kappa m |\chi| V_x-|\chi|V\\
&= \kappa |\chi|\left[ -\frac{m }{2}e^{- x}\int_{-\infty}^xe^{ y}u^\gamma (y)dy + \frac{m }{2}e^{x}\int_{x}^{\infty}e^{- y}u^\gamma (y)dy\right]-|\chi| \frac{1 }{2}\int_{\R}e^{- |x-y|}u^\gamma (y)dy
\\
&=-\frac{|\chi|}{2}(m\kappa +1)e^{-x}\int_{-\infty}^ x e^y u^\gamma (y)dy +\frac{|\chi|}{2}(m\kappa-1)e^{x}\int_x^\infty e^{-y} u^\gamma (y)dy\\
&\le \frac{|\chi|}{2}\frac{(m\kappa -1)_+}{1+\kappa \gamma} e^{-\gamma \kappa x}.
\end{align*}
This  together with $\alpha= m+\gamma-1$ implies that
\begin{align*}
\mathcal{A} (e^{-\kappa x};u)&\le  \left (\frac{|\chi|}{2}\frac{(m\kappa -1)_+}{1+\kappa \gamma} e^{-(\gamma+m-1) \kappa x}  -\big( e^{- \alpha \kappa x}-|\chi| e^{-(m+\gamma -1)\kappa x}\big)\right) e^{-\kappa x}\\
&=  e^{-(1+\alpha )\kappa x}\Big[
\Big( \frac{|\chi|}{2} \frac{(m\kappa -1)_+}{1+\kappa \gamma}+|\chi|\Big)  -1\Big].
\end{align*}
Note that
\begin{align*}
\Big( \frac{|\chi|}{2} \frac{(m\kappa -1)_+}{1+\kappa \gamma}+|\chi|\Big) -1\le\begin{cases} \chi-1\quad & {\rm if}\,\,\kappa\le 1/m\cr
 \chi \Big(\frac{m-1}{2\big(1+\frac{\gamma}{m}\big)}+1\Big)-1\,\, & {\rm if}\,\, \kappa>1/m.
\end{cases}
\end{align*}
This together with $\chi<\chi^*(m,\alpha,\gamma)$ implies that
$$
\mathcal{A} (e^{-\kappa x};u)\le 0\quad \forall\,x\ge  -\frac{\ln M}{\kappa}.
$$
This completes the proof of (1).
\smallskip

(2)  Note that
\begin{align*}
\mathcal{A}(M;u)&= M\big(1- |\chi| M ^{m-1} V -(M^\alpha -|\chi| M^{m+\gamma-1})\big)\\
&\le M \big(1-(M^\alpha -|\chi| M^{m+\gamma-1})\big)\\
&=M\big(-(1-|\chi|) M^\alpha \qquad\qquad \text{(by}\,\, \alpha=m+\gamma-1)\\
&\le 0 \qquad\qquad\qquad \qquad\qquad \quad \big(\text{since}\,\, M\ge \big(\frac{1}{1-\chi}\big)^{\frac{1}{\alpha}}\big).
\end{align*}
This proves (2).
\end{proof}

In the following, we present some subsolutions of some equations relevant to \eqref{main-moving-eq1}. 
We first introduce some notations.

For given $0<\kappa<\tilde \kappa\le 1$ and  $D>0$, let 
$$
U_{\kappa,\tilde\kappa, D}(x)= e^{-\kappa x}-De^{-\tilde{\kappa}x}.
$$
Let $x^\pm _{\kappa,\tilde\kappa,D}>0$ be such that
$$
U_{\kappa,\tilde\kappa,D}(x^-_{\kappa,\tilde \kappa, D})=0.
$$
$$
U_{\kappa,\tilde \kappa,D}(x^+_{\kappa,\tilde\kappa,D})=\max_{x\in\RR} U_{\kappa,\tilde\kappa, D}(x).
$$
Observe that
$$
x^-_{\kappa,\tilde\kappa,D}=\frac{1}{\tilde\kappa -\kappa}\ln D,
$$
and
$$
x^+_{\kappa,\tilde\kappa,D}=\frac{1}{\tilde\kappa -\kappa}\ln \frac{\tilde\kappa D}{\kappa}.
$$

Let 
$$
U^-_{\kappa,\tilde\kappa,D}(x)=\begin{cases}  U_{\kappa,\tilde\kappa,D}(x^+_{\kappa,\tilde\kappa,D}) \quad &{\rm if}\quad x\le x^+_{\kappa,\tilde \kappa,D}\cr
U_{\kappa,\tilde\kappa,D} (x)\quad &{\rm if}\quad x>x^+_{\kappa,\tilde\kappa,D}.
\end{cases}
$$

Furthermore,
for given $M\ge 1$, let 
$$
K_{M, \kappa,\tilde\kappa,m,\gamma}=\begin{cases}
\big(m(\tilde k+k)+1\big) \big(M^\gamma+\frac{3}{4}\big) ,\quad &{\rm if}\quad \gamma\kappa =1\cr\cr
 \big(m (\tilde \kappa+\kappa)+1\big)\cdot \frac{1}{1-\gamma^2 \kappa^2} \quad &{\rm if}\quad \gamma\kappa<1\cr\cr
 \big(m(\tilde\kappa +\kappa)+1\big)\cdot \frac{M^\gamma(\kappa^2\gamma^2-1)+\gamma\kappa}{ (\kappa^2\gamma^2-1)} ,\quad &{\rm if}\quad \gamma \kappa>1,
\end{cases}
$$
and
$$
D_{M,\kappa,\tilde\kappa,\chi,m,\gamma}:=\frac{1+|\chi|K_{\kappa,\tilde\kappa,m,\gamma}}{c\tilde \kappa-\tilde\kappa^2-1},
$$
where $c=\kappa+\frac{1}{\kappa}$.

\begin{lemma}[Sub-solutions]
\label{prelim-sub-lm}
Assume that  $0<\kappa<1$,   $\tilde \kappa$ is such that $0<\kappa<\tilde \kappa \le \min\{ (1+\alpha)\kappa,  m\kappa +1/2,  1\}$, and $M\ge 1$.
Then 
for any $u\in \mathcal{E}_{\kappa,M,T}$, the following hold:
\begin{itemize}
\item[(1)] For $D> D_{M,\kappa,\tilde\kappa,\chi,m,\gamma}$, 
{$W=U_{\kappa,\tilde\kappa, D}(x)$}  is a sub-solution of
$$
W_t= \mathcal{A}(W;u)\quad  t\in (0,T),\quad t\in (0,T),\,\,  x>x^-_{\kappa,\tilde\kappa,D}.
$$

\item[(2)]   For $0<d\le  d_{\kappa,\tilde \kappa,D,\chi}:=\min\Big\{\frac{1}{1+|\chi|}, \Big(\frac{\kappa}{\tilde\kappa D}\Big)^{\frac{\kappa}{\tilde \kappa -\kappa}}\Big(1-\frac{\kappa}{\tilde \kappa}\Big)\Big\}$, $W=d$ is a sub-solution of
$$
W_t= \mathcal{A}(W;u),\quad t\in (0,T),\,\,  x\in\RR.
$$ 
\end{itemize}
\end{lemma}

\begin{proof} The Lemma can be proved by properly modified  arguments of \cite[Theorem 2.1(3)]{SaSh2020}.
For self-completeness and reader's convenience, we provide a proof in the following.

First of all,  by Lemma \ref{prelim-v-lm1},  for  {$x\ge 0$}, we have
\begin{align*}
V(t,x)=\frac{1 }{2}e^{- x}\int_{-\infty}^0 e^{ y}u^\gamma (t,y)dy
+\frac{1}{2} e^{-x}\int_0^ x e^y u^\gamma(t,y)dy + \frac{1}{2}e^{x}\int_{x}^{\infty}e^{- y}u^\gamma (t,y)dy.
\end{align*}
This implies that for $x\ge 0$, 
\begin{equation*}
V(t,x)
\le  \begin{cases}
\Big(\frac{M^\gamma}{2}+\frac{x}{2}+\frac{1}{4}\Big) e^{-x},\quad &{\rm if}\quad  \gamma \kappa=1\cr
 \frac{1}{1-\gamma^2\kappa^2}  e^{-\gamma \kappa x},\quad &{\rm if}\quad \gamma\kappa<1\cr
\Big(\frac{M^\gamma}{2}+\frac{\gamma\kappa}{\kappa^2\gamma^2-1}\Big) e^{-x},\quad &{\rm if}\quad \gamma \kappa>1.
\end{cases}
\end{equation*}
Note that
$$
\Big(\frac{M^\gamma}{2}+\frac{x}{2}+\frac{1}{4}\Big) e^{-x/2}\le  \frac{M^\gamma}{2}+\frac{3}{4}\quad \forall\, x\ge 0.
$$
Hence, for $x\ge 0$
\begin{equation}
\label{V-estimate-eq1}
V(t,x)
\le  \begin{cases}
\Big(M^\gamma +\frac{3}{4}\Big) e^{-x/2},\quad &{\rm if}\quad  \gamma \kappa=1\cr
 \frac{1}{1-\gamma^2\kappa^2}  e^{-\gamma \kappa x},\quad &{\rm if}\quad \gamma\kappa<1\cr
\Big(\frac{M^\gamma(\kappa^2\gamma^2-1)+\gamma\kappa}{ (\kappa^2\gamma^2-1)}\Big) e^{-x},\quad &{\rm if}\quad  \gamma \kappa>1.
\end{cases}
\end{equation}
This together with Lemma \ref{prelim-v-lm2} implies that for $x\ge 0$,
\begin{equation}
\label{V-estimate-eq2}
-|V_x(t,x)|
\ge  \begin{cases}
-\Big(M^\gamma +\frac{3}{4}\Big) e^{-x/2},\quad &{\rm if}\,\, \gamma \kappa=1\cr
- \frac{1}{1-\gamma^2\kappa^2}  e^{-\gamma \kappa x},\quad &{\rm if}\quad \gamma\kappa<1\cr
-\Big(\frac{M^\gamma (\kappa^2\gamma^2-1) +\gamma\kappa}{(\kappa^2\gamma^2-1)}\Big) e^{-x},\quad &{\rm if}\,\, \gamma \kappa>1.
\end{cases}
\end{equation}
We divide the rest of the proof into two cases.

\smallskip

\noindent{\bf Case 1.} $\chi\le 0$.

\smallskip

In this case, 
$$
\mathcal{A}(W;u)=W_{xx}+cW_x+m W^{m-1} |\chi| V_x W_x + W\big(1+|\chi|W^{m-1} V -(W^\alpha +|\chi| W^{m+\gamma-1})\big),
$$
where $V(t,x)=\Psi(x;u^\gamma (t,\cdot),1,1)$

\smallskip

(1) It suffices to prove that 
$$
\mathcal{A}(U_{\kappa,\tilde\kappa,D};u)   \ge 0\quad \forall\, x\ge x_{\kappa,\tilde\kappa,D}^-.
$$
Note that  $x_{\kappa,\tilde\kappa,D}^->0$ and 
\begin{align*}
\mathcal{A}(U_{\kappa,\tilde\kappa,D};u)   
 &= (e^{-\kappa x}-D e^{-\tilde \kappa x})_{xx}+c (e^{-\kappa x}-D e^{-\tilde\kappa x})_x+ |\chi|m ( e^{-\kappa x}-De^{-\tilde\kappa x})^{m-1} V_x (e^{-\kappa x}-D e^{-\tilde \kappa x})_x \\
&\quad +( e^{-\kappa x}-De^{-\tilde \kappa x}) (1+|\chi| ( e^{-\kappa x}-De^{-\tilde\kappa x})^{m-1}  V)\\
&\quad -(e^{-\kappa x}-De^{-\tilde\kappa x})\Big ((e^{- \kappa x}-D e^{-\tilde \kappa x})^\alpha +|\chi|  (e^{-\kappa x}-D e^{-\tilde \kappa x})^{m+\gamma -1}\Big ) \\
&\ge { D\Big(c\tilde\kappa-{\tilde\kappa^2}-1\Big) e^{-\tilde \kappa x} }+{ |\chi| m (e^{-\kappa x}-De^{-\tilde \kappa x})^{m-1} \Big(\tilde\kappa D e^{-\tilde\kappa x}-\kappa e^{-\kappa x}\Big) V_x }\\
&\quad -{ (e^{-\kappa x}-De^{-\tilde\kappa x})\Big ((e^{- \kappa x}-D e^{-\tilde \kappa x})^\alpha +|\chi|  (e^{-\kappa x}-D e^{-\tilde \kappa x})^{m+\gamma -1}\Big ) }.
\end{align*}
{By \eqref{V-estimate-eq2},  for  $x\ge x^-_{\kappa,\tilde\kappa, D}$,}  we have
\begin{align*}
&{ |\chi| m (e^{-\kappa x}-De^{-\tilde \kappa x})^{m-1} \Big(\tilde\kappa D e^{-\tilde\kappa x}-\kappa e^{-\kappa x}\Big) V_x }\\
&\ge -  |\chi| m e^{-(m-1)\kappa x} e^{-\kappa x} \Big(\tilde\kappa D e^{-(\tilde\kappa-\kappa) x}+\kappa \Big) |V_x|\\
&\ge \begin{cases}
-|\chi| m(\tilde k+k) \big(\frac{M^\gamma}{2}+\frac{3}{4}\big) e^{-(m\kappa +1/2)x},\quad &{\rm if}\quad \gamma\kappa =1\cr
-\frac{|\chi| m (\tilde \kappa+\kappa)}{1-\gamma^2 \kappa^2} e^{-(m+\gamma)\kappa x}\quad &{\rm if}\quad \gamma\kappa<1\cr
-{|\chi| m(\tilde\kappa +\kappa)\Big(\frac{M^\gamma(\kappa^2\gamma^2-1)+\gamma\kappa}{ 2(\kappa^2\gamma^2-1)}\Big)} e^{-(m\kappa +1)x},\quad &{\rm if}\quad \gamma \kappa>1.
\end{cases}
\end{align*}
It then follows that
\begin{align*}
\mathcal{A} (U_{\kappa,\tilde\kappa, D};u)
&\ge \left( D\big( c\tilde\kappa-\tilde\kappa^2 -1\big) -1-|\chi| K_{M,\kappa,\tilde\kappa,m,\gamma} \right)  e^{-\tilde \kappa x}\quad \forall\, x>x_{\kappa,\tilde\kappa,D}^-.
\end{align*}
Note that 
for $\kappa<\tilde\kappa\le 1$, 
$$
c\tilde\kappa -\tilde \kappa^2 -1>0.
$$ 
Hence
\begin{align*}
\mathcal{A}(\mathcal{U}_{\kappa,\tilde\kappa,D};u)(x)&\ge 0\quad \forall\, x\ge x_{\kappa,\tilde\kappa,D}^-,\,\, D\ge D_{M,\kappa,\tilde\kappa,\chi,m,\gamma}.
\end{align*}

(2)  It follows directly.

\smallskip

\noindent {\bf Case 2.} $\chi>0$.

\smallskip

In this case,
$$
\mathcal{A}(W;u)=W_{xx}+cW_x- m W^{m-1} |\chi| V_x W_x + W\big(1- |\chi|W^{m-1} V -(W^\alpha -|\chi| W^{m+\gamma-1})\big),
$$

\smallskip

(1) Note that 
\begin{align*}
\mathcal{A}(U_{\kappa,\tilde\kappa,D};u)   
&={ D\Big(c\tilde\kappa-{\tilde\kappa^2}-1\Big) e^{-\tilde \kappa x} } -{ |\chi| m (e^{-\kappa x}-De^{-\tilde \kappa x})^{m-1} \Big(\tilde\kappa D e^{-\tilde\kappa x}-\kappa e^{-\kappa x}\Big) V_x }\\
&\quad -{ |\chi|  V\big(e^{-\kappa x}-D e^{-\tilde \kappa  x}\big)^m}\\
&\quad -{ (e^{-\kappa x}-De^{-\tilde\kappa x})\Big ((e^{- \kappa x}-D e^{-\tilde \kappa x})^\alpha -|\chi|  (e^{-\kappa x}-D e^{-\tilde \kappa x})^{m+\gamma -1}\Big ) }\\
&\ge { D\Big(c\tilde\kappa-{\tilde\kappa^2}-1\Big) e^{-\tilde \kappa x} } -{ |\chi| m (e^{-\kappa x}-De^{-\tilde \kappa x})^{m-1} \Big(\tilde\kappa D e^{-\tilde\kappa x}-\kappa e^{-\kappa x}\Big) V_x }\\
&\quad -{ |\chi|  V\big(e^{-\kappa x}-D e^{-\tilde \kappa  x}\big)^m} -{ (e^{-\kappa x}-De^{-\tilde\kappa x})\big(e^{- \kappa x}-D e^{-\tilde \kappa x}\big)^\alpha }.
\end{align*}
By \eqref{V-estimate-eq1} and \eqref{V-estimate-eq2},
we have
\begin{align*}
\mathcal{A} (U_{\kappa,\tilde\kappa, D};u)
&\ge \left( D\big( c\tilde\kappa-\tilde\kappa^2 -1\big) -1-|\chi| K_{M,\kappa,\tilde\kappa,m,\gamma} \right)  e^{-\tilde \kappa x}\\
&\ge 0\quad \forall\, x\ge x_{\kappa,\tilde\kappa,D}^-,\,\, D\ge D_{M,\kappa,\tilde\kappa,\chi,m,\gamma}.
\end{align*}

(2) It follows directly.
\end{proof}

\begin{remark}
\label{prelim-sub-rk1}
Suppose that $0<\kappa <\frac{1}{2}$, $\gamma \kappa<1$,  and $\tilde\kappa =2\kappa$.  Then
$$
\tilde c \tilde \kappa-\tilde\kappa^2-1=2c\kappa -4\kappa^2-1=2\big(\kappa+\frac{1}{\kappa}\big)\kappa -4\kappa^2 -1=1-2\kappa^2,
$$
and for any $M\ge 1$,
\begin{equation}\label{K-eq}
K_{M,\kappa,\tilde\kappa,m,\gamma}=K_{\kappa,m,\gamma}:=\frac{3 m\kappa +1}{1-\gamma^2\kappa^2},
\end{equation}
and
\begin{equation}
\label{D-eq}
D_{M,\kappa,\tilde\kappa,\chi,m,\gamma}=\frac{1-\gamma^2\kappa^2 +|\chi|\big(3m\kappa+1\big)}{(1-2\kappa^2)(1-\gamma^2\kappa^2)}\le D_{\kappa,\chi,m,\gamma}:= \frac{2\Big(1-\gamma^2\kappa^2 +|\chi|\big(3m\kappa+1\big)\Big)}{1-\gamma^2\kappa^2}.
\end{equation}
We also have
\begin{align}
\label{d-eq}
d_{\kappa,\chi,m,\gamma}&:=\min\Big\{\frac{1}{1+|\chi|}, \Big(\frac{\kappa}{\tilde\kappa D_{\kappa,\chi,m,\gamma}}\Big)^{\frac{\kappa}{\tilde \kappa -\kappa}}\Big(1-\frac{\kappa}{\tilde \kappa}\Big)\Big\}\nonumber\\
&=\min\Big\{\frac{1}{1+|\chi|},\frac{1}{4} \frac{1}{D_{\kappa,\chi,m,\gamma}}\Big\}\nonumber\\
&= \frac{1}{8}   \frac{1-\gamma^2\kappa^2} {1-\gamma^2\kappa^2 +|\chi|\big(3m\kappa+1\big)}.
\end{align}
Assume {$\kappa<\frac{1}{\gamma+ |\chi|^\sigma}$.}  Then
$$
d_{\kappa,\chi,m,\gamma}\ge \frac{1}{8}\frac{(1+\gamma)|\chi|^\sigma}{(1+|\chi|+2m|\chi|)(\gamma+|\chi|^\sigma)}.
$$
The above estimates will be used in the proof of   Theorem \ref{stability-thm}.
\end{remark}

\begin{remark}
\label{prelim-sub-rk2}
For given $M\ge 1$, $0<\kappa<1$, and $T>0$, let
$$
\widetilde {\mathcal{E}}_{\kappa,M,T}=\{u\in C\big([0,T], C_{\rm unif}^b(\RR)\big)\,|\, 0\le u(t,x)\le \widetilde U_{\kappa,M}^+(x)\quad \forall\, t\in [0,T],\,\, x\in\RR\},
$$
where
$$
\widetilde U_{\kappa, M}^+=\min\{M, Me^{-\kappa x}\}.
$$
Assume that  $0<\kappa<1$,   and  $\tilde \kappa$ is such that $0<\kappa<\tilde \kappa \le \min\{ (1+\alpha)\kappa,  m\kappa +1/2,  1\}$. 
By the similar arguments of Lemma \ref{prelim-sub-lm}, we also have that
for any $u\in \widetilde{\mathcal{E}}_{\kappa,M,T}$, the following hold:
\begin{itemize}
\item[(1)] For $D\gg 1$, 
{$W=U_{\kappa,\tilde\kappa, D}(x)$}  is a sub-solution of
$$
W_t= \mathcal{A}(W;u)\quad  t\in (0,T),\quad t\in (0,T),\,\,  x>x^-_{\kappa,\tilde\kappa,D}.
$$

\item[(2)]   For $0<d\ll 1$, $W=d$ is a sub-solution of
$$
W_t= \mathcal{A}(W;u),\quad t\in (0,T),\,\,  x\in\RR.
$$ 
\end{itemize}
This remark  will also be used in the proof of   Theorem \ref{stability-thm}.
\end{remark}

\subsection{Existence of monotone traveling wave solutions with negative sensitivity and proof of Theorem \ref{monotone-thm}(1)}

In this subsection, we prove Theorem \ref{monotone-thm}(1). Throughout this subsection, we assume that $\chi\le 0$ and $\alpha\le m+\gamma-1$.

\begin{proof}[Proof of Theorem \ref{monotone-thm}(1)]
First, for fixed   $c$ satisfying \eqref{cond-on-c-eq1},  i.e,  $c> c^*_{\chi,m,\gamma}$, let  $\kappa=\frac{c-\sqrt{c^2-4}}{2}$. For fixed
$\kappa_1$ satisfying $\kappa<\kappa_1<   \min\{(\alpha+1)\kappa,   m\kappa+1/2, 1\}$,   let $\tilde \kappa$ be such that
  $0<\kappa_1<\tilde \kappa \le \min\{(1+\alpha)\kappa,  m\kappa+1/2,  1\}$. Fix
$M\ge 1$ and  $D\ge D_{M,\kappa,\tilde\kappa,\chi,m,\gamma}$, where $D$ is as Lemma \ref{prelim-sub-lm}(1).
Let
$$
\mathcal{E}_{\kappa,M}=\{u\in C_{\rm unif}^b(\RR)\,|\, U_{\kappa,\tilde\kappa,D}^-(x) \le u(x)\le U_{\kappa,M}^+(x)\,\, \forall\, x\in\RR\},
$$
and
$$
\mathcal{E}'_{\kappa,M}=\{u\in\mathcal{E}_{\kappa,M}\,|\,  u(x_1)\ge u(x_2)\,\, \forall\, -\infty<x_1<x_2<\infty\}.
$$
It is clear that $\mathcal{E}_{\kappa,M}$ $\mathcal{E}'_{\kappa,M}$ are  convex closed subsets of $C_{\rm unif}^b(\RR)$.

Our idea to prove Theorem \ref{monotone-thm}(1)  is to construct a mapping
$\mathcal{T}_{\kappa,1}:\mathcal{E}'_{\kappa,1} \to\mathcal{E}'_{\kappa,1} $ and prove that it has a fixed point $U^*$, which gives rise to a traveling wave solution $u(t,x)=U^*(x-ct)$ of \eqref{main-eq}.
We divide the proof into four steps.
In the rest of the proof, we always assume  that $\chi<0$.

\smallskip
 
\noindent {\bf Step 1.}
In this step,   for given $u\in \mathcal{E}'_{\kappa,1} $,  let $u(t,x; U^+_{\kappa,1}, u)$ be the solution of 
\begin{equation}
\label{main-moving-eq2}
u_t=u_{xx}+cu_x-\chi m u^{m-1} u_x V_x -\chi u^m V+\chi u^{m+\gamma}
+u(1-u^\alpha)
\end{equation}
with initial condition  $u(0,x;U_{\kappa,1}^+,u)=U_{\kappa,1}^+(x)$, where
$V(x)=\Psi(x;u^\gamma(\cdot),1,1)$.
 We prove that
$$
U(x;u):=\lim_{t\to\infty} u(t,x;U_{\kappa,1}^+,u)
$$
exists, and that $U(\cdot;u)\in \mathcal{E}'_{\kappa,1}$.

\smallskip

First, 
note that $V$ is the solution of
$$
V_{xx}-V+u^\gamma(x)=0.
$$
By comparison principle for elliptic equations, there holds
$$
0\le V(x)\le 1\quad {\rm and}\quad 
V_x(x)\le 0\quad  \forall \, x\in\RR.
$$
 Let $w=u_x(t,x; U_{\kappa,1}^+, u)$. Then $w$ satisfies
\begin{align*}
w_t&= w_{xx}+\big(c+|\chi| m u^{m-1} V_x\big)w_x  +\big(|\chi| m (m-1)u^{m-2} w^2
+|\chi|u^m\big)V_x\\
&\quad +\Big(|\chi| m u^{m-1} V_{xx}+|\chi| mu^{m-1} v+\chi (m+\gamma) u^{m+\gamma -1}+1 -(\alpha+1) u^\alpha\Big)w\\
&\le w_{xx}+\big(c+|\chi| m u^{m-1} V_x\big)w_x  \\
&\quad +\Big(|\chi| m u^{m-1} V_{xx}+|\chi| mu^{m-1} v+\chi (m+\gamma) u^{m+\gamma -1}+1 -(\alpha+1) u^\alpha\Big)w\,\,\, 
  \text{(since}\,\,  \,  V_x\le 0)
\end{align*}
for $t>0$, 
and
$$
w(0,x)\le 0\quad \forall\, x\not =0.
$$
It then follows from comparison principle for parabolic equations that
\begin{equation}
\label{U-eq3}
u_x(t,x;U_{\kappa,1}^+, u)=w(t,x)\le 0\quad \forall\, t>0, \,\,  x\in\RR.
\end{equation}
By Lemmas \ref{prelim-super-lm} and \ref{prelim-sub-lm}, 
\begin{equation}
\label{U-eq4}
U_{\kappa,\tilde\kappa,D}^-(x)\le u(t,x;U_{\kappa,1}^+, u)\le U_{\kappa,1}^+(x)\quad \forall\,  t\ge 0,\,\, x\in\RR.
\end{equation}
Hence $u(t,\cdot;U_{\kappa,1}^+,u)\in\mathcal{E}'_{\kappa, 1}$ for all $t\ge 0$.

Next, by \eqref{U-eq4} and comparison principle for parabolic equations, it holds that
\begin{equation*}
U_{\kappa,\tilde\kappa,D}(x)\le u(t_2,x;U_{\kappa,1}^+,u)\le u(t_1,x;U_{\kappa,1},u)\le U_{\kappa,1}^+(x)\quad \forall\, 0\le t_1\le t_2,\,\,\, x\in\RR.
\end{equation*}
It then follows that
$$
U(x;u):=\lim_{t\to\infty} u(t,x;U_{\kappa,1}^+,u)
$$
exists. Moreover, by a priori estimates for parabolic equations, $U(x;u)$ is uniformly continuous on $\RR$. Then by \eqref{U-eq4} again, $U(\cdot;u)\in\mathcal{E}'_{\kappa,1}$.

\medskip

\noindent {\bf Step 2.} In this step, we define 
the mapping $\mathcal{T}_{\kappa,1}:\mathcal{E}'_{\kappa,1}\to \mathcal{E}'_{\kappa,1}$ by 
$$
\mathcal{T}_{\kappa,1} u=U(\cdot;u),
$$
and  prove that $\mathcal{T}_{\kappa,1}$ has a fixed point.

\smallskip

To do so, we let
$$
\mathcal{X}=C_{\rm unif}^b(\RR)
$$
be equipped with the norm
$$
\|u\|_{\mathcal{X}}=\sum_{k=1}^\infty \frac{1}{2^k}\|u\|_{L^\infty([-k,k])}.
$$
Viewing $\mathcal{E}'_{\kappa,1}$ as a subset of $\mathcal{X}$ and applying a priori estimate for parabolic equations, it is not difficult to see that the mapping $\mathcal{T}_{\kappa,1}$ is compact. 
By the arguments in \cite[Theorem 3.1]{SaSh2017-1}, $\mathcal{T}_{1,\kappa}$ is also continuous. Then by 
Schauder fixed-point theorem, there is $U^*\in\mathcal{E}'_{\kappa,1}$ such that
$$
U(\cdot;U^*)=U^*(\cdot).
$$

\smallskip

\noindent {\bf Step 3.} In this step, we prove that
$(u(t,x),v(t,x)=(U^*(x-ct), V^*(x-ct))$ is a traveling wave solution of \eqref{main-eq} satisfying the properties in the statement of 
Theorem \ref{monotone-thm}(1), where $V^*(x)=\Psi(x;U^*,1,1)$.

\smallskip

First,  it is clear that $(u,v)=(U^*(x),V^*(x))$ is a stationary solution of \eqref{main-moving-eq1}. Hence
$(u(t,x)$, $v(t,x))=(U^*(x-ct),V^*(x-ct))$ is a traveling wave solution of \eqref{main-eq} with speed $c$.
Moreover, by \eqref{U-eq4}, it holds 
\begin{equation*}
U_{\kappa,\tilde\kappa,D}^-(x)\le U^*(x)\le U_{\kappa,1}^+(x)\quad \forall\, x\in\RR.
\end{equation*}
This implies that
$$
\lim_{x\to\infty} e^{(\kappa_1-\kappa)x} \Big(\frac{U^*(x-ct)}{e^{-\kappa(x-ct)}}-1\Big)=0.
$$
 It remains to prove that
$$
\lim_{x\to -\infty} U^*(x)=1.
$$
Assume by contradiction that the above does not hold. Then there is $x_n\to -\infty$ such that
$$
0<\lim_{n\to\infty} u^*(x_n)<1.
$$
Without loss of generality, we may assume that
$$
(\widetilde U^*(x),\widetilde V^*(x))=\lim_{n\to\infty} (U^*(x+x_n), V^*(x+x_n))
$$
exists locally uniformly in $x\in\RR$. Then $\inf_{x\in\RR} \widetilde U^*(x)>0$ and
$(\widetilde U^*,\widetilde V^*)$ is a stationary solution of \eqref{main-moving-eq1}. By Proposition
\ref{constant-stability-prop}, $\widetilde U^*\equiv 1$, in particular,
$\widetilde U^*(0)=1$, which is a contradiction.
\end{proof}

\subsection{Existence of  traveling wave solutions with positivity sensitivity and proof of Theorem \ref{monotone-thm}(2)}

In this subsection, we prove Theorem \ref{monotone-thm}(2). Throughout this subsection, we assume that 
$0\le \chi<\min\{\frac{1}{2},\frac{2m+2\gamma}{m^2+m+2\gamma}\}$ and $\alpha=m+\gamma-1$.

\begin{proof}[Proof of Theorem \ref{monotone-thm}(2)]
First, for fixed   $c> 2$,  let  $\kappa=\frac{c-\sqrt{c^2-4}}{2}$. For fixed
$\kappa_1$ satisfying $\kappa<\kappa_1<   \min\{(1+\alpha)\kappa, m\kappa+1/2,  1\}$,   let $\tilde \kappa$ be such that
  $0<\kappa_1<\tilde \kappa \le \min\{(1+\alpha)\kappa, m\kappa+1/2, 1\}$. 
Let $M=M_{\chi}:=\big(\frac{1}{1-\chi}\big)^{1/\alpha}$ and  fix  $D\ge D_{M,\kappa,\tilde\kappa,\chi,m,\gamma}$, where $D$ is as Lemma \ref{prelim-sub-lm}(1).

Next, for given $u\in\mathcal{E}_{\kappa,M_{\chi}}$, let $V(x)=\Psi(x;u,1,1)$ and 
$u(t,x;U_{\kappa,M_{\chi}}^+,u)$ be the solution 
of \eqref{main-moving-eq2} with $u(0,x;U_{\kappa,M_\chi}^+,u)=U_{\chi,M_{\chi}}^+(x)$. By the arguments of Step 1 in the proof of Theorem \ref{monotone-thm}(1), we have
$$
U(x;u):=\lim_{t\to\infty} u(t,x;U_{\kappa,M_{\chi}}^+,u)
$$
exists, and $U(\cdot;u)\in\mathcal{E}_{\kappa,M_{\chi}}$.

Now, by the arguments of Step 2 in the proof of Theorem \ref{monotone-thm}(1),  the mapping $\mathcal{T}_{\kappa,M_{\chi}}:\mathcal{E}_{\kappa,M_{\chi}}\to \mathcal{E}_{\kappa,M_{\chi}}$  defined by 
$$
\mathcal{T}_{\kappa,M_{\chi}} u=U(\cdot;u)
$$
 has a fixed point $U^*\in\mathcal{E}_{\kappa,M_{\chi}}$.

Finally, by the   arguments of Step 3 in the proof of Theorem \ref{monotone-thm}(1),  and  Proposition
\ref{constant-stability-prop},  $(u(t,x),v(t,x))=(U^*(x-ct),V^*(x-ct))$ is a traveling wave solution of \eqref{main-eq} satisfying \eqref{wave-eq2}, \eqref{wave-eq3}, and \eqref{wave-eq4}. 
\end{proof}

\begin{remark}
\label{existence-rk}
(1) The traveling wave solution $(u(t,x),v(t,x))=U(x-ct),V(x-ct))$ obtained from the above construction satisfies
$$
\lim_{x\to \infty} e^{\eta x} \Big(\frac{U(x)}{e^{-\kappa x}}-1\Big)=0
$$
for any $0<\eta\le \min\{\alpha \kappa, (m-1)\kappa +1/2, 1-\kappa\}$. Note that, if  $(U_1,V_1)$ and $(U_2,V_2)$ are two such traveling wave solutions obtained from the above arguments, then
$$
\lim_{x\to \infty} e^{(\eta +\kappa)x} \big(U_1(x)-U_2(x)\big)=
\lim_{x\to \infty} e^{\eta x} \Big(\frac{U_1(x)}{e^{-\kappa x}}-\frac{U_2(x)}{e^{-\kappa x}}\Big)=0
$$
for  any $0<\eta< \min\{\alpha \kappa, (m-1)\kappa +1/2, 1-\kappa\}$.
This implies that
$$
\int_{\RR} e^{2(\eta+\kappa) x}|U_1(x)-U_2(x)|^2 dx<\infty
$$
for  any $0<\eta< \min\{\alpha \kappa, (m-1)\kappa +1/2, 1-\kappa\}$.

\smallskip

(2) In the case that $\frac{1}{2}<\frac{2m+2\gamma}{m^2+m+2\gamma}$,
by the arguments of Theorem \ref{monotone-thm}(2), when $\frac{1}{2}\le \chi<\min\{\frac{2m+2\gamma}{m^2+m+2\gamma},1\}$,  for any $c>2$, \eqref{main-eq} has a traveling wave solution
$(u(t,x),v(t,x))=(U(x-ct),V(x-ct))$ satisfying
$$
\lim_{x\to\infty} (U(x),V(x))=(0,0), \quad \liminf_{x\to -\infty} U(x)>0.
$$
\end{remark}

\section{Uniqueness and stability of traveling waves connecting constant solutions}
\label{uniqueness-stability-section}

In this section we study the uniqueness and stability of traveling wave solutions of \eqref{main-eq} and prove Theorems \ref{stability-thm} and \ref{uniqueness-thm}. To this end, we first establish some a priori  estimates for traveling wave solutions of \eqref{main-eq} connecting $(1,1)$ and $(0,0)$. 

\subsection{Some a priori estimates for traveling wave solutions}

In the first lemma, we provide  some a priori estimates  for traveling wave solutions of \eqref{main-eq} or  stationary solutions of
\eqref{main-moving-eq1} connecting $(1,1)$ and $(0,0)$.  Throughout this subsection, we assume 
$$M_\chi=\begin{cases} 1\quad  &\rm {if}\quad \chi\le 0\cr
\big(\frac{1}{1-\chi}\big)^{1/\alpha}\quad &{\rm  if}\quad  0<\chi<1/2.
\end{cases}
$$

\begin{lemma}
\label{estimate-lm1}
Suppose that $c>2$, and  $(u(x),v(x))$ is  s stationary  solution of \eqref{main-moving-eq1}  connecting $(1,1)$ and $(0,0)$  and $0<u(x)\le \min\{M_\chi, e^{-\kappa x}\}$, where $k=\frac{c-\sqrt{c^2-4}}{2}$.

\begin{itemize}
\item[(1)] There hold
\begin{equation}
\label{estimate-eq0}
|v(x)|\le  M_\chi^\gamma 
 \,\,\, {\rm and}\,\,\,
|v_x(x)|\le  M_\chi^\gamma  \quad \forall\, x\in\RR.
\end{equation}
\end{itemize}

If, in addition,   {$c>\gamma+\frac{1}{\gamma}$,}  which is equivalent to $\kappa<\frac{1}{\gamma}$, then
\begin{equation}
\label{estimate-eq1}
|v(x)|\le  \min\big\{M_\chi^\gamma , \frac{1}{1-\kappa^2 \gamma^2}e^{-\kappa\gamma  x}\big \}
 \,\,\, {\rm and}\,\,\,
|v_x(x)|\le  \min\{M_{\chi}^\gamma ,  \frac{1}{1-\kappa^2 \gamma^2}  e^{-\kappa\gamma  x}\}  \quad \forall\, x\in\RR.
\end{equation}

\begin{itemize}
\item[(2)] $\lim_{x\to\pm \infty} u_x(x)=0$.

\item[(3)] If {$c> m|\chi| M^{m+\gamma-1}_\chi$,}  then  
\begin{equation}
\label{estimate-eq2}
 -\frac{|\chi| M_{\chi}^{m+\gamma}+M_{\chi}}{c-m|\chi| M^{m+\gamma-1} _{\chi} } \,   \le\,  u_x(x)\,  \le  \,  \frac{|\chi| M_\chi^{m+\gamma} +M_\chi(M_\chi^\alpha -1)}{c- m |\chi| M^{m+\gamma-1}_\chi }  \quad \forall\, x\in\RR.
\end{equation}

\item[(4)]  If {$c>\max\big\{\gamma+\frac{1}{\gamma}, m|\chi| M^{m+\gamma-1}_\chi\big\}$,}  then 
\begin{align}
\label{estimate-eq3}
|u_x(x)|\le  \Big(1+\frac{2}{c}\Big) \Big(M_{1,\kappa,\chi,m,\alpha,\gamma}  e^{-\kappa x}+M_{2,\kappa,\chi,m,\alpha,\gamma}e^{-\gamma\kappa x}\Big)\quad \forall\, x\in\RR,
\end{align}
where
$$
M_{1,\kappa,\chi,m,\alpha,\gamma}:= 1+2 |\chi| M^{m+\gamma-1}_\chi+M_\chi^{\alpha}.
$$
and
$$
M_{2, \kappa,\chi,m,\alpha,\gamma}:=\frac{|\chi| m M_{\chi}^{m-1} (|\chi| M_{\chi}^{m+\gamma}+M_\chi^{\alpha+1}) }{(c- m|\chi| M^{m+\gamma-1}_{\chi})(1-\gamma^2\kappa^2)}.
$$
\end{itemize}
\end{lemma}

\begin{proof} 
(1) By the assumption, 
\begin{equation}
\label{proof-estimate-eq1}
0\le u(x)\le \min\{M_{\chi} ,e^{-\kappa x}\}\quad  \forall\, x\in\RR.
\end{equation}
This together with Lemmas \ref{prelim-v-lm2} and \ref{prelim-v-lm3}  implies that \eqref{estimate-eq0} and \eqref{estimate-eq1}  hold.

\smallskip

(2)  First, we 
 prove that $\lim_{x\to -\infty} u_x(x)=0$. Assume by contradiction that 
$\lim_{x\to\infty}u_x(x)\not =0$. Then there is $\epsilon_0>0$ and $x_n\to-\infty$ such that
$$
|u_x(x_n)|\ge \epsilon_0.
$$
Without loss of generality, we may assume that 
the limit
$$
u_\infty(x)=\lim_{n\to\infty} u(x+x_n),\quad v_\infty(x)=\lim_{n\to\infty} v(x+x_n)
$$
exist locally uniformly in $x\in\RR$. Then by regularity for parabolic and elliptic equations,
 $(u_\infty,v_\infty)$ is a stationary solution of \eqref{main-moving-eq1} and
$$
\inf_{x\in\RR} u_\infty (x)>0,\quad \inf_{x\in\RR}v_\infty(x)>0.
$$
We then must have $u_\infty(x)\equiv 1$, but
$|(u_\infty)_x(0)|\ge \epsilon_0$, which is a contradiction. Hence
$\lim_{x\to -\infty} u_x(x)=0$. 

\smallskip

Next, we prove $\lim_{x\to\infty}u_x(x)=0$. Assume by contradiction that $\lim_{x\to\infty} u_x(x_0)\not =0$.
Then by the uniform continuity of $u_x(x)$,  there are $\epsilon_0>0$, $\delta_0>0$, and $x_n\to\infty$ such that
$$
|u_x(x)|\ge \delta_0\quad \forall\, x\in [x_n-\delta_0,x_n+\delta_n],\,\, n=1,2, \cdots.
$$
Without loss of generality, we may assume that
$$
u_x(x)\ge \delta_0\quad \forall\, x\in [x_n-\delta_0,x_n+\delta_n],\,\, n=1,2, \cdots.
$$
Note that
$$
\lim_{n\to\infty} u(x+x_n)=0\quad \forall\, x\in [x_n,x_n+\delta_0],\,\, n=1,2,\cdots.
$$
On the other hand, 
$$
u(x_n+\delta_0)=u(x_n)+\int_{x_n}^{n_n+\delta_0} u_x(y)dy\ge u(x_n)+\delta_0 \epsilon_0\not \to 0
$$
as $n\to\infty$, which is a contradiction.
Hence $\lim_{x\to\infty}u_x(x)=0$.

\smallskip

(3) First, note that
$$
u_{xx}+cu_x-\chi m u^{m-1} u_x v_x-\chi u^m (v-u^\gamma)+u(1-u^\alpha)=0.
$$
By (2), 
$$
\lim_{x\to\pm\infty}u_x(x)=0.
$$
This together with $\lim_{x\to -\infty} u(x)=1$ and $\lim_{x\to\infty}u(x)=0$ implies that 
$u_x(x)<0$ for some $x\in\RR$. Hence 
 there is $x_{\min}\in\RR$ such that
$$
0>u_x(x_{\min})=\min_{x\in\RR} u_x(x),\quad u_{xx}(x_{\min})=0.
$$
It then follows from   \eqref{estimate-eq0}
that
\begin{align*}
\big(c -\chi m u^{m-1}(x_{\min}) v_x(x_{\min}\big)u_x(x_{\min}) &\le \Big (c- {m |\chi| M_\chi^{m+\gamma-1}} \Big) u_x(x_{\min}).
\end{align*}
This together with \eqref{estimate-eq0}  implies that
\begin{align*}
0&=u_{xx}(x_{\min})+cu_x(x_{\min})-\chi m u^{m-1}(x_{\min}) u_x (x_{\min})v_x(x_{\min})\\
&\quad -\chi u^m(x_{\min}) (v(x_{\min})-u^\gamma(x_{\min})) +u(x_{\min})(1-u^\alpha(x_{\min}))\\
&\le \big(c- m|\chi| M^{m+\gamma -1}_{\chi}\big)u_x(x_{\min})+ |\chi| M_{\chi}^{m+\gamma}+M_{\chi}.
\end{align*}
It then follows that
$$
u_x(x)\ge u_x(x_{\min})\ge -\frac{|\chi| M_{\chi}^{m+\gamma}+M_{\chi}}{c-m|\chi| M^{m+\gamma-1} _{\chi} }   \quad \forall\, x\in\RR.
$$

Next, note that if $u_x(x)\le 0$ for all $x\in\RR$, then \eqref{estimate-eq2} is proved. If $u_x(x)>0$ for some $x\in\RR$, then there is $x_{\max}\in\RR$ such that
$$
0<u_{x}(x_{\max})=\max_{x\in\RR} u_x(x),\quad u_{xx}(x_{\max})=0.
$$
By \eqref{estimate-eq0} again,
$$
\big(c-\chi m u^{m-1}(x_{\max}) v_x(x_{\max})\big)u_x(x_{\max})\ge \Big (c- m |\chi| M^{m+\gamma-1}_\chi \Big)u_x(x_{\max}).
$$
This together with \eqref{estimate-eq0}  implies that 
\begin{align*}
0&=u_{xx}(x_{\max})+cu_x(x_{\max})-\chi m u^{m-1}(x_{\max}) u_x (x_{\max})v_x(x_{\max})\\
&\quad -\chi u^m(x_{\max}) (v(x_{\min})-u^\gamma(x_{\max})) +u(x_{\max})(1-u^\alpha(x_{\max}))\\
&\ge (c- m |\chi| M^{m+\gamma-1}_\chi \Big) u_{x}(x_{\max})-|\chi| M_\chi^{m+\gamma}-M_\chi \big(M_{\chi}^{\alpha}-1\big).
\end{align*}
Therefore
$$
u_x(x_{\max})\le \frac{|\chi| M_\chi^{m+\gamma} +M_\chi(M_\chi^\alpha -1)}{c- m |\chi| M^{m+\gamma-1}_\chi }  .
$$
Hence \eqref{estimate-eq2} holds.

\smallskip

(4) 
First, note that
$$
u_{xx}+cu_x-u+ f(x)=0,
$$
where
$$
f(x)= u(x)-\chi m u^{m-1}(x) u_x(x) v_x(x)-\chi u^m (x)(v(x)-u^\gamma(x))+u(x)(1-u^\alpha(x)).
$$
By Lemma \cite[Lemma 2.1]{SaSh2020-1}, it holds that
\begin{align}
\label{proof-estimate-eq2}
u(x)&=\frac{1}{\sqrt{4+c^2}}\int_{\RR} e^{\frac{-\sqrt{4+c^2} |x-y| -c(x-y)}{2}} f(y)dy\nonumber\\
&=\frac{1}{\sqrt{4+c^2}}\left[e^{-\lambda_1^c x}\int_{-\infty}^x e^{\lambda_1^c y} f(y)dy
+e^{\lambda_2^c x}\int_x^\infty e^{-\lambda_2 ^c y} f(y) dy\right],
\end{align}
and
\begin{align}
\label{proof-estimate-eq3}
u_x(x)=\frac{1}{\sqrt{4+c^2}}\left[-\lambda_1^c e^{-\lambda_1^c x}\int_{-\infty}^x e^{\lambda_1^c y} f(y)dy
+\lambda_2^c e^{\lambda_2^c x}\int_x^\infty e^{-\lambda_2 ^c y} f(y) dy\right],
\end{align}
where
$$
\lambda_1^c=\frac{\sqrt{4+c^2}+c}{2},\quad \lambda_2^2=\frac{\sqrt{4+c^2}-c}{2}.
$$

By \eqref{estimate-eq0}, \eqref{estimate-eq0},  and \eqref{estimate-eq2}, it holds that 
\begin{align*}
|f(x)|&\le u(x)+|\chi m u^{m-1} u_x v_x|+|\chi u^m (v-u^\gamma)|+|u(1-u^\alpha)|\\
&\le  u(x)+\frac{|\chi| m M_\chi^{m-1} (|\chi| M_{\chi}^{m+\gamma}+M_{\chi}^{\alpha+1})}{c- m|\chi| M^{m+\gamma-1}_\chi} |v_x| + |\chi| u^m (u^\gamma +v) +| u(1-u^\alpha)| \\
&\le   e^{-\kappa x}+\frac{|\chi| m M_{\chi}^{m-1} (|\chi| M_{\chi}^{m+\gamma}+M_\chi^{\alpha+1}) }{(c- m|\chi| M^{m+\gamma-1}_{\chi})(1-\gamma^2\kappa^2)} e^{-\gamma \kappa  x} +\Big( 2 |\chi| M^{m+\gamma-1}_\chi+M_\chi^{\alpha}\Big) e^{-\kappa   x}\\
&=M_{1,\kappa,\chi,m,\alpha,\gamma} e^{-\kappa x} +M_{2,\kappa,\chi,m,\alpha,\gamma} e^{-\gamma\kappa x}.
\end{align*}
It then follows that
\begin{align*}
|u_x(x)|&\le M_{1,\kappa,\chi,m,\alpha,\gamma}  \frac{1}{\sqrt{4+c^2}}\left[ \lambda_1^c e^{-\lambda_1^c x}\int_{-\infty}^x e^{\lambda_1^c y} e^{-\kappa y} dy
+\lambda_2^c e^{\lambda_2^c x}\int_x^\infty e^{-\lambda_2 ^c y} e^{-\kappa y} dy\right]\\
&\quad + M_{2,\kappa,\chi,m,\alpha,\gamma}  \frac{1}{\sqrt{4+c^2}}\left[ \lambda_1^c e^{-\lambda_1^c x}\int_{-\infty}^x e^{\lambda_1^c y} e^{-\gamma \kappa y} dy
+\lambda_2^c e^{\lambda_2^c x}\int_x^\infty e^{-\lambda_2 ^c y} e^{-\gamma \kappa y} dy\right]\\
&=
 M_{1,\kappa,\chi,m,\alpha,\gamma}  \Big(\frac{\lambda_1^c}{\sqrt{4+c^2}}\frac{1}{\lambda_1^c -\kappa}+
\frac{\lambda_2^c}{\sqrt{4+c^2}}\frac{1}{\lambda_2^c+\kappa}\Big) e^{-\kappa x}
\\
&\quad + M_{2,\kappa,\chi,m,\alpha,\gamma}  \Big(\frac{\lambda_1^c}{\sqrt{4+c^2}}\frac{1}{\lambda_1^c -\gamma \kappa}+
\frac{\lambda_2^c}{\sqrt{4+c^2}}\frac{1}{\lambda_2^c+\gamma \kappa}\Big) e^{-\gamma \kappa x}
\end{align*}
 Note that $0<\kappa<1$, $0<\gamma\kappa<1$, and $c>2$. Hence
\begin{align*}
\frac{\lambda_1^c}{\sqrt{4+c^2}}\frac{1}{\lambda_1^c -\kappa}\le  
\frac{\sqrt{4+c^2}+c}{\sqrt{4+c^2}}\frac{1}{\sqrt{4+c^2}}\le 1+ \frac{1}{c}\, ,
\end{align*}
and
\begin{align*}
\frac{\lambda_2^c}{\sqrt{4+c^2}}\frac{1}{\lambda_2^c+\kappa}&\le  \frac{1}{\sqrt{4+c^2}}\le \frac{1}{c}.
\end{align*} 
Similarly, 
\begin{align*}
\frac{\lambda_1^c}{\sqrt{4+c^2}}\frac{1}{\lambda_1^c -\gamma \kappa}\le    
\frac{\sqrt{4+c^2}+c}{\sqrt{4+c^2}}\frac{1}{\sqrt{4+c^2}}\le  1+\frac{1}{c}\, ,
\end{align*}
and
\begin{align*}
\frac{\lambda_2^c}{\sqrt{4+c^2}}\frac{1}{\lambda_2^c+\gamma \kappa}&\le \frac{1}{\sqrt{4+c^2}}\le \frac{1}{c}.
\end{align*} 
Hence
$$
|u_x(x)|\le  \Big(1+\frac{2}{c}\Big) \Big(M_{1,\kappa,\chi,m,\alpha,\gamma}  e^{-\kappa x}+M_{2,\kappa,\chi,m,\alpha,\gamma}e^{-\gamma\kappa x}\Big).
$$
This proves \eqref{estimate-eq3}.
\end{proof}

\begin{remark}
\label{estimate-rk1}
 If {$c>\max\big\{\gamma+|\chi|^\sigma+\frac{1}{\gamma+|\chi|^\sigma }, m|\chi| M^{m+\gamma-1}_\chi +|\chi|^\sigma\}$ for some $\sigma>0$,} then  by \eqref{estimate-eq2}, there holds
\begin{equation}
\label{estimate-eq2-1}
 |  u_x(x)|  \le  \,  \frac{M'_{\chi,m,\alpha,\gamma}}{|\chi|^\sigma }  \quad \forall\, x\in\RR,
\end{equation}
where 
$$
M'_{\chi,m,\alpha,\gamma}= |\chi| M_\chi^{m+\gamma} +M_\chi^{1+\alpha}.
$$
By \eqref{estimate-eq3}, there holds
\begin{align}
\label{estimate-eq3-1}
|u_x(x)|\le   \frac{M''_{\chi,m,\alpha,\gamma,\sigma}}{|\chi|^{2\sigma}} e^{-\kappa x} \quad \forall\, x\ge 0,
\end{align}
where
$$
M''_{\chi,m,\alpha,\gamma,\sigma}=2\Big(\big(1+2 |\chi| M^{m+\gamma-1}_\chi+M_\chi^{\alpha}\big)|\chi|^{2\sigma}+
+|\chi| m M_{\chi}^{m-1} (|\chi| M_{\chi}^{m+\gamma}+M_\chi^{\alpha+1}) (\gamma+|\chi|^\sigma)\Big). 
$$
Note that both $M'_{\kappa,m,\alpha,\gamma}$ and $M''_{\kappa,m,\alpha,\gamma,\sigma}$ stay bounded as $\chi\to 0$.
\end{remark}

In the second lemma, we provide estimates for $\frac{u_x}{u}$ for stationary solutions of \eqref{main-moving-eq1}.

\begin{lemma}
\label{estimate-lm2}Assume that $c>\max\big\{\gamma+\frac{1}{\gamma}, m|\chi| M^{m+\gamma-1}_\chi\big\}$.   Suppose that  $(u(x),v(x))$ is  s stationary  solution of \eqref{main-moving-eq1}  connecting $(1,1)$ and $(0,0)$  and $0<u(x)\le \min\{M_\chi, e^{-\kappa x}\}$, where $\kappa=\frac{c-\sqrt{c^2-4}}{2}$.
Then
$$
\big|\frac{u_x(x)}{u(x)}\big|\le  \widetilde M_{c,\chi,m,\alpha,\gamma},
$$
where
$$
\widetilde M_{c,\chi,m,\alpha,\gamma}:= \frac{1}{2}\left( c+|\chi| m M_{\chi}^{m+\gamma-1}+\sqrt{ (c+|\chi| m M_{\chi}^{m+\gamma-1})^2 +4|\chi| M_{\chi}^{m+\gamma-1}+4M_{\chi}^\alpha}\right).
$$
\end{lemma}

\begin{proof}
First, let $w(x)=\frac{u_x(x)}{u(x)}$.
Note that
$$
u_{xx}+c u_x(x)-\chi m u^{m-1} u_x(x) v_x(x)-\chi u^m(x) (v(x)-u^\gamma (x))+u(x) (1-u^\alpha(x))=0.
$$
Hence
\begin{equation}
\label{proof-estimate-3-eq1}
\frac{u_{xx}}{u}+c w -\chi m u^{m-1} v_x w -\chi u^{m-1} (v-u^\gamma)+(1-u^\alpha(x))=0.
\end{equation}
By Lemmas \ref{prelim-super-lm} and \ref{prelim-sub-lm}, and   Lemma \ref{estimate-lm1}(2) and (4), 
$$
\lim_{x\to -\infty} w(x)=0,\quad \limsup_{x\to\infty}|w(x)|<\infty.
$$
Hence 
$$
\|w\|_\infty<\infty.
$$
Then by \eqref{proof-estimate-3-eq1},
$$
\|\frac{u_{xx}}{u}\|_\infty<\infty.
$$

Next, if there is $x_0\in\RR$ such that
$$
|w(x_0)|=\|w\|_\infty,
$$
then
$$
w_x(x_0)=\frac{u_{xx}(x_0)}{u(x_0)}-w^2 (x_0)=0.
$$
This together with \eqref{proof-estimate-3-eq1} implies that
$$
w(x_0)^2+\Big(c-\chi m u^{m-1}(x_0)v_x(x_0)\big) w(x_0) -\chi u^{m-1}(x_0) (v(x_0)-u^\gamma(x_0))+(1-u^\alpha(x)(x_0))=0
$$
Let 
$$a(x_0)=c-\chi m u^{m-1}(x_0)v_x(x_0),
$$
and
$$
b(x_0)=-\chi u^{m-1}(x_0) (v(x_0)-u^\gamma(x_0))+(1-u^\alpha(x)(x_0)).
$$
Then 
$$
w(x_0)= \frac{-a(x_0)\pm \sqrt{(a(x_0))^2- 4 b(x_0)}}{2}.
$$
Hence
\begin{align*}
|w(x_0)|&\le\frac{|a(x_0)| +\sqrt{|a(x_0)|^2+4|b(x_0)|}}{2}\\
&\le \frac{1}{2}\left( c+|\chi| m M_{\chi}^{m+\gamma-1}+\sqrt{ (c+|\chi| m M_{\chi}^{m+\gamma-1})^2 +4|\chi| M_{\chi}^{m+\gamma-1}+4M_{\chi}^\alpha}\right)\\
&=\widetilde M_{c,\chi,m,\alpha,\gamma}.
\end{align*}

If $w(x_0)\not = \|w\|_\infty$ for any $x_0\in\RR$, then there is $x_n\to\infty$ such that
$$
\lim_{n\to\infty} |w(x_n)|=\|w\|_\infty.
$$
Note that
$$
w_x=\frac{u_{xx}}{u}-w^2.
$$
Hence $\|w_x\|_\infty<\infty$. This together with \eqref{proof-estimate-3-eq1} implies that
$\|\big(\frac{u_{xx}}{u}\big)_x\|_\infty<\infty$ and then
$$
\|w_{xx}\|_\infty=\|\big(\frac{u_{xx}}{u}\big)_x-2w w_x\|_\infty<\infty.
$$

  Without loss of generality, we may assume that
$$
w^*(x):=\lim_{x\to\infty} w(x+x_n) \quad {\rm and} \quad w^*_x(x)=\lim_{n\to\infty} w_x(x+x_n)
$$
exist locally uniformly in $x\in\RR$.  Clearly,
$$
|w^*(0)|=\|w^*\|_\infty=\|w\|_\infty,
$$
and 
$$
\lim_{n\to\infty} \frac{u_{xx}(x_n)}{u(x_n)}=\lim_{n\to\infty} \big(w^2(x_n)+w_x(x_n)\big)=(w^*(0))^2+w_x^*(0)=(w^*(0))^2.
$$
Then  by \eqref{proof-estimate-3-eq1}, it holds that
$$
(w^*(0))^2+cw^*(0) +1=0.
$$
This implies that
$$
\|w\|_\infty=|w^*(0)|\le \frac{c+\sqrt{c^2-4}}{2}.
$$
Note that
$$
\frac{c+\sqrt{c^2-4}}{2}\le \widetilde M_{c,\chi,m,\alpha,\gamma}.
$$
The lemma is thus proved.
\end{proof}

\begin{remark}
\label{estimate-rk2}
 Assume that
 $c>\max\big\{\gamma+|\chi|^\sigma +\frac{1}{\gamma+|\chi|^\sigma }, m|\chi| M^{m+\gamma-1}_\chi+|\chi|^\sigma \big\}$, 
and   $(u(x),v(x))$ is  s stationary  solution of \eqref{main-moving-eq1}  connecting $(1,1)$ and $(0,0)$  and $0<u(x)\le \min\{M_\chi, e^{-\kappa x}\}$, where $\kappa=\frac{c-\sqrt{c^2-4}}{2}$. Then
\begin{equation}
\label{ratio-eq}
\big|\frac{u_x(x)}{u(x)}\big|\le  \frac{M'''_{\chi,m,\alpha,\gamma,\sigma}}{|\chi|^{2\sigma}},
\end{equation}
where
\begin{align*}
&M'''_{\chi,m,\alpha,\gamma,\sigma}\\
&:= \begin{cases}
\frac{|\chi|^{2\sigma}}{2}\left( 2.5+|\chi| m M_{\chi}^{m+\gamma-1}+\sqrt{ (2.5+|\chi| m M_{\chi}^{m+\gamma-1})^2 +4|\chi| M_{\chi}^{m+\gamma-1}+4M_{\chi}^\alpha}\right)\,\,\, & {\rm if}\,\, c\le 2.5\cr\cr
\max\Big\{  \frac{8 (1+|\chi|+2m|\chi|)(\gamma+|\chi|^\sigma) } {1+\gamma} \, M'_{\chi, m,\alpha,\gamma},\,  2M''_{\chi,m,\alpha,\gamma,\sigma}\Big\}\quad &{\rm if}\,\, c>2.5.
\end{cases}
\end{align*}
In fact,  if  $c\le 2.5$, \eqref{ratio-eq}   follows from 
Lemma \ref{estimate-lm2}. If $c>2.5$, then  $ \kappa<1/2$, and \eqref{ratio-eq}  follows from  Remark \ref{prelim-sub-rk1},   \eqref{estimate-eq2-1}, and \eqref{estimate-eq3-1}.
Note that $M'''_{\kappa,m,\alpha,\gamma,\sigma}$ stays bounded as $\chi\to 0$.
\end{remark}

\begin{lemma}
\label{estimate-lm3}
Suppose that $u_i\in C_{\rm unif}^b(\RR)$, $0\le u_i\le M$ ($i=1,2$) for some $M\ge 1$,  and $\int_{\RR} e^{2\eta x}
|u_2(x)-u_1(x)|^2<\infty$
for some $0<\eta<1$.
Let $v(x)=\Psi(x;u_2^\gamma-u_1^\gamma ,1,1)$. Let $U(x)=e^{\eta x} (u_2(x)-u_1(x))$ and $V(x)=e^{\eta x} v(x)$. If $\gamma \ge 1$, 
then
\begin{equation}
\label{estimate-eq4}
\int_{\RR }V^2(x)  dx\le \frac{\gamma^2 M^{2(\gamma-1)} }{(1-\eta)^2}\int_{\RR} U^2,
\end{equation}
and
\begin{equation}
\label{estimate-eq5}
\int_{\RR} V_x^2(x)dx\le \frac{\gamma^2 M^{2(\gamma-1)}}{1-\eta^2}\int_{\RR} U^2,
\end{equation}
where $U(x)=e^{\eta x}(u_2(x)-u_1(x))$ and $V(x)=e^{\eta x}(v_2(x)-v_1(x))$.
\end{lemma}

\begin{proof}
First,  we prove \eqref{estimate-eq4}.
Observe that  
\begin{align*}
|u_2^{\gamma} (x)-u_1^\gamma(x)|& =|{\gamma}\int_{u_1(x)}^{u_2(x)} w^{\gamma-1} dw|\\
&=\gamma |(u_2(x)-u_1(x))|\int_0^1 \big(\tau u_2(x)+(1-\tau)u_1(x)\big)^{\gamma -1}d\tau\\
&\le
\gamma M^{\gamma-1}  |u_2(x)-u_1(x)| \quad {\rm (since}\,\, \gamma\ge 1, \,\, 0\le u_1(x)\le M,\,\, 0\le u_2(x)\le M.
\end{align*} 
Then by
Lemma \ref{prelim-v-lm1}, we have
\begin{align*}
|V(x)|&=|\frac{1}{2} e^{\eta x} \int_{\RR} e^{-|x-y|}(u_2^\gamma (y) -u_1^\gamma(y))dy|\\
&\le | \frac{1}{2}\int_{\RR} e^{-|x-y|}e^{\eta (x-y)} e^{\eta y} (u_2^\gamma(y)-u_2^\gamma(y))dy|\\
&\le M^{\gamma-1}  \frac{\gamma}{2} \int_{\RR} e^{-|x-y|}e^{\eta(x-y)} e^{\eta y} |u_2(y)-u_1(y)| dy\\
&=M^{\gamma-1} \frac{\gamma}{2}\int_{-\infty}^ x e^{-(x-y)} e^{\eta(x-y)} |U(y)|dy
+M^{\gamma-1} \frac{\gamma}{2}\int_x^\infty e^{-(y-x)} e^{\eta(x-y)} |U(y)|dy\\
&\le M^{\gamma-1}\frac{\gamma}{2}\int_{\RR} e^{-(1-\eta)|x-y|} |U(y)|dy.
\end{align*}
Let
$$
\overline V(x)=\frac{\gamma}{2}\int_{\RR} e^{-(1-\eta)|x-y|} e^{\eta y}|U(y)|dy.
$$
Then by Lemma \ref{prelim-v-lm1}, $\overline V$ is the solution of
$$
\overline V_{xx}-(1-\eta)^2\overline V+\gamma(1-\eta)|U|=0.
$$ 
This implies that
$$
\int_{\RR} V^2\le M^{2(\gamma-1)} \int_{\RR} \overline V^2 \le \frac{\gamma^2 M^{2(\gamma-1)}}{(1-\eta)^2}\int_{\RR}|U|^2.
$$
Therefore, \eqref{estimate-eq4} holds.

Next,  we prove \eqref{estimate-eq5}. Observe that $V$ satisfies
$$
V_{xx}-2\eta V_x+(\eta^2 -1)V +e^{\eta x}(u_2^\gamma-u_1^\gamma)=0.
$$
This implies that
\begin{align*}
\int_{\RR} V_x^2&=(\eta^2 -1)\int_{\RR}^2 V^2+\int_{\RR} e^{\eta x}(u_2^\gamma(x)-u_1^\gamma(x)) V\\
&\le -(1-\eta^2)\int_{\RR}V^2+\gamma M^{\gamma-1}\int_{\RR} |U(x)|\cdot |V(x)|\\
&\le \frac{\gamma^2 M^{2(\gamma-1)}}{1-\eta^2}\int_{\RR} U^2.
\end{align*}
This proves \eqref{estimate-eq4}.
\end{proof}

\subsection{Stability of traveling wave solutions and proof of Theorem \ref{stability-thm}}

In this subsection, we prove Theorem \ref{stability-thm}.

Observe that for any $\beta>0$, and $u_1,u_2>0$,
\begin{align*}
u_2^{\beta} -u_1^\beta& ={\beta}\int_{u_1}^{u_2} w^{\beta-1} dw\\
&=\beta (u_2-u_1)\int_0^\tau \big(\tau u_2+(1-\tau)u_1\big)^{\beta -1}d\tau.
\end{align*}

\begin{proof}[Proof of Theorem \ref{stability-thm}]
Suppose that $(U^*,V^*)$ is a positive stationary solution of \eqref{main-moving-eq1} connecting $(1,1)$ and $(0,0)$ with speed $c>c^{**}_{\chi,m,\alpha,\gamma}$
for some $c^{**}_{\chi,m,\alpha,\gamma}$ to be determined later, and satisfying \eqref{decay-rate-eq1}.
Let $(u(t,x),v(t,x))=(u(t,x;u_0),v(t,x;u_0))$ be the solution of \eqref{main-moving-eq1} with $u_0\in C_{\rm unif}^b(\RR)$ satisfying \eqref{stability-u0-eq1} for  some $\eta\in (\kappa,\frac{1}{1+|\chi|^{1/6}})$.

We  first assume that
\begin{equation}
\label{cond-on-c}
c>c_{1, \chi,\gamma}:=\gamma+|\chi|^{1/6}+\frac{1}{\gamma +|\chi|^{1/6}}\quad {\rm and}\quad  c>
c_{2,\chi,m,\gamma}:=m |\chi| M_{\chi}^{m+\gamma-1} +|\chi|^{1/6}.
\end{equation}
In the following, we put
$$
\sigma=\frac{1}{6}.
$$
Note that  \eqref{cond-on-c} implies 
\begin{equation}
\label{cond-on-kappa}
 \kappa<\frac{1}{\gamma +|\chi|^\sigma} \quad {\rm and}\quad  c>m|\chi| M_\chi^{m+\gamma-1}.
\end{equation}
By Proposition \ref{global-existence-prop},
$$
\limsup_{t\to\infty} \sup_{x\in\mathbb{R}}u(t,x;u_0)\le M_\chi.
$$
Then for any $\widetilde M_\chi>M_\chi$,  without loss of generality, we may assume that
\begin{equation}
\label{u-bound-eq}
0\le u(t,x;u_0)\le \widetilde M_{\chi}\quad \forall\, t\ge 0,\,\, x\in\mathbb{R}.
\end{equation}
This together with Lemmas \ref{prelim-v-lm2} and comparison principle for elliptic equations  implies that
\begin{equation}
\label{v-bound-eq}
0\le v(t,x;u_0)\le \widetilde M_\chi^\gamma\quad {\rm and}\quad 
 |v_x(t,x;u_0)|\le \widetilde M_{\chi}^\gamma\quad \forall\, t\ge 0,\,\, x\in\mathbb{R}.
\end{equation}
In the following, we choose $\widetilde M_\chi>M_\chi$ such that  \eqref{cond-on-kappa} holds with $M_\chi$ being replaced by $\widetilde M_\chi$.

\smallskip

Next, let $\widetilde U(t,x)=u(t,x)-U^*(x)$ and $\widetilde V(t,x)=v(t,x)-V^*(x)$. 
For $\beta\ge 0$, let
$$
a_\beta(t,x):=\beta \int_0^1 \big(\tau u(x)+(1-\tau) U^*(x)\big)^{\beta-1}d\tau.
$$
Then $(\widetilde U,\widetilde V)$ satisfies
\begin{equation}
\label{stability-eq1}
\begin{cases}
\widetilde U_t(t,x)=\widetilde U_{xx}+(c- \chi b_1(t,x)) \widetilde  U_x  \cr
\qquad\qquad + \Big(1-a_{1+\alpha} (t,x)-\chi\big (v a_m(t,x)  + a_{m+\gamma}(t,x)
 +b_2(t,x)\big)\Big)\widetilde U\\
\qquad\qquad  - \chi b_3(t,x) \widetilde  V_x + \chi b_4(t,x) \widetilde V,\quad &x\in\mathbb{R},
\cr\cr
0=\widetilde V_{xx}v- \widetilde V_x+  a_\gamma(t,x)  \widetilde U,\quad &x\in\mathbb{R},
\end{cases}
\end{equation}
where 
$$
 b_1(t,x)=  m  v_{x}(t,x) u^{m-1}(t,x),
\,\,
 b_2(t,x)=m U_{x}^*(x) v_{x}(t,x) a_{m-1}(t,x)  ,
$$
and
$$
 b_3(t,x)= m  (U^*(x))^{m-1} U^*_x(x),\,\,
 b_4(t,x)=  (U^*(x))^m.
$$
Let $U=e^{\eta x}\widetilde U$ and $V=e^{\eta x}\widetilde V$.
Then $(U,V)$ satisfies
\begin{equation}
\label{uniqueness-eq2}
\begin{cases}
U_t=U_{xx} +\big(c-\chi b_1(x)-2\eta\big)U_x\cr
\qquad +\big(\eta^2-c\eta+1-a_{1+\alpha} -\chi( v a_m  + a_{m+\gamma}-\eta b_1 + b_2\big) U\cr
\qquad\,\,  -\chi b_3 V_x+\chi (b_4+\eta b_3\big) V,\quad &x\in\mathbb{R}, \cr\cr
0=V_{xx}-2\eta V_x-(1-\eta^2) V +a_\gamma(t,x) U,\quad &x\in\mathbb{R}.
\end{cases}
\end{equation}
By \cite[Theorem 7.1.3]{Hen}, $U(t,\cdot)$, $U_x(t,\cdot)\in L^2(\mathbb{R})$ for all $t>0$.
Multiplying the first equation in \eqref{uniqueness-eq2} by $U$ and integrating over $\RR$ yields
\begin{align*}
\frac{1}{2}\frac{d}{dt}\int_{\RR} U^2&=-\int_{\RR} U_x^2(x)dx+\underbrace{\int_{\RR} \big(c-\chi b_1-2\eta\big)U_xUdx}_{J_1} \\
&\quad +\underbrace{\int_{\RR}\big(\eta^2-c\eta+1-a_{1+\alpha}  -\chi( v a_m + a_{m+\gamma}-\eta b_1+b_2\big) U^2dx}_{J_2}\\
&\quad -\underbrace{\int_{\RR} \chi b_3 V_xUdx}_{J_3}+\underbrace{\int_{\RR}\big(\chi b_4+\eta\chi  b_3\big) V Udx}_{J_4}.
\end{align*}

We divide the rest of the proof into four steps.

\smallskip

\noindent{\bf Step 1.} In this step, we provide estimates for $b_i$ ($i=1,2,3,4$). 

\smallskip

\noindent {\bf (i) Estimate for  $ b_1(t,x)=  m  v_{x}(t,x) u^{m-1}(t,x)$.}  By \eqref{u-bound-eq} and \eqref{v-bound-eq}, we have
\begin{equation}
\label{b-1-eq}
|b_1(t,x)|\le b_{1,\chi,m,\gamma}(\widetilde M_\chi):= m \widetilde M_\chi^{m+\gamma-1}.
\end{equation}

\smallskip

\noindent {\bf (ii)  Estimate for  $b_2(t,x)=m U_{x}^*(x) v_{x}(t,x) a_{m-1}(t,x)$}.  We divide  the estimate of $b_2$ into three cases.

\smallskip

\noindent {\bf Case ii.1.}  $m=1$. In this case,  $a_{m-1}(t,x)\equiv 0$. Hence
\begin{align}
\label{b-2-eq1}
\|b_2\|_\infty=0.
\end{align}

\smallskip

\noindent {\bf  Case ii. 2.}  $m\ge 2$. In this case,  we have 
$$
|a_{m-1}(t, x)|\le (m-1) \widetilde M_{\chi}^{m-2}.
$$
Then we have
\begin{align}
\label{b-2-eq2}
 \|b_2\|_\infty&\le m(m-1)\widetilde  M_{\chi}^{m+\gamma-1}\|U_x^*\|_\infty\nonumber\\
&\le m(m-1) \widetilde M_{\chi}^{m+\gamma-1} \frac{|\chi| \widetilde M_\chi^{m+\gamma} +\widetilde M_\chi(\widetilde M_\chi^\alpha -1)}{c- m |\chi| \widetilde M^{m+\gamma-1}_\chi } \nonumber\\
&\le m(m-1) \widetilde M_{\chi}^{m+\gamma-1} \frac{|\chi| \widetilde M_\chi^{m+\gamma} +\widetilde M_\chi(\widetilde M_\chi^\alpha -1)}{|\chi|^\sigma} .
\end{align}

\smallskip

\noindent {\bf  Case ii.3.} $1<m<2$. Note hat
\begin{align*}
|b_2(t,x)|&=m|U_x^*(x)|\cdot |v_x(t,x)|\cdot |a_{m-1}(t,x)|\\
&\le m(m-1)|U_x^*(x)| \cdot |v_x(t, x)|\cdot \int_0^1 (\tau u(t,x)+(1-\tau)U^*(x))^{m-2}\\
&\le \frac{m(m-1)|U_x^*(x)| \cdot |v_x(t,x)|}{(U^*(x))^{2-m}}\int_0^1 \frac{1}{(1-\tau)^{2-m}}d\tau\\
&=\frac{m|U_x^*(x)| \cdot |v_x(t,x)| (U^*)^{m-1}}{(U^*(x))}\\
&\le m \widetilde M_\chi ^{m +\gamma -1} \frac{|U_x^*(x)| }{U^*(x)}.
\end{align*}
By \eqref{ratio-eq},
\begin{align*}
\frac{|U^*_x(x)|}{U^*(x)|}\le \frac{M'''_{\chi, m, \alpha,\gamma,\sigma}}{|\chi|^{2\sigma}},
\end{align*}
which implies that
\begin{align}
\label{b-2-eq3}
\|b_2\|_\infty\le \frac{m \widetilde M_{\chi}^{m+\gamma-1} M'''_{\chi,m,\alpha,\gamma,\sigma}}{|\chi|^{2\sigma}}.
\end{align}

\smallskip

By \eqref{b-2-eq1}-\eqref{b-2-eq3}, there is $b_{2,\chi,m,\alpha,\gamma}(\widetilde M_\chi)$ which stays bounded as $\chi\to 0$ such that
\begin{equation}
\label{b-2-eq}
\|b_2\|_\infty\le \frac{1}{|\chi|^{2\sigma}} b_{2,\chi,m,\alpha,\gamma}(\widetilde M_\chi).
\end{equation}

\smallskip

\noindent{\bf (iii) Estimate for  
$ b_3(t,x)= m  (U^*(x))^{m-1} U^*_x(x)$.}
By Lemma \ref{estimate-lm1}, it holds that
\begin{align}
\label{b-3-eq}
\|b_3\|_\infty &\le  m \widetilde M_{\chi}^{m-1} \cdot  \frac{|\chi| \widetilde M_\chi^{m+\gamma} +\widetilde M_\chi^{1+\alpha}}{c- m |\chi| \widetilde M^{m+\gamma-1}_\chi } \le   m \widetilde M_{\chi}^{m-1} \cdot  \frac{|\chi| \widetilde M_\chi^{m+\gamma} +\widetilde M_\chi^{1+\alpha}}{|\chi|^\sigma }\nonumber\\
&=\frac{b_{3,\chi,m,\alpha,\gamma}(\widetilde M_\chi)}{|\chi|^\sigma} .
\end{align}
where
$$
b_{3,\chi,m,\alpha,\gamma}(\widetilde M_\chi):=m \widetilde M_{\chi}^{m-1} \cdot (|\chi| \widetilde M_\chi^{m+\gamma} +\widetilde M_\chi^{1+\alpha}).
$$

\noindent {\bf (iv) Estimate for  $b_4(t,x)=  (U^*(x))^m$.}
It  is clear that
\begin{equation}
\label{b-4-eq}
\|b_4\|_\infty\le b_{4,\chi,m}(\widetilde M_\chi):=\widetilde  M_{\chi}^m.
\end{equation}

\smallskip

\noindent {\bf Step 2.} In this step, we  give estimates for $J_i$ $(i=1,2,3,4)$ based on the estimates of $\|b_j\|_\infty$ ($j=1,2,3,4$).

\smallskip

\noindent{\bf (i) Estimate  for $J_1$.} By \eqref{b-1-eq}, it holds that
\begin{align}
\label{J-1-eq}
J_1& =\int_{\RR} (c-2\eta) U_x Udx -\chi \int_{\RR} b_1(t,x) U_x U=-\chi \int_{\RR} b_1(t, x) U_xU\nonumber\\
&\le \frac{1}{2}\int_{\RR} U_x^2 + \frac{  |\chi|^{2 }}{2} b_{1,\chi,m,\gamma} (\widetilde M_\chi)\int_{\RR}  U^2.
\end{align}
\smallskip

\noindent {\bf (ii)  Estimate for $J_2$.} 

Note that  
$$
-\chi\big(v a_m(x)+a_{m+\gamma}(x)\big)\le \begin{cases} \le |\chi|(2m+\gamma) \widetilde M_{\chi}^{m+\gamma-1}  \quad &{\rm if}\quad \chi<0\cr\cr
0\quad &{\rm if}\quad \chi>0.
\end{cases}
$$
Hence
\begin{align*}
&\eta^2-c\eta+1-a_{1+\alpha}  -\chi( - v a_m(x)  + a_{m+\gamma})-\eta\chi  b_1(t,x) +\chi b_2(t,x)\\
&\le \eta^2 -c\eta +1+|\chi| (2m+\gamma)\widetilde M_{\chi}^{m+\gamma-1}-\eta\chi  b_1(t,x) +\chi b_2(t,x).
\end{align*}
This  together with \eqref{b-1-eq} and \eqref{b-2-eq} implies that
\begin{equation}
\label{J-2-eq}
J_2\le \Big ( \eta^2 -c\eta +1+|\chi| (2m+\gamma)\widetilde M_\chi^{m+\gamma-1}+ |\chi| \eta  b_{1,\chi,m,\gamma} (\widetilde M_\chi)  +|\chi|^{1-2\sigma}  b_{2,\chi,m,\alpha,\gamma}(\widetilde M_\chi)\Big) \int_{\RR}  U^2.
\end{equation}

\smallskip

\noindent{\bf (iii)  Estimate for $J_3$.} By Lemma \ref{estimate-lm3} and \eqref{b-3-eq}, we have
\begin{align*}
-J_3&=-\int_{\RR} \chi b_3(x) V_xUdx\le |\chi| \frac{\|b_3\|_\infty }{2}\int_{\RR} |V_x|^2 +|\chi| \frac{\|b_3\|_\infty }{2}\int_{\RR}  U^2\\
&\le |\chi|^{1-\sigma} \frac{b_{3,\chi,m,\alpha,\gamma}(\widetilde M_\chi)}{2}\Big(1+\frac{\gamma^2}{1-\eta^2}\Big)\int_{\RR} U^2.
\end{align*}
 Since $\eta\in(\kappa,\frac{1}{1+|\chi|^\sigma})$, there holds that
$$
1+\frac{\gamma^2 }{1-\eta^2}\le 1+\frac{\gamma^2}{1-\eta}\le \frac{1}{|\chi|^\sigma}\big(|\chi|^\sigma +\gamma^2 (1+|\chi|^\sigma)\big).
$$
Hence
\begin{equation}
\label{J-3-eq}
J_3\le \frac{ |\chi|^{1-2\sigma} }{2}  b_{3,\chi,m,\alpha,\gamma} (\widetilde M_\chi)\big(|\chi|^\sigma +\gamma^2 (1+|\chi|^\sigma)\big) \int_{\RR} U^2.
\end{equation}

\smallskip

\noindent{\bf (iv)  Estimate for $J_4$.}   
Since  $\eta\in (\kappa,\frac{1}{1+|\chi|^\sigma})$, it holds
$$
1+\frac{\gamma^2 }{(1-\eta)^2}\le  \frac{1}{|\chi|^{2\sigma}}\big(|\chi|^{2\sigma}+\gamma^2 (1+|\chi|^\sigma)^2\big).
$$
Then by Lemma \ref{estimate-lm3}, \eqref{b-3-eq} and \eqref{b-4-eq}, it holds that
\begin{align}
\label{J-4-eq}
J_4&\le \frac{1}{2}\int_{\RR}
|\chi b_4(t,x)+\eta \chi b_3(t,x)| V^2+\frac{1}{2} \int_{\RR}
|\chi b_4(x)+\eta \chi b_3(x)| U^2\nonumber\\
&\le \frac{1}{2}\Big(|\chi| \widetilde M_\chi^m +\eta |\chi|^{1-\sigma} b_{3,\chi,m,\alpha,\gamma}(\widetilde M_\chi)\Big)\Big(\int_{\RR} V^2+\int_{\RR} U^2\Big)\nonumber\\
&\le \frac{1}{2}\Big(|\chi| \widetilde M_\chi^m+\eta |\chi|^{1-\sigma} b_{3,\chi,m,\alpha,\gamma}(\widetilde M_\chi)\Big) \Big(1+\frac{\gamma^2}{(1-\eta)^2}\Big)\int_{\RR}U^2\nonumber\\
&\le \frac{|\chi|^{1-3\sigma}}{2}\big(|\chi|^{3\sigma} \widetilde M_\chi^m+\eta b_{3,\chi,m,\alpha,\gamma}(\widetilde M_3)\big)\big(|\chi|^{2\sigma}+\gamma^2 (1+|\chi|^\sigma)^2\big)\int_{\mathbb{R}} U^2.
\end{align}

\smallskip

By \eqref{J-1-eq}-\eqref{J-4-eq}, 
\begin{align}
\label{U-square-eq1}
&\frac{1}{2}\frac{d}{dt}\int_{\RR} U^2+\frac{1}{2}\int_{\RR} U_x^2 dx\le \left( \eta^2 -(c-|\chi|^{1-3\sigma} \widetilde D'_{\chi}) \eta +(1+|\chi|^{1-3\sigma} \widetilde D''_{\chi})\right)\int_{\RR}U^2,
\end{align}
where 
\begin{align}
\label{D-1-eq}
\widetilde D'_\chi=\widetilde D'_{\chi,m,\alpha,\gamma}(\widetilde M_\chi):=|\chi|^{3\sigma} b_{1,\chi,m,\kappa}(\widetilde M_\chi)+\frac{b_{3,\chi,m\alpha,\kappa}(\widetilde M_\chi)}{2}\big(|\chi|^{2\sigma} +\gamma^2 (1+|\chi|^\sigma)^2\big),
\end{align}
and
\begin{align}
\label{D-2-eq}
\widetilde D''_\chi&=\widetilde D''_{\chi,m,\alpha,\gamma}(\widetilde M_\chi):=\frac{1}{2} |\chi|^{1+\sigma} b_{1,\chi,m,\gamma}(\widetilde M_\chi)+|\chi|^{3\sigma} (2m+\gamma)\widetilde M_\chi^{m+\gamma-1}+|\chi|^\sigma b_{2,\chi,m,\alpha,\gamma}(\widetilde M_\chi)\nonumber\\
&\quad \qquad\qquad \quad\,\, +\frac{|\chi|^\sigma}{2} b_{3,\chi,m,\alpha,\gamma}(\widetilde M_\chi) \big(|\chi|^\sigma +\gamma^2 (1+|\chi|^\sigma)\big) \nonumber\\
&\quad\qquad\qquad\quad\,\,  +\frac{1}{2} |\chi|^{3\sigma} \widetilde M_\chi^m \big(|\chi|^{2\sigma}+\gamma^2 \big(1+|\chi|^\sigma)^2\big).
\end{align}

\noindent {\bf Step 3.}  In this  step, we prove there is $c^{**}_{\chi,m,\alpha,\gamma}>0$ such that for any $c>c^{**}_{\chi,m,\alpha,\gamma}$ and any $\eta\in (\kappa,\frac{1}{1+|\chi|^\sigma})$,  $\lim_{t\to\infty} \int_{\RR}U^2(t,x)dx=0$ provided that $\int_{\RR} e^{2\eta x}(u_0(x)-U^*(x))^2dx<\infty$.

\smallskip

To this end, first,  let 
$$
D'_{\chi,m,\alpha,\gamma}:=D'_{\chi,m,\alpha,\gamma}(M_\chi)\quad {\rm and}\quad D''_{\chi,m,\alpha,\gamma}:=D''_{\chi,m,\alpha,\gamma}(M_\chi).
$$ 
Let 
$$
c_{3,\chi,m,\alpha,\gamma}:=|\chi|^{1-3\sigma} D'_{\chi,m,\alpha,\gamma}+ 2\sqrt{1+|\chi|^{1-3\sigma} D''_{\chi,m,\alpha,\gamma}},
$$
and
\begin{equation}
\label{c-star-star-eq}
c^{**}_{\chi,m,\alpha,\gamma}:=\max\{c_{1,\chi,\sigma},c_{2,\chi,m,\gamma},c_{3,\chi,m,\alpha,\gamma}\}.
\end{equation} 
Note that for any $c>c^{**}_{\chi,m,\alpha,\gamma}$, 
\begin{align}
\label{eta-eq}
&\kappa^-_{c,\chi,m,\alpha,\gamma}:=\frac{c-|\chi|^{1-3\sigma} D'_{\chi,m,\alpha,\gamma}-\sqrt {(c-|\chi|^{1-3\sigma}D'_{\chi,m,\alpha,\gamma})^2 -4 (1+|\chi|^{1-3\sigma}D''_{\chi,m,\alpha,\gamma})}}{2}\nonumber\\
&\le \kappa <\frac{1}{1+|\chi|^\sigma}\nonumber\\
&<\kappa^+_{c,\chi,m,\alpha,\gamma}:=\frac{c-|\chi|^{1-3\sigma} D'_{\chi,m,\alpha,\gamma}+\sqrt {(c-|\chi|^{1-3\sigma}D'_{\chi,m,\alpha,\gamma})^2 -4 (1+|\chi|^{1-3\sigma}D''_{\chi,m,\alpha,\gamma})}}{2}.
\end{align}
Then
$$
 \eta^2 -(c-|\chi|^{1-3\sigma}  D'_{\chi}) \eta +(1+|\chi|^{1-3\sigma}  D''_{\chi})
<0
$$
for any $\eta\in (\kappa,\frac{1}{1+|\chi|^\sigma})$. 

Next,  for any $\eta\in (\kappa,\frac{1}{1+|\chi|^\sigma})$,
choose $\widetilde M_\chi$ sufficiently close to $M_\chi$ such that
$$
\lambda:= \eta^2 -(c-|\chi|^{1-3\sigma} \widetilde D'_{\chi}) \eta +(1+|\chi|^{1-3\sigma} \widetilde D''_{\chi})
<0
$$
It then follows that,   if $\int_{\RR} e^{2\eta x} (u_0(x)-U^*(x))^2dx<\infty$, then
$$
\int_{\RR}U^2(t,x)dx\le e^{-\lambda t} \int_{\RR} U^2(0,x)dx\to 0
$$
as $t\to\infty$. Moreover, 
\begin{equation}
\label{new-aux-U-eq1}
\int_0^\infty \int_{\RR} U^2(t,x)dx<\infty,
\end{equation}
and
\begin{equation}
\label{new-aux-U-eq2}
\int_0^\infty \int_{\RR} U_x^2(t,x)dx dt<\infty.
\end{equation}

\noindent{\bf Step 4.} In this step, we prove $\lim_{t\to\infty}\|u(t,\cdot;u_0)-U^*(\cdot)\|_\infty=0$
provided that $\int_{\RR} e^{2\eta x}(u_0(x)-U^*(x))^2 dx<\infty$, 
$c>c_{\chi,m,\alpha, \gamma}^{**}$, and $\eta\in (\kappa,\frac{1}{1+|\chi|^\sigma})$.

\smallskip
First,  for any fixed $t_0>0$,  there hold $\int_{\mathbb{R}}U^2(t_0,x)dx<\infty$ and 
$\int_{\mathbb{R}} U_x^2(t_0,x)dx<\infty$. Then, by Morrey's inequality,
there is $\widetilde M>0$ such that
$$
|U(t_0,x)|=e^{\eta x}|u(t_0,x)-U^*(x)|\le \widetilde M\quad \forall\, x\in\mathbb{R},
$$
i.e.,
$$
U^*(x)-\widetilde M e^{-\eta x} \le u(t_0,x)\le U^*(x)+\widetilde M e^{-\eta x}\quad \forall\, x\in\mathbb{R}.
$$
This together with \eqref{decay-rate-eq1}  and the boundedness of $u(t_0,x)$  implies that
\begin{equation}
\label{u-t0-eq}
e^{-\kappa x}-De^{-\tilde\kappa x}\le u(t_0,x)\le \max\{M,M e^{-\kappa x}\}\quad \forall\, x\in\mathbb{R}
\end{equation}
for some  $\tilde \kappa \in (\kappa,  \min\{(1+\alpha)\kappa, m\kappa+1/2,1\})$, $M\ge 1$ and $D\gg 1$.
By Remark \ref{prelim-sub-rk2}, 
$$
u(t,x;u_0)\ge  e^{-\kappa x}-D e^{-\tilde \kappa x}\quad \forall\, x\in\mathbb{R}\quad \forall\, t\ge t_0,\,\, x\in\mathbb{R},
$$
and  
\begin{equation}
\label{lower-bound-eq1}
 \inf_{t\ge t_0,x\le -R_0} u(t,x;u_0)>0
\end{equation}
for some $R_0>0$.

Next, assume that 
$$
\limsup_{t\to\infty} \|u(t,\cdot;u_0)-U^*(\cdot)\|_\infty>0.
$$
Then there are $\epsilon_0>0$, $t_n\to\infty$ with $1\le t_n\le t_{n+1}-1$, and $x_n\in\mathbb{R}$ such that
$$
|u(t_n,x_n)-U^*(x_n)|\ge \epsilon_0.
$$
By the uniform continuity of $u(t,x)-U^*(x)$ in $t\ge 1$ and $x\in\mathbb{R}$, there is $\delta_0>0$ such that
$$
|u(t,x)-U^*(x)|\ge \frac{\epsilon_0}{2}\quad \forall\, t\in [t_n,t_n+\delta_0],\,\, x\in [x_n-\delta_0,x_n+\delta_0].
$$
Without loss of generality, we may assume that 
$\lim_{n\to\infty} x_n=\infty$, $\lim_{n\to\infty }x_n=x^*$ for some $x^*\in\mathbb{R}$, or
$\lim_{n\to\infty} x_n=-\infty$.

In the case that $\lim_{n\to\infty} x_n=\infty$ or $\lim_{n\to\infty} x_n=x^*$ for some $x^*\in\mathbb{R}$, there is $R\in\mathbb{R}$
such that $x_n\ge R$ for $n\gg 1$.  
Then
$$
\int_{\RR} e^{2\eta x} (u(t_n,x;u_0)-U^*(x))^2 dx\ge
\frac{\epsilon_0^2}{4} \int_{x_n}^{x_n+\delta_0} e^{2\eta x}dx \ge \frac{\epsilon_0^2}{4}\int_{R}^{R+\delta_0} e^{2\eta x}dx \not \to 0
\quad {\rm as}\quad n\to\infty,
$$
 which is a contradiction.

Hence,  $\lim_{x\to\infty}x_n=-\infty$. By \eqref{lower-bound-eq1}, there is   $d_0>0$ such that
\begin{equation}
\label{aux-decay-rate-eq2}
u(t,x;u_0)\ge d_0\quad \forall\, t\ge t_1,\,\, x\le -R_0.
\end{equation}
Let $(\tilde u_n(t,x),\tilde v_n(t,x))=(u(t+t_n,x+x_n;u_0),v(t+t_n,x+x_n;u_0))$. Without  loss of generality, we may assume that there is $(u^*(t,x),v^*(t,x))$ such that
$$
\lim_{n\to\infty} (u_n(t,x),v_n(t,x))=(u^*(t,x),v^*(t,x))
$$
locally uniformly in $(t,x)\in\mathbb{R}\times\mathbb{R}$ and $(u^*(t,x),v^*(t,x))$ is a classical solution of \eqref{main-moving-eq1}.
Note that
$$
u^*(t,x)\ge  d_0\quad \forall\, t\in\mathbb{R},\,\, x\in\mathbb{R},
$$
and 
$$
|u^*(0,0)-1|\ge \epsilon_0.
$$
But by Proposition \ref{constant-stability-prop}, $u^*(t,x)\equiv 1$, which is a contradiction.
Therefore, $\lim_{t\to\infty}\|u(t,\cdot;u_0)-u^*(\cdot)\|_\infty=0$. Theorem \ref{stability-thm} is thus proved.
\end{proof}

\subsection{Uniqueness of traveling wave solutions and proof of Theorem \ref{uniqueness-thm}}
\label{ss:uniqueness}

In this subsection, we prove Theorem  \ref{uniqueness-thm}.

\begin{proof}[Proof of Theorem \ref{uniqueness-thm}]
By  \eqref{decay-rate-eq2} and Remark  \ref{existence-rk}(1),  there is $\eta$ such that
$$
\kappa<\eta<\min\{ \frac{1}{1+|\chi|^{1/6}},1\}, 
$$
$$
\int_{\RR} e^{2\eta x} |U_1^*(x)-U_2^*(x)|^2dx<\infty,
$$
and
$$
\liminf_{x\to -\infty} U_i^*(x)>0,\quad i=1,2.
$$
Then by Theorem \ref{stability-thm}, we have
$$
U_1(x)\equiv U_2(x).
$$
This proves the theorem  \ref{uniqueness-thm}.
\end{proof}

\end{document}